\documentclass[12pt]{article}
\usepackage[margin=0.8in]{geometry}
\usepackage[latin1]{inputenc}
\usepackage{amsmath, amssymb, amsbsy, mathdots, mathabx}
\usepackage{amsthm}
\usepackage{enumitem,makeidx}
\usepackage[all]{xy}
\usepackage{tikz,graphics}
\usepackage{changepage, comment, url}
\usepackage[mathcal]{euscript}
\usepackage{thmtools, thm-restate}

\makeatletter
  \def\title@font{\Large\bfseries}
  \let\ltx@maketitle\@maketitle
  \def\@maketitle{\bgroup%
    \let\ltx@title\@title%
    \def\@title{\resizebox{\textwidth}{!}{%
      \mbox{\title@font\ltx@title}%
    }}%
    \ltx@maketitle%
  \egroup}
\makeatother

\title{A Satake homomorphism for the $\bmod \, p$ derived Hecke algebra}
\author{Niccol\`o Ronchetti}
\date{\today}

\theoremstyle{theorem}
\newtheorem{thm}{Theorem}
\newtheorem*{thm*}{Theorem}
\newtheorem{lem}[thm]{Lemma}
\newtheorem{prop}[thm]{Proposition}
\newtheorem{fact}[thm]{Fact}
\newtheorem{claim}[thm]{Claim}
\newtheorem{cor}[thm]{Corollary}
\newtheorem{conj}[thm]{Conjecture}
\theoremstyle{definition}
\newtheorem{defn}{Definition}
\newtheorem{exmp}{Example}

\theoremstyle{remark}
\newtheorem*{rem}{Remark}

\newcommand{\Hom}{\ensuremath{\mathrm{Hom}}}
\newcommand{\Ext}{\ensuremath{\mathrm{Ext}}}

\newcommand{\Ad}{\ensuremath{\mathrm{Ad}}}
\newcommand{\Gal}{\ensuremath{\mathrm{Gal}}}

\newcommand{\Ind}{\ensuremath{\mathrm{Ind}}}
\newcommand{\cInd}[2]{\ensuremath{\iota^{#1}_{#2}}} 

\newcommand{\supp}{\ensuremath{\mathrm{Supp} \,}}
\newcommand{\val}{\ensuremath{\mathrm{val}}} 
\newcommand{\Stab}{\ensuremath{\mathrm{Stab}}}
\newcommand\floor[1]{\lfloor#1\rfloor}
\newcommand{\Ob}{\ensuremath{\mathrm{Ob}}} 

\newcommand{\pr}{\ensuremath{\mathrm {pr}}}

\newcommand{\id}{\ensuremath{\mathrm{id}}}

\newcommand{\im}{\ensuremath{\mathrm{Im}}}

\newcommand{\lra}{\ensuremath{\longrightarrow}}

\newcommand{\Q}{\ensuremath{\mathbb Q}}
\newcommand{\Qp}{\ensuremath{\mathbb Q_p}}

\newcommand{\Fp}{\ensuremath{\mathbb F_p}}
\newcommand{\FFp}{\ensuremath{\overline { \mathbb F_p}}}

\newcommand{\R}{\ensuremath{\mathbb R}}
\newcommand{\Z}{\ensuremath{\mathbb Z}}
\newcommand{\Zp}{\ensuremath{\mathbb Z_p}}

\newcommand{\C}{\ensuremath{\mathbb C}}

\newcommand{\smat}[4]{\left(\begin{smallmatrix} #1 & #2\\ #3 & #4\end{smallmatrix}\right)}

\newcommand{\PGl}[1]{\ensuremath{\mathrm{PGL}_{#1} }}
\newcommand{\GL}[2]{\ensuremath{\mathrm{GL}_{#1} (#2) }}
\newcommand{\PGL}[2]{\ensuremath{\mathrm{PGL}_{#1} (#2) }}

\newcommand{\res}[2]{\ensuremath{\mathrm{res}^{#1}_{#2} }} 
\newcommand{\cores}[2]{\ensuremath{\mathrm{cores}^{#1}_{#2} }} 
\newcommand{\OO}{\ensuremath{\mathcal O}}

\makeindex

\begin{document}

\maketitle

\begin{abstract}
We explore the structure of the derived spherical Hecke algebra of a $p$-adic group, a graded associative algebra whose degree $0$ subalgebra is the classical spherical Hecke algebra.
Working with $\Z/p^a$ coefficients, we establish a Satake homomorphism relating this graded algebra to the corresponding graded algebra for the torus. We investigate the image of this homomorphism in degree $1$, as well as other properties, such as transitivity with respect to inclusion of Levi subgroups. \end{abstract}

\tableofcontents 

\section{Introduction}
%
Let $\mathrm G$ be a connected, reductive, split group over a $p$-adic field $F$. Let $K$ be a hyperspecial maximal compact subgroup and $I$ an Iwahori subgroup. 
The following result (which holds even for non-split $\mathrm G$) has proven extremely important in the study of complex representation of $p$-adic groups.
\begin{thm*}[Borel \cite{Bo}, Matsumoto \cite{Mat}] Let $G \supset I$ be as above. The category of complex smooth representations of $G$ generated by their $I$-invariant vectors is equivalent to the category of finite-dimensional modules over the Iwahori-Hecke algebra $\mathcal H_{\C} \left( G, I \right)$ of complex-valued, compactly supported, bi-$I$-invariant functions on $G$.
\end{thm*}

In the past few decades, there has been growing interest in modular representations of $p$-adic groups, for example due to their connections to $\bmod l$ and $\bmod p$ automorphic forms. The $\bmod l$-representation theory of $p$-adic groups bears some similarities to the complex theory, but also displays some interesting differences: it has been studied by Dat (\cite{dat1,dat2} and related papers) Vign\'eras (\cite{vigneras,vigneras3,vigneras4,vigneras5} and related work) and others.

On the other hand, the $\bmod p$-representation theory of $p$-adic groups shows radically different features from the complex theory and is a current and exciting area of research.
A crucial new aspect in the $\bmod p$-theory is that the above theorem by Borel and Matsumoto does not hold anymore (see for instance the discussion in sections 2 and 3 of \cite{vigneras2}). On the other hand, Schneider obtained the following noteworthy result by `deriving' the classical picture.
\begin{thm*}[Schneider, \cite{schneider}] Let $G$ be a connected, reductive $p$-adic group and $I'$ be the pro-$p$ Sylow of an Iwahori subgroup. Suppose that $I'$ has no $p$-torsion. The derived category of smooth $\FFp$-representations of $G$ is equivalent to the derived category of modules over a differential graded pro-$p$ Iwahori-Hecke algebra $\mathcal H^{\bullet}$ constructed from $G$ and $I$.
\end{thm*}
Further study of this algebra (and of its associated cohomology algebra) has been undertaken by Ollivier and Schneider in \cite{OS}.
Their results suggest that an important role in the $\FFp$-representation theory of $p$-adic groups should be played by `derived' versions of the classical Hecke algebras. \\

In this paper, we switch our attentions from Hecke algebras constructed using the Iwahori subgroup (or its pro-$p$ Sylow) to Hecke algebras constructed using the hyperspecial subgroup $K$. The classical spherical Hecke algebra is the algebra of compactly supported, bi-$K$-invariant functions on $G$ under convolution: this algebra detects \emph{unramified} representations of $G$\footnote{A smooth representation $V$ is unramified if it has some nonzero $K$-invariants: $V^K \neq 0$.}, the most accessible class of representations.

We build a graded algebra $\mathcal H(G,K)$, the \emph{derived spherical Hecke algebra}, whose degree $0$ subalgebra is the classical spherical Hecke algebra. This graded algebra was already introduced by Venkatesh in \cite{akshay}, but we consider $p$-torsion coefficients while loc. cit. takes torsion coefficients where $p$ is invertible - the $p$-torsion situation is considerably more complicated. One should think of $\mathcal H(G,K)$ as `$G$-equivariant cohomology classes on $G/K \times G/K$ (see definition \ref{primadefdha}).

We then investigate the structure of this derived Hecke algebra by constructing a \emph{derived Satake homomorphism} into the derived spherical Hecke algebra of a maximal torus, and we study the properties of this morphism. An important consequence of our results is that the subalgebra of the derived Hecke algebra generated in degree $1$ is graded commutative, and relatively well-understood as it embeds into the derived Hecke algebra of a maximal torus.

A good understanding of the structure of $\mathcal H(G,K)$ is very important because of the following `derived' fact\footnote{see section 2.5 of \cite{akshay} for a more detailed explanation.}: just like for each smooth representation $V$ the $K$-invariants $V^K$ are a module for the classical Hecke algebra, for each complex of smooth representations $M^{\bullet}$ the derived $K$-invariants $H^*(K,M^{\bullet})$ are a graded module for $\mathcal H(G,K)$.

From a local perspective, one can then ask how large the class of `derived unramified' representations - those having nonzero derived $K$-invariants and thus detected by $\mathcal H(G,K)$ - is. It would also be interesting to relate this algebra to Schneider's derived Hecke algebra: the classical versions of the spherical and the Iwahori-Hecke algebras are related by the Satake and the Bernstein homomorphisms - is there an analog of the latter in the derived picture?

When taking a more arithmetic viewpoint, one notices (see section 2.6 of \cite{akshay}) that the singular cohomology of an arithmetic manifold\footnote{Arithmetic manifolds can be thought of as generalizations of modular curves, and more generally Shimura varieties, to all reductive groups $\mathrm G$ over number fields. Their singular cohomology is then a good avatar for a space of automorphic forms on $G$, in the spirit of Eichler -Shimura.} can also be interpreted as a graded module for $\mathcal H(G,K)$, it is then very interesting to study the structure of this graded module.
This is the global perspective taken in \cite{akshay}, which allows Venkatesh to prove impressive and extremely interesting results on the cohomology of these arithmetic spaces. In this situation it is easy to see that one can get a lot of mileage by only using `degree $1$' derived Hecke operators, which explains our focus in theorems \ref{mainthm} and \ref{satakeimagerevisited} on the degree $1$ submodule of the derived Hecke algebra. In current work in progress, we investigate similar questions in a '$l=p$' setup, using only $p$-torsion derived Hecke algebras of $p$-adic groups.

\subsection{Statement of results}
For classical spherical Hecke algebras in characteristic $p$ Herzig (\cite{herzig}) has constructed a Satake homomorphism into the Hecke algebra of a maximal split torus, and has explicitly described the image of this homomorphism. Later on, Henniart and Vign\'eras (\cite{HV}) generalized his constructions to Hecke algebras over a general coefficient ring.

Our first result is an extension of these homomorphisms to the degree $1$ submodule of the derived Hecke algebra.
\begin{thm}\label{mainthm} Let $S = \Z / p^a$. Let $k$ be the residue field of $F$, and assume that $|k| \ge 5$.  Fix a maximal split $F$-torus $\mathrm T$ of $\mathrm G$, and let $\mathrm M$ be a standard Levi subgroup of $\mathrm G$.
There exists a graded-algebra homomorphism \[ \mathcal S^G_M: \mathcal H_G^{ \le 1} \lra \mathcal H_M^{\le 1} \]
whose degree zero component agrees with Herzig's and Henniart and Vigneras' Satake morphisms, in their versions with $S$-coefficients (see section \ref{secHerzig}).

Moreover, if $\mathrm M_1 \supset \mathrm M_2$ is an inclusion of standard Levi subgroups of $\mathrm G$, then we have \[ \mathcal S^G_{M_2} = \mathcal S^{M_1}_{M_2} \circ \mathcal S^G_{M_1}. \]
\end{thm}
The restriction to degree $1$ is of a technical nature: we discuss precisely in the remark within the proof of proposition \ref{modisom} why it arises.
We also make explicit an important case where we can get rid of this restriction.
\begin{thm} \label{satakealldegrees} With assumptions as above, suppose moreover that $p \ge h$ and\footnote{$h$ is the Coxeter number, an invariant depending only on the root datum of $\left( \mathrm G, \mathrm T \right)$.} that either \begin{itemize}
    \item $F = \Qp$ or
    \item $F / \Qp$ is a Galois extension of degree $d=ef$, where $|k| = p^f$, and $2d(h-1) < p^f-1$.
\end{itemize}
Then we can extend the homomorphism in theorem \ref{mainthm} to all degrees, and obtain a graded-algebra homomorphism \[ \mathcal S^G_T: \mathcal H_G \lra \mathcal H_T. \]
\end{thm}
Notice that the target of this homomorphism is the derived spherical Hecke algebra of the torus, not of a general Levi subgroup. Adapting the strategy of the proof of theorem \ref{satakealldegrees} for a general Levi subgroup $\mathrm M$ would require to computate  the cohomology of some subgroups of $\mathrm M(\OO)$ with coefficients in certain algebraic representations. This seems currently out of reach.

We also record a different generalization of theorem \ref{mainthm}, since the proof is immediate from our work and it may be of use for global applications.
\begin{cor} \label{casepinvertible} Let $S$ be a ring where $p$ is invertible, and let $G$ and $M$ be as in theorem \ref{mainthm}. Then we can extend the homomorphism in theorem \ref{mainthm} to all degrees, and obtain a graded-algebra homomorphism \[ \mathcal S^G_M: \mathcal H_G \lra \mathcal H_M. \]
Transitivity for inclusion among standard Levi subgroups also holds.
\end{cor}
Notice that Venkatesh in \cite{akshay} (section 3) defines already a Satake homomorphism in the case where $S = \Z / l^n \Z$ for a special family of primes $l$ depending on $p^f$, but his approach is quite different and it does not immediately generalize to the case of general $S$.

Our final result gives a first description of the image of the degree $1$ Satake homomorphism $\mathcal S^G_T$, generalizing the explicit description of Herzig's degree $0$ morphism from \cite{herzig}.
\begin{thm} \label{satakeimagerevisited} Let $S$ and $k$ be as in theorems \ref{mainthm} and \ref{satakealldegrees}. Let $X_*(\mathrm T)$ be the cocharacter group of the torus.
With reference to the isomorphism $\mathcal H_T \cong S [ X_*(\mathrm T)] \otimes_S H^*(\mathrm T(\OO), S)$ (see proposition \ref{torusdha} for details) the image of $\mathcal S^G_T: \mathcal H_G^{\le 1} \lra \mathcal H_T^{\le 1}$ is supported on the following thickening of the antidominant cone: \[ \left\{ \lambda \in X_*(T) \, | \, \langle \lambda, \alpha \rangle \cdot f \le a \, \forall \alpha \in \Delta \right\} \] where $\Delta$ is the basis of the root system $\Phi(\mathrm G,\mathrm T)$ corresponding to the Borel subgroup chosen in the definition of the Satake homomorphism $\mathcal S^G_T$. \\
More precisely, if $d = \langle \lambda, \alpha \rangle \cdot f$ for some simple root $\alpha$, then $\mathcal S^G_T(F) (\lambda)$ is divisible by $p^d$ for all $F \in \mathcal H_G^{\le 1}$.
\end{thm}
The description of a `lower bound' for the image $\mathcal S^G_T \left( \mathcal H^{\le 1}_G \right)$ - that is to say, a $\mathcal H^0_G$-submodule of $\mathcal S^G_T \left( \mathcal H^{\le 1}_G \right)$ - has been obtained in the $\PGl{2}$ case (see also example \ref{PGL2example} below), and will be object of future work for general $\mathrm G$. Roughly speaking, we expect this lower bound to be a `translate' of the upper bound obtained in the theorem by a generic and `very antidominant' cocharacter.

We describe explicitly our derived Satake homomorphism in a simple case.
\begin{exmp} \label{PGL2example} Let $\mathrm G = \PGl{2}$ over $\Qp$, where we suppose $p \ge 3$. Let $\mathrm T$ be the diagonal torus, and $\mathrm B$ be the Borel subgroup of upper unipotent matrices, so that the only positive root $\alpha$ has root space the unipotent upper triangular matrices.
Let $K = \PGL{2}{\Zp}$ be a fixed maximal compact: notice that by proposition 3.10 in \cite{conrad2}, we have $\PGL{2}{\Zp} = \GL{2}{\Zp} / \Zp^*$, and the matrix $\smat{a}{b}{c}{d} \in \PGL{2}{\Zp}$ denotes in fact its scaling class.
Let $S = \Z / p \Z$ be our ring of coefficients. We will write $\lambda$ for the function $\PGL{2}{\Zp} \lra \Z / p \Z$ \[ \lambda: \smat{a}{b}{c}{d} \lra \log(ad^{-1}) \bmod p, \]
where $\log ( \cdot ) \bmod p$ is the composition $\Zp^* \twoheadrightarrow 1 + p \Zp \stackrel{\log}{\lra}\Zp \twoheadrightarrow \Zp / p \Zp \cong \Z / p \Z$.

Let $T_n \in \mathcal H_G^0$ be the indicator function of $K\smat{p^{-n}}{0}{0}{1}K$ and $\tau_n \in \mathcal H_T^0$ be the indicator function of $\smat{p^{-n}}{0}{0}{1} \mathrm T(\OO)$. On these bases, the degree $0$ Satake homomorphism is \[ \mathcal S(T_n) = \left\{ \begin{array}{ccc} \tau_n & \textnormal{ if } & n =0,1 \\ \tau_n - \tau_{n-2} & \textnormal{ if } & n \ge 2 \end{array}. \right. \]
For each $n \ge 0$, let $K_n = \left\{ \smat{a}{b}{c}{d} \in \PGL{2}{\Zp} \, | \, c \in p^n \Zp \right\}$ be the common stabilizer of the identity coset $K$ and $\smat{p^{-n}}{0}{0}{1} K$. 
Define elements $f_n \in \mathcal H_G^1$ (for $n \ge 2$) and $c_n \in \mathcal H_T^1$ (for $n \in \Z$) so that 
 \[ f_n: (K,  \smat{\varpi^{-n}}{0}{0}{1}  K) \mapsto \mbox{ class of $\lambda$ in $H^1(K_n, \Z/p\Z)$}, \]
\[ c_n: (\mathrm T(\OO),  \smat{\varpi^{-n}}{0}{0}{1} \mathrm T(\OO)) \mapsto \mbox{ class of $\lambda$ in $H^1(\mathrm T(\OO),  \Z/p\Z)$}, \] and $f_n$, $c_n$ vanish off the $G$- and $T$-orbits of the left-hand elements.
It is readily verified that the restriction of $\lambda$ to $K_n$ (for $n \ge 2$) and a fortiori to $T(\OO)$ is indeed a homomorphism, and that the $f_n$ and $c_n$ give bases for the $S$-modules $\mathcal H_G^1$ and $\mathcal H_T^1$. \\
With these notations, the degree $1$ Satake homomorphism is given by \[ \mathcal S^G_T (f_n) = c_n - c_{n-2} \qquad \forall n \ge 2. \]
Given the multiplication rule on $\mathcal H_G$: \[ T_n \circ f_m =  \left\{ \begin{array}{ccc} f_{m+n} & \textnormal{ if } & n =0,1 \\ f_{m+n} - f_{m+n-2} & \textnormal{ if } & n \ge 2 \\ \end{array} \right. \] it is immediate to check that this is an extension of the degree $0$ Satake homomorphism to a map of algebras $\mathcal H_G^{\le 1} \lra \mathcal H_T^{\le 1}$.

We notice that in this situation, both Herzig's and Henniart and Vigneras' degree $0$ Satake homomorphism and our degree $1$ extension have imaged supported on the antidominant cone, just as theorem \ref{satakeimagerevisited} predicts. Indeed, since $\alpha \left( \smat{t}{0}{0}{s} \right) = ts^{-1}$, the antidominant cone consists of characters $t \mapsto \smat{t^{-n}}{0}{0}{1}$ having $n \ge 0$ so that they pair non-positively with $\alpha$.
%
%
\end{exmp}
\begin{rem} A simple generalization of our results holds for more general coefficient rings of the form $S = \OO_E / \varpi_E^m$, where $E / \Qp$ is a degree $n$ extension. In particular, the proof of theorem \ref{satakeimagerevisited} goes through in the same way to show that if $d = \langle \lambda, \alpha \rangle \cdot f >0$ for some simple root $\alpha$, then $\mathcal S^G_T F (\lambda)$ is divisible by $p^d$ for all $F \in \mathcal H_G^1$.
Let now $k = \floor{\frac{m}{n}}$, so that $kn \le m < (k+1)n$, then $\OO_E \cong \Zp^n$ as a free $\Zp$-module and for an appropriate choice of basis we have $\varpi_E^m \OO_E \cong \left( p^k \Zp \right)^{n-(m-kn)} \oplus \left( p^{k+1} \Zp \right)^{m - kn}$ so that \[ S = \OO_E / \varpi^m \OO_E \cong \left( \Z / p^k \right)^{n- (m-kn)} \oplus \left( \Z / p^{k+1} \right)^{m-kn}. \]
In particular, as soon as $d \ge k+1 = \floor{\frac{m}{n}} +1$ we obtain that $\mathcal S^G_T F (\lambda) = 0$ for all $F \in \mathcal H_G^1$.
\end{rem}

\subsection{Outline of the paper}
We now explain the structure of the paper.

In section \ref{secnotation} we set up some notation that we will use in the course of the paper.

Section \ref{secHerzig} recalls the results of Herzig in \cite{herzig} and of Henniart and Vigneras in \cite{HV} where they study the classical Hecke algebras with coefficients in, respectively, an algebraically closed field of characteristic $p$ and any coefficient ring. We extend Herzig's results to any coefficient ring of the form $S = \Z / p^a \Z$, the proofs go through almost verbatim and we get a Satake homomorphism in degree $0$. This has been proven already by Henniart and Vigneras (see section 7.13 of \cite{HV}), but we describe an explicit formula mirroring our bound from theorem \ref{satakeimagerevisited} in the degree $1$ case.

In section \ref{secgroupoids} we set up some machinery related to groupoids, which will greatly clarify and simplify some subsequent proofs. In particular, we are able to re-interpret the derived Hecke algebra of $G$ as compactly supported cohomology of the groupoid where $G$ acts diagonally on $G/K \times G/K$ - see propositions \ref{convgroupoid} and \ref{convassociative}.

Section \ref{secSatakegroupoids} is devoted to defining the Satake map via a push-pull diagram in groupoid cohomology. Using this setup, we also prove transitivity of the Satake map for inclusion among Levi subgroups: that is to say, if $\mathrm M$ is a standard Levi subgroup with $\mathrm G \supset \mathrm M \supset \mathrm T$, then $\mathcal S^G_T = \mathcal S^M_T \circ \mathcal S^G_M$.

We give in section \ref{secSatakeUPS} another description of the Satake homomorphism by using the Universal Principal Series as a bi-module for two derived Hecke algebras. We also complete the proofs of theorems \ref{mainthm} and \ref{satakealldegrees}.

Section \ref{secSatakeimage} is dedicated to proving theorem \ref{satakeimagerevisited} by studying the image of the Satake homomorphism $\mathcal S^G_T : \mathcal H_G^1 \lra \mathcal H_T^1$. To study the support of the image, we use the transitivity result of section \ref{secSatakegroupoids} to put us in the setting where $G$ has semisimple rank $1$, where we compute the image as explicitly as needed.

In appendix \ref{appendice} we collect the proofs of some results about groupoid cohomology, which we postponed until here to improve the flow of the reading.

Appendix \ref{dhadiscussion} describes the relationship between our definition of the derived Hecke algebra and a categorical definition, both in characteristic $p$ and different from $p$.

\subsection{Acknowledgments}
I am deeply indebted to my advisor Akshay Venkatesh, for introducing me to the concept of derived Hecke algebra and for his continuous and enthusiastic guidance. I also want to thank Daniel Bump, for reading a preliminary version of this manuscript and suggesting improvements and new directions. I am extremely grateful to Florian Herzig for his many suggestions and for explaining to me a few arguments that improved the results of this paper. Finally, I want to thank Karol Koziol, Rachel Ollivier, Peter Schneider, Claus Sorensen and Marie-France Vign\'eras for some very interesting conversations and comments on this paper.

\section{Notation} \label{secnotation}
Let $F$ be a $p$-adic field, with ring of integers $\OO$, residue field $k$ of cardinality $q=p^f$ and fix a uniformizer $\varpi \in \OO$. We assume throughout the paper that $|k| \ge 5$. The valuation $\val_F$ is normalized to be $1$ on the uniformizer $\varpi$.

Let $\mathrm G$ be a reductive, split group scheme\footnote{We follow \cite{conrad1} for group schemes over rings. In particular, $\mathrm G$ is fiberwise connected by definition of reductivity.} over $\OO$. Thus, $\mathrm G_F$ is a reductive, connected, split group over $F$, and we denote $G = \mathrm G_F(F)$ its group of $F$-points. We also fix $K = \mathrm G(\OO)$, a hyperspecial maximal compact. 

Thanks to the splitness condition we can fix a maximal split torus $\OO$-scheme $\mathrm T$, and we have $\mathrm T(\OO)$ being the maximal compact subgroup of $\mathrm T_F(F)$. We also fix a Borel $\OO$-subgroup scheme $\mathrm B$ of $\mathrm G$, with Levi decomposition $\mathrm B = \mathrm T \ltimes \mathrm U$ defined over $\OO$, so that $\mathrm U$ is a unipotent $\OO$-group scheme. We assume throuhgout the paper that $K$ and $T$ are in good position, so that $K  \cap  T = \mathrm T(\OO)$ is the maximal compact subgroup of $T$.

We will often suppress the base change notation and just denote the subgroups above by $B = \mathrm B_F(F)$ and similarly for the torus and the unipotent subgroups. Other short-hand notations that we will use are $H_F$ or $H$ for $\mathrm H_F(F)$, and $H_{\OO}$ for $\mathrm H(\OO)$ where $\mathrm H$ is one of $\mathrm G, \mathrm B, \mathrm T, \ldots$.

The character lattice $X^*(\mathrm T)$ is dual to the cocharacter lattice $X_*(\mathrm T)$, and the latter is identified with $\mathrm T(F) / \mathrm T(\OO)$ (once we fixed the uniformizer $\varpi$) via \[ X_*(\mathrm T) \lra \mathrm T(F) / \mathrm T(\OO) \qquad \mu \mapsto \mu(\varpi). \]
Upon choosing the Borel $\mathrm B$, we obtain a set of positive roots $\Phi^+(\mathrm G,\mathrm T) \subset \Phi(\mathrm G,\mathrm T)$ inside the root system of $(\mathrm G,\mathrm T)$, and a unique basis $\Delta \subset \Phi^+(\mathrm G, \mathrm T)$.
We define the antidominant coweights \[ X_*(\mathrm T)_- = \left\{ \mu \in X_*(\mathrm T) \, | \, \langle \mu, \alpha \rangle \le 0 \quad \forall \alpha \in \Phi^+(\mathrm G,\mathrm T) \right\}; \] geometrically these form a cone inside the lattice $X_*(\mathrm T)$. The identification of the cocharacter lattice with $\mathrm T(F) / \mathrm T(\OO)$ carries the antidominant cone into $T^- / \mathrm T(\OO)$, where \[ T^- =\left\{ t \in \mathrm T(F) \, | \, \val_F \left( \alpha(t) \right) \le 0 \quad \forall \alpha \in \Phi^+(\mathrm G,\mathrm T) \right\}. \]

Given two $p$-adic groups $H \supset L$ and a smooth representation $V$ of $L$, we denote by $\cInd{H}{L} V$ the compactly induced representation from $L$ to $H$. This consists of the smooth part of the space of functions $f: H \lra V$ such that $f(lh) = l.f(h)$ for each $l \in L$, $h \in H$, with $H$ acting by right translation.

We set up our notation for the derived Hecke algebras (see also definition \ref{primadefdha}): whenever $\mathrm H$ is a reductive $\OO$-group scheme, so that we have the $p$-adic group $\mathrm H_F(F)$ and a maximal compact $\mathrm H(\OO)$, we denote by $\mathcal H_H = \mathcal H \left( \mathrm H(F), \mathrm H(\OO) \right)$ the derived Hecke algebra of $\mathrm H(F)$. More generally, we denote by $\mathcal H_H^k$ (resp. $\mathcal H_H^{\le k}$) the $k$-th graded submodule of the DHA (resp. the submodule supported on degrees at most $k$ with the induced subalgebra structure).
%

For each standard parabolic $\mathrm P = \mathrm M \ltimes \mathrm V$ - where $\mathrm M$ is the standard Levi factor containing the maximal torus $\mathrm T$ - we consider the subgroup $P^{\circ} = \mathrm M(\OO) \ltimes \mathrm V(F)$. We denote by $\mathcal H_P$ the algebra of $\mathrm P(F)$-invariant cohomology classes on $\mathrm P(F) / P^{\circ} \times \mathrm P(F) / P^{\circ}$ supported on finitely many orbits - with convolution as multiplication.
\begin{rem}
Notice that there is no possible ambiguity with the previous notation for the derived Hecke algebra, since we only defined DHA for a reductive group, hence whenever the subscript $P$ is a parabolic we mean indeed the algebra defined in the previous paragraph.
\end{rem}
We will often consider stabilizers and isotropy groups. If $G$ is a groupoid and $x \in \Ob(G)$, we use $G_x$ as a shorthand for $\Stab_G(x) = \Hom_G (x,x)$. In case we consider common stabilizers we use multiple subscripts, for instance if $x, y \in G/K$ then $G_{x,y} = \Stab_G(xK, yK) = \Ad(x) K \cap \Ad(y) K$. In case one of the cosets is $K$, we write $K_y = G_{1,y} = K \cap \Ad(y) K$.

We will often define sums over orbits, in the following sense: let $G$ be a group acting on a set $X$, the notation \[ \sum_{G \backslash \backslash X} \] means that we sum over a choice of representatives for the $G$-orbits in $X$. Each time we do that, we will check that the sum does not depend on the chosen representatives, and it has finitely many nonzero summands.
\section{Satake homomorphism in characteristic $p$} \label{secHerzig}
In this section we recall Herzig's construction of a Satake homomorphism with $\overline k = \FFp$-coefficients from \cite{herzig} and \cite{herzig2} and generalize it (in the case of the trivial representation) to the case of $p$-torsion coefficients $\Z / p^a \Z$. This generalization (see corollary \ref{degzerogeneral} below) is quite simple, and its proof basically consists of rephrasing Herzig's arguments in the special case of the trivial representation and keeping track of the powers of $p$ appearing as coefficients.

We remark that Henniart and Vigneras in \cite{HV} work out a wide generalization of Herzig's results as well, allowing for instance any ring of coefficients (such as $\Z / p^a \Z$), see sections 7.13-7.15 of loc. cit.. Their work includes the results on this section, but we prefer to explain in full detail corollary \ref{degzerogeneral} as the upper bound (on the image of the degree $0$ Satake homomorphism) that we obtain is extremely similar to the one obtained for our degree $1$ Satake homomorphism in theorem \ref{satakeimagerevisited}.

We build upon the following result:
%
%
\begin{restatable}{thm}{herzigsatake}[Herzig, theorem 1.2 of \cite{herzig}] Let $H_{\overline k}(G,K)$ be the Hecke algebra of $\overline k$-valued, bi-$K$-invariant functions on $G$, and similarly by $H_{\overline k}(\mathrm T(F), \mathrm T(\OO))$ the Hecke algebra of $\overline k$-valued, $\mathrm T(\OO)$-invariant functions on $\mathrm T(F)$. The Satake map \[ \mathcal S: H_{\overline k}(G,K) \lra H_{\overline k}(\mathrm T(F), \mathrm T(\OO)) \] defined by \[ f \mapsto \left( t \mapsto \sum_{u \in \mathrm U(F) / \mathrm U(\OO)} f(tu) \right) \] is an injective $\overline k$-algebra homomorphisms with image the subalgebra $H^-(\mathrm T(F), \mathrm T(\OO))$ of functions supported on antidominant coweights.
\end{restatable}
\begin{rem}In fact, Herzig's statement is much more general (as in, he considers Hecke algebras not just for the trivial representations but for more general representations), but for our purpose the above result is sufficient.
\end{rem}

Following Herzig, we remark that the formula for $\mathcal S$ is extremely similar to the classical one, for Hecke algebras with $\C$-coefficients. The main difference is that the classical formula has a twist by $\delta_B(t)^{1/2}$ (where $\delta_B$ is the modulus character of $B$), whose purpose is mainly to force the Satake homomorphism to land in the Weyl-invariant subalgebra of $H(\mathrm T(F), \mathrm T(\OO))$.
As Gross remarks in \cite{gross}, one can get rid of this twist and obtain a homomorphism defined over $\Z$, at the cost of losing the remarkable property of Weyl-invariance of the image (or having to twist the natural Weyl-invariant action on the image). This formula is then compatible with Herzig's Satake map $\mathcal S$ under reduction $\bmod p$ for the coefficients.

Herzig's result has an immediate extension to our setup with general finite $p$-torsion coefficients. More precisely, we have the following corollary:

\begin{restatable}{cor}{satakedegzerogeneral} \label{degzerogeneral} Let $S = \Z / p^a \Z$ be the ring of coefficients for the Hecke algebras $H_S(G,K)$ and $H_S(\mathrm T(F), \mathrm T(\OO))$. Consider the Satake map \[ \mathcal S: H_S(G,K) \lra H_S(\mathrm T(F), \mathrm T(\OO)) \] defined by \[ F \mapsto \left( t \mapsto \sum_{u \in \mathrm U(F) / \mathrm U(\OO)} F(tu) \right). \]
This is an injective $S$-algebra homomorphism with image supported on the thickening of the antidominant cone: \[ \left\{ \lambda \in X_*(\mathrm T) \, | \, \langle \lambda, \alpha \rangle f < a \, \forall \alpha \in \Delta \right\}. \]
In fact, we have that if $0 < h = \langle \lambda, \alpha \rangle f < a$ for some simple root $\alpha$, then for each $F \in H_S(G,K)$ we have \[ (\mathcal S F)(\lambda(\varpi)) \in p^h \Z / p^a \Z \subset \Z / p^a \Z. \]
\end{restatable}
\begin{proof}
This is a simple adaptation of the proof in \cite{herzig}: one starts with the following adaptation of lemma 2.7(iii) of loc.cit.:
\begin{claim}
Let $\lambda \in X_*(\mathrm T)$ and $\alpha \in \Phi$. Let $t = \lambda(\varpi)$. Suppose that $A$ is an abelian group of exponent $p^a$, and that $\psi: \mathrm U_{\alpha}(F) / t \mathrm U_{\alpha} (\OO) t^{-1}$ is a function with finite support such that $\psi$ is left-invariant by $\mathrm U_{\alpha}(\OO)$. Then \[ \sum_{u_{\alpha} \in \mathrm U_{\alpha}(F) / t \mathrm U_{\alpha}(\OO) t^{-1}} \psi(u_{\alpha}) \equiv 0 \bmod p^{f \langle \lambda, \alpha \rangle}. \]
In particular, if $\langle \lambda, \alpha \rangle \cdot f \ge a$, then \[ \sum_{u_{\alpha} \in \mathrm U_{\alpha}(F) / t \mathrm U_{\alpha}(\OO) t^{-1}} \psi(u_{\alpha}) = 0. \]
\end{claim}
The proof of this claim is immediate since we are in the split case, so the argument given by Herzig adapts as it is: $t \mathrm U_{\alpha}(\OO) t^{-1}$ is a subgroup of $\mathrm U_{\alpha}(\OO)$ of index $[ \OO : \varpi^{\langle \lambda, \alpha \rangle} \OO ] = p^{f \langle \lambda, \alpha \rangle}$. \\
Therefore if $u_{\alpha} \in \mathrm U_{\alpha}(F) / t \mathrm U_{\alpha} (\OO) t^{-1}$ is in the support of $\psi$, we have $\psi(g_i u_{\alpha}) = \psi(u_{\alpha})$ for each of the $p^{f \langle \lambda, \alpha \rangle}$ coset representatives $g_i$ of $\mathrm U_{\alpha}(\OO) / t \mathrm U_{\alpha}(\OO) t^{-1}$, and hence the contribution of those summands is a multiple of $p^{f \langle \lambda, \alpha \rangle }$.

Following through with Herzig's proof of his theorem 1.2, step 1,2 and 3 go through word by word, while step 4 - whose proof relies entirely on lemma 2.7, in particular lemma 2.7(iii) - is replaced by
\begin{claim}[Step 4'] Suppose $0 < h = \langle \mu, \alpha \rangle f$ for some simple root $\alpha$. Then $(\mathcal SF)(\mu(\varpi)) \equiv 0 \mod p^h$. \\
In particular, if $f \langle \lambda, \alpha \rangle \ge a$, then $(\mathcal S F)(\mu(\varpi)) = 0$ in $\Z / p^a \Z$.
\end{claim}
The proof of this claim follows the same argument used by Herzig, with the new input from our modified version of lemma 2.7(iii).
\end{proof}

\section{Interlude with groupoids} \label{secgroupoids}
In this section we introduce some notions related to groupoids that allow for a more formal and conceptual understanding of the derived Hecke algebra and the Satake map that we will define later. These notions include groupoid cohomology, pullback and pushforward maps and their properties, and we give an interpretation of the derived Hecke algebra in this setup. Some proofs of formal results, e.g. push-pull squares in groupoid cohomology, are postponed to appendix \ref{appendice}, to not disrupt the flow of the paper.

\begin{defn}[Topological groupoid] A topological groupoid $X$ is a groupoid together with the data of a topology on the morphism sets $\Hom_X(x,y)$ for all $x,y \in \Ob(X)$. We require that the composition maps $\Hom_X(x,y) \times \Hom_X(y,z) \lra \Hom_X(x,z)$ are continuous for all $x,y,z \in \Ob(X)$.

A morphism of topological groupoids $f: X \lra Y$ is a morphism of groupoids such that the induced maps on morphism sets $\Hom_X(x_1, x_2) \lra \Hom_Y(f(x_1),f(x_2))$ are continuous.
\end{defn}
\begin{exmp} \label{topgroupoidexmp}
Let $G$ be a topological group acting on a set $X$. Consider the groupoid they generate: $X$ is the set of objects, and $\Hom(x,y) = \{ g \in G \, | \, gx=y \}$. Then each morphism set has a natural topology as a subspace of $G$, and the composition map are clearly continuous, so this is a topological groupoid. We often denote this groupoid by $\left( G \curvearrowright X \right)$.
\end{exmp}
Every groupoid in this paper will be of the type discussed in this example, thus from now on we will often simply say `groupoid', dropping the adjective `topological'.
\begin{prop} Given a $G$-set $X$ and a $H$-set $Y$, the pair of a group homomorphism $\rho: G \lra H$ and a map of sets $f: X \lra Y$ which is $\rho$-equivariant induces a morphism of groupoids. \\
Explicitly, the morphism of groupoids $(f, \rho) : \left( G \curvearrowright X \right) \lra \left( H \curvearrowright Y \right) $ is $f$ on the objects and $\rho$ on the morphisms. \\
Moreover, if $G$ and $H$ are topological groups and $\rho$ is continuous, then $(f, \rho)$ is a morphism of topological groupoids.
\end{prop}
\begin{proof} The first part of the proposition is clear.
The only thing to check is that the map induced by $\rho$ on each morphism set is continuous. It suffices to check this for $\rho_{x,x}: \Hom(x,x) \lra \Hom(y,y)$ where $f(x)=y$. But then $\rho_{x,x}$ is the composition of $\rho|_{G_x}: G_x \lra H_y \cap \rho(G)$ and the inclusion of subgroups $H_y \cap \rho(G) \hookrightarrow H_y$. The former is continuous because $\rho$ is, and the latter is continuous since $H_y \cap \rho(G)$ has the subspace topology of $H$, and so does $H_y$. 
\end{proof}
\begin{defn}[Pullback and homotopy pullback] \label{defpullback} Let \begin{displaymath} \xymatrix{ & E \ar[d]^p \\ A \ar[r]_f & C } \end{displaymath} be a diagram of groupoids. The pullback is defined to be the subgroupoid $D$ of $A \times E$ whose objects are $\{ (a,e) \in A \times E \, | \, f(a) = p(e) \}$. The morphisms are \[ \Hom_D ( (a_1, e_1) , (a_2, e_2) ) = \] \[ = \left\{ ( \alpha, \varepsilon) \in \Hom_A(a_1, a_2) \times \Hom_E (e_1, e_2) \, | \, f(\alpha) = p(\varepsilon) \in \Hom_C \left( f(a_1) = p(e_1), f(a_2) = p(e_2) \right) \right\} \] This has projection maps onto $A$ and $E$ and satisfies the evident universal property.

The homotopy pullback is defined to be the groupoid $Z$ whose objects are triples $(a, \gamma, e) \in A \times \Hom_C \left( f(a), p(e) \right) \times E$ and where the morphisms are \[ \Hom_Z ( (a_1, \gamma_1, e_1) , (a_2, \gamma_2, e_2) ) = \] \[ = \left\{ ( \alpha, \varepsilon) \in \Hom_A(a_1, a_2) \times \Hom_E (e_1, e_2) \, | \, \gamma_2 \circ f(\alpha) = p(\varepsilon) \circ \gamma_1 \in \Hom_C \left( f(a_1) , p(e_2) \right) \right\}. \]
Again, this has obvious projection maps onto $A$ and $E$ and satisfies a `homotopy' universal property.

If $E, A$ and $C$ are topological groupoids and $f, p$ are morphisms of topological groupoids, then the pullback and the homotopy pullback are also topological groupoids, where the topology on each morphism sets is the subspace topology of $\Hom_A(a_1, a_2) \times \Hom_E(e_1,e_2)$. Moreover, the projection maps to $A$ and $E$ are continuous.
\end{defn}
\begin{rem}
These definitions above are clearly inspired by analogies with homotopy-theoretic constructions - like the following one.
\end{rem}
\begin{defn}[Covering morphism] Let $q: X' \lra X$ be a morphism of groupoids. We say that $q$ is a covering morphism if for each object $x' \in X'$ and each morphism $m: q(x') \lra x$ in $X$, there exists a \emph{unique} morphism $\widetilde m: x' \lra \widetilde x$ in $X'$ such that $q(\widetilde m) = m$.
\end{defn}
Notice that by example 3.7(i) and remark 3.9(ii) of \cite{bhk}, we obtain that if in a diagram \begin{displaymath} \xymatrix{ & E \ar[d]^p \\ A \ar[r]_f & C } \end{displaymath} one of the map is a covering morphism of groupoids, then the homotopy pullback and the pullback are homotopy equivalent. Since any homotopy equivalence (defined before proposition 3.5 in \cite{bhk}) is the identity on morphism sets, it is immediate that in the setup of topological groupoids the homotopy equivalence preserves the topological structure as well.
\begin{defn}[Connected components] Given a groupoid $X$, we denote by $\pi_0 X$ the set of its connected components - that is to say, the equivalence classes for the relation $x \sim y \iff \Hom_X(x,y) \neq \emptyset$ for any $x, y \in \Ob(X)$.
\end{defn}
\begin{defn}[Homotopy equivalence] \label{homotopyequiv} Let $f: X \lra Y$ be a morphism of (topological) groupoids. This is called a homotopy equivalence if it induces a bijection $\pi_0 f: \pi_0 X \lra \pi_0 Y$ between connected components and for each $a \in \Ob(X)$ we have that $f_a : \Stab_X(a) \lra \Stab_Y(f(a))$ is an isomorphism of (topological) groups.
\end{defn}
\subsection{Cohomology of topological groupoids}
We give an ad-hoc definition for the (compactly supported) cohomology of a topological groupoid, which will correspond to our derived Hecke algebra as a module, and then define pullback and pushforward maps on groupoid cohomology.

We will consider cohomology groups for a general locally profinite topological group $L$ with trivial coefficient in the ring $S$: $H^*(L,S)$. By this, we mean the cohomology in the category of discrete $S$-modules with a continuous action of $L$, i.e. the derived functors of $\Hom_{S[L]}(S, \cdot)$ computed at $S$. In all our applications, $L$ will be a $p$-adic analytic group and $S$ a finite $p$-torsion ring, so section 2 of \cite{emerton2} provide details for the definition of these cohomology groups. See also \cite{flach} for a more general perspective on topological groups (and in particular section 9.2 for totally disconnected groups) and \cite{sym} for the special case of profinite groups (e.g. compact $p$-adic analytic groups).
\begin{defn}[Groupoid cohomology] \label{groupoidcoh}
Let $X$ be a topological groupoid and $S$ be a coefficient ring. We define the cohomology of $X$ to be the $S$-module $\mathbb H^*(X)$ of maps \[ F:  \Ob(X) \lra \bigoplus_{x \in \Ob(X)} H^* \left( \Stab_X(x) , S \right) \] such that
\begin{enumerate}
\item $F(x) \in H^* \left( \Stab_X(x), S \right)$ for each object $x$ of $X$.
\item Each $\phi \in \Hom_X(x,y)$ induces an isomorphism of groups $\Stab_X(x) \lra \Stab_X(y)$ given by $g \mapsto \phi \circ g \circ \phi^{-1}$, and hence a map in cohomology $H^*\left( \Stab_X(y) \right) \stackrel{\phi^*}{\lra} H^* \left( \Stab_X(x) \right)$. We require that $F(x) = \phi^* (F(y))$.
\end{enumerate}
When $x=y$ this latter condition means that $F(x) \in H^*(\Stab_X(x))$ is invariant under the action of any element $\phi \in \Stab_X(x)$ - a well-known fact from group cohomology. Thus we can think of this second condition as a generalization of the above invariance and we will informally call it `the $X$-invariance condition'.

We also define the finitely supported cohomology $\mathbb H^*_c(X)$ to be the subspace of $\mathbb H^*(X)$ satisfying the additional condition
\begin{enumerate}
\item [3.] $F$ is supported on finitely many connected components of $X$.
\end{enumerate}
\end{defn}
\begin{rem}
We notice that the second condition in the definition (the `invariance' condition) means that on each connected component of $X$ either $F$ is zero everywhere or nonzero everywhere - thus making sense of condition 3.
\end{rem}
Finally, we have a natural `cup product' operation on $\mathbb H^*(X)$ defined as pointwise cup product: \begin{equation} \label{cupproduct} (F_1 \circ F_2)(x) = F_1(x) \cup F_2(x) \in H^* \left( \Stab_X(x) \right) \qquad \forall F_1 , F_2 \in \mathbb H(X). \end{equation}
It is immediate that this operation is associative (since cup product in group cohomology is) and it has a unit element, namely the map \[ \underline 1: \Ob(X) \lra \bigoplus_{x \in \Ob(X)} H^* \left( \Stab_X(x) , S \right) \textnormal{ defined as } \underline 1(x) = 1 \in H^0(\Stab_X(x), S) \] for all objects $x \in \Ob(X)$.
Since obviously $\supp (F_1 \circ F_2) \subset \supp F_1 \cap \supp F_2$, we obtain that the compactly supported cohomology $\mathbb H^*_c(X)$ is an ideal of $\mathbb H^*(X)$ under this operation, and in particular cup product preserves $\mathbb H^*_c(X)$. \\

Let's set up some notation. We denote $[G] = \mathrm G(F) / \mathrm G(\OO) = G/K$, and for any standard parabolic $\mathrm P = \mathrm M \ltimes \mathrm V$ with standard Levi $\mathrm M$, we similarly denote $[M] = \mathrm M(F) / \mathrm M(\OO) = \mathrm P(F) / P^{\circ}$. \\
We have the left-multiplication action of $\mathrm G(F)$ on $[G]^2$, and we interpret this action as giving us a groupoid $\underline G = \left( G_F \curvearrowright [G]^2 \right)$. Similarly we have the groupoid $\underline M$. \index{$ \underline G$} \index{$ \underline M$} 
%
Explicitly, $\underline G$ has underlying set $[G]^2$ and \[ \Hom \left( (g_1, g_2) , (g_1', g_2') \right) = \{ g \in \mathrm G(F) \, | \, gg_1 = g_1' \textnormal{ and } gg_2 = g_2' \textnormal{ in } [G] \}. \]
\begin{defn}[derived Hecke algebra] \label{primadefdha}
Let $\underline G = \left( G_F \curvearrowright [G]^2 \right)$ be our favorite groupoid. We define the derived Hecke algebra to be the compactly supported groupoid cohomology $\mathbb H^*_c(\underline G)$. The multiplication operation is described in proposition \ref{convgroupoid}.
We will also denote the derived Hecke algebra by $\mathcal H_G$.
\end{defn}
To describe what the convolution operation looks like on $\mathbb H^*_c(\underline G)$, we need the notions of pullback and pushforward in cohomology.
\begin{defn}[Pullback in cohomology]
Let $i:X \lra Y$ be a continuous morphism of topological groupoids, and let $F \in \mathbb H^*(X)$.
Since $i$ is a natural transformation, for each object $x \in \Ob(X)$ we have an induced continuous map \[ \Stab_X(x) = \Hom_X(x,x) \stackrel{i_x}{\lra} \Hom_Y(i(x), i(x)) = \Stab_Y(i(x)) \] and hence a map in cohomology in the other direction: $i^x: H^* (\Stab_Y(i(x)),S) \lra H^* ( \Stab_X(x), S)$. 
We define \[ (i^*F)(x) : = i^x F(i(x)) \]
\end{defn}
\begin{rem} In general, pullback does not preserve the `compact support' condition 3 of definition \ref{groupoidcoh}. But it does if we require the additional assumption that the map induced by $i$ on connected components $\pi_0 i: \pi_0 X \lra \pi_0 Y$ has finite fibers.

Notice in particular that an homotopy equivalence $i: X \lra Y$ as in definition \ref{homotopyequiv} induces via pullback an isomorphism of groupoid cohomology, preserving the compactly supported cohomology.
\end{rem}
In appendix \ref{appendice} we check that the definition is well-posed (proposition \ref{pullbackwellposed}). \\

We now define pushforward maps in cohomology for a special class of covering morphisms $i: X \lra Y$; notice that for each covering morphism $i: X \lra Y$ we have a natural injection between isotropy groups $X_x \hookrightarrow Y_{i(x)}$ for each $x \in \Ob(X)$.
\begin{defn}[Finite covering morphism] Let $i: X  \lra Y$ be a continuous morphism of topological groupoids. We call it a \emph{finite covering morphism} if it is a covering morphism such that for each $x \in \Ob(X)$ the inclusion $X_x \hookrightarrow Y_{i(x)}$ is finite index and open.
\end{defn}
\begin{rem}
By a theorem of Nikolov and Segal (see \cite{NS}) in case we have a finite index inclusion of profinite groups $G \hookrightarrow H$ with $H$ topologically finitely generated, then $G$ is automatically open. Moreover, a compact $p$-adic analytic group is necessarily finitely generated (this follows from theorem 8.1 in \cite{DDMS}), so the `open' condition is redundant in many cases of interest for us, for instance whenever $H \subset K = \mathrm G (\OO)$.
\end{rem}
\begin{defn}[Pushforward map in cohomology] \label{pushforwarddef} Let $i: X \lra Y$ be a finite covering morphism. We define the pushforward as follows: let $F \in \mathbb H^*_c(X)$, and $y \in \Ob(Y)$, we set \[ (i_*F)(y) = \sum_{Y_y \backslash \backslash x \in i^{-1}(y)} \cores{\Stab_Y(y)}{\Stab_X(x)} F(x). \]
Then $i_*F \in \mathbb H^*_c(Y)$.

The action of the stabilizer $Y_y = \Stab_Y(y) = \Hom_Y(y,y)$ on the fiber $i^{-1}(y)$ is given as follows: let $h \in Y_y$ and $x \in i^{-1}(y)$, then the definition of covering morphism says that there exists a \emph{unique} morphism $\widetilde h$ in $X$ lifting $h: i(x)=y \mapsto y$. So $\widetilde h: x \mapsto x'$, and since it lifts $h$ we must have $i(x') = i(\widetilde h (x)) = h(i(x)) = h(y)=y$, so that $x' \in i^{-1}(y)$. The action is then defined as $h.x = x'$.
\end{defn}
In the appendix, we check that this definition is well-posed (proposition \ref{pushforwardwellposed}).
\begin{rem}
In fact, the definition of pushforward via the same formula applies to a subspace of $\mathbb H^*(X)$ slightly larger than $\mathbb H^*_c(X)$.

For the formula $(i_*F)(y) = \sum_{Y_y \backslash \backslash x \in i^{-1}(y)} \cores{\Stab_Y(y)}{\Stab_X(x)} F(x)$ to be well-defined, we will only use that $F$ is compactly supported when checking that the summation has only finitely many summands - so in fact it is enough to require that for each $y \in \Ob(Y)$, $F$ is supported on finitely many connected components of $X$ above $y$, a condition which we can sum up as \emph{$i$-fiberwise compactly supported}.

If $F \in \mathbb H^*(X)$ is $i$-fiberwise compactly supported, $i_* F \in \mathbb H^*(Y)$ may be not compactly supported. Notice however that the submodule of $i$-fiberwise compactly supported cohomology classes is an ideal of $\mathbb H^*(X)$ under pointwise cup product.
\end{rem}
We describe now the convolution operation on the derived Hecke algebra.
\begin{prop}[Convolution on the derived Hecke algebra] \label{convgroupoid}
Consider the groupoid $\underline G = \left( G_F \curvearrowright [G]^2 \right)$ whose compactly supported cohomology $\mathbb H^*_c ( \underline G) $ is the derived Hecke algebra $\mathcal H_G$. Consider the following diagram: \begin{displaymath} \xymatrix{ \underline G & & \underline G \\ & \left( G_F \curvearrowright [G]^3 \right) \ar[lu]_{i_{1,2}} \ar[ru]^{i_{2,3}} \ar[d]_{i_{1,3}} \\ & \underline G } \end{displaymath}
where the map $i_{1,2}: \left( G_F \curvearrowright [G]^3 \right) \lra \underline G$ is projection on the first and second factors on object, and the obvious inclusion of morphisms - similarly for $i_{1,3}$ and $i_{2,3}$.

Letting $F_1, F_2 \in \mathbb H^*_c(\underline G)$, we describe their convolution in the following way: we pullback them to $\left( G_F \curvearrowright [G]^3 \right)$ via the maps $i_{1,2}$ and $i_{2,3}$, we cup them as in equation \ref{cupproduct}, and then push them forward to $\underline G$ via $i_{1,3}$.

This convolution operation puts an algebra stucture on the derived Hecke algebra $\mathcal H_G$.
\end{prop}
\begin{proof}
This is a particular instance of our general setup described in the appendix \ref{appendice}, fact \ref{generaldha}. The well-definedness of this operation follows from the fact that $K$ is open and compact in $G$, as explained in the remark following fact \ref{generaldha}. We will record in proposition \ref{convassociative} that this operation is indeed associative, but we can already notice that this too follows from fact \ref{generaldha}.

We record for the sake of future computation an explicit formula for convolution in this algebra: fix $x,y \in [G]$, we have \[ (i_{1,3})_* \left( i_{1,2}^*F_1 \cup i_{2,3}^* F_2 \right) (x,y) = \sum_{G_{x,y} \backslash \backslash (x,z,y) \in (i_{1,3})^{-1}(x,y)} \cores{G_{x,y}}{G_{x,y,z}} \left( i_{1,2}^*F_1 \cup i_{2,3}^* F_2 \right) (x,z,y) = \] \[ = \sum_{G_{x,y} \backslash \backslash z \in [G]} \cores{G_{x,y}}{G_{x,y,z}} \left( i_{1,2}^*F_1 (x,z,y) \cup i_{2,3}^* F_2 (x,z,y) \right) = \] \begin{equation} \label{convexplicit} = \sum_{G_{x,y} \backslash \backslash z \in [G]} \cores{G_{x,y}}{G_{x,y,z}} \left( \res{G_{x,z}}{G_{x,z,y}} F_1 (x,z) \cup \res{G_{z,y}}{G_{x,z,y}} F_2 (z,y) \right). \end{equation}
Notice that this is exactly the convolution formula for the derived Hecke algebra in its interpretation as equivariant cohomology classes as described, for instance, in \cite{akshay}.
\end{proof}
We make use of the newly computed formula for convolution to show that the derived Hecke algebra of the torus has the simple description mentioned in the introduction:
\begin{prop} \label{torusdha} The map \[ \mathcal H_T \lra S[X_*(\mathrm T)] \otimes_S H^*(\mathrm T(\OO), S) \qquad F \mapsto \widetilde F \] where \[ \tilde F( \mu) = F(\mathrm T(\OO), \mu(\varpi) \mathrm T(\OO)) \in H^*(\mathrm T(\OO), S) \textnormal{ for all }\mu \in X_*(\mathrm T) \] is an isomorphism of graded algebras.
The target  $S[X_*(\mathrm T)] \otimes_S H^*(\mathrm T(\OO), S)$ is the tensor product of $S$-algebras and the grading is on the cohomology factor of this tensor product.
\end{prop}
\begin{proof}
Since $F$ is compactly supported, $\widetilde F$ is well-defined as an element of the tensor product. Moreover, the map $F \mapsto \widetilde F$ is obviously linear.
We notice that $X_*(\mathrm T)$ and $\mathrm T(F) / \mathrm T(\OO)$ are in bijection as mentioned in the \ref{secnotation} section, via $\mu \mapsto \mu(\varpi)$, and thus the map $F \mapsto \widetilde F$ gives a bijection between the `canonical bases' of $\mathcal H_T \cong \bigoplus_{t \in \mathrm T(F) / \mathrm T(\OO)} H^*(\Stab_{T(\OO)}(t T(\OO)), S) = \bigoplus_{t \in \mathrm T(F) / \mathrm T(\OO)} H^*(\mathrm T(\OO), S)$ and of $S[ X^*(\mathrm T)] \otimes_S H^*(\mathrm T(\OO),S)$.

It remains to check that the map respects the algebra structure, i.e. that given $F_1, F_2 \in \mathcal H_T$ we have $\widetilde F_1 \cdot \widetilde F_2 = \widetilde{F_1 \circ F_2}$ in $S[ X^*(\mathrm T)] \otimes_S H^*(\mathrm T(\OO),S)$.
Notice that when we specialize equation \ref{convexplicit} to the case $G = T$, we obtain that for any $x,y,z \in [T]$ the stabilizers become $G_{x,z} = G_{x,y} = G_{y,z} = G_{x,y,z} = \mathrm T(\OO)$ and hence the formula becomes \[ (F_1 \circ F_2) (x \mathrm T(\OO), y \mathrm T(\OO)) = \sum_{\mathrm T(\OO) \backslash \backslash z \in [T]}  F_1 (x,z) \cup F_2 (z,y). \]
Since $\mathrm T(F)$ is abelian, the $\mathrm T(\OO)$-action on $[T]$ is trivial and hence we are not summing across orbits, but simply across points $z \in [T]$. In particular we get \[ \widetilde{F_1 \circ F_2}(\mu) = (F_1 \circ F_2)( \mathrm T(\OO), \mu(\varpi) \mathrm T(\OO)) = \sum_{ \lambda \in X_*(\mathrm T) \equiv [T]}  F_1 (\mathrm T(\OO), \lambda(\varpi) \mathrm T(\OO)) \cup F_2 (\lambda(\varpi) \mathrm T(\OO), \mu(\varpi) \mathrm T(\OO)). \]
On the other hand, the algebra tensor product structure gives \[ \widetilde F_1 \cdot \widetilde F_2 (\mu) = \sum_{\lambda \in X_*(\mathrm T)} \widetilde F_1(\lambda) \cdot \widetilde F_2(\lambda^{-1} \mu) = \sum_{\lambda \in X_*(\mathrm T)} F_1(\mathrm T(\OO), \lambda(\varpi) \mathrm T(\OO)) \cup F_2(\mathrm T(\OO), \lambda(\varpi)^{-1} \mu(\varpi) \mathrm T(\OO)), \] so it remains to show that $F_2(\mathrm T(\OO), \lambda(\varpi)^{-1} \mu(\varpi) \mathrm T(\OO)) = F_2(\lambda(\varpi)\mathrm T(\OO), \mu(\varpi) \mathrm T(\OO))$ for every $\lambda, \mu \in X_*(\mathrm T)$. This follows immediately from $T$-invariance of $F_2$ and the fact that the conjugation action of $\lambda(\varpi)$ on $H^*(\mathrm T(\OO), S)$  is trivial.
\end{proof}
In a way similar to proposition \ref{convgroupoid}, we can re-interpret convolution on the algebra $\mathcal H_P$ whose definition was outlined in section \ref{secnotation}.
\begin{prop} \label{convgroupoidP}
Let $[M] = P_F / P^{\circ} = M_F / M_{\OO}$. Consider the groupoid $\left( P_F \curvearrowright [M]^2 \right)$ whose compactly supported cohomology $\mathbb H^*_c ( \left( P_F \curvearrowright [M]^2 \right) ) $ is the derived Hecke algebra $\mathcal H_P$. Consider the following diagram: \begin{displaymath} \xymatrix{ \left( P_F \curvearrowright [M]^2 \right) & & \left( P_F \curvearrowright [M]^2 \right) \\ & \left( P_F \curvearrowright [M]^3 \right) \ar[lu]_{i_{1,2}} \ar[ru]^{i_{2,3}} \ar[d]_{i_{1,3}} \\ & \left( P_F \curvearrowright [M]^2 \right) } \end{displaymath}
where the map $i_{1,2}: \left( P_F \curvearrowright [M]^3 \right) \lra \left( P_F \curvearrowright [M]^2 \right)$ is projection on the first and second factors on object, and the obvious inclusion of morphisms - similarly for $i_{1,3}$ and $i_{2,3}$.

Letting $F_1, F_2 \in \mathbb H^*_c(\left( P_F \curvearrowright [M]^2 \right))$, we describe their convolution in the following way: we pullback them to $\left( P_F \curvearrowright [M]^3 \right)$ via the maps $i_{1,2}$ and $i_{2,3}$, we cup them as in equation \ref{cupproduct}, and then push them forward to $\left( P_F \curvearrowright [M]^2 \right)$ via $i_{1,3}$.
\end{prop}
\begin{proof}
This follows from fact \ref{generaldha} once we check that the three conditions of the fact are satisfied. \\
Obviously $i_{1,3}$ is a covering morphism. Given $(m_1P^{\circ}, m_2P^{\circ}, m_3P^{\circ}) \in [M]^3$ its stabilizer is $\mathrm M(F)_{m_1, m_2, m_3} \mathrm V(F)$ while the stabilizer of $(m_1P^{\circ}, m_2P^{\circ})$ is $\mathrm M(F)_{m_1, m_2} \mathrm V(F)$. The index is then $ \left[ \mathrm M(F)_{m_1, m_2} : \mathrm M(F)_{m_1, m_2, m_3} \right]$ which is finite since $\mathrm M(\OO)$ is open compact in $\mathrm M(F)$. This proves condition 3. \\
Condition 1 says that $i_{1,2}^* F_1 \cup i_{2,3}^* F_2$ is $i_{1,3}$-fiberwise compactly supported for each $F_1, F_2 \in \mathcal H_P$ and as explained in the appendix, it suffices to check that for all $m_1, m_2 \in [M]$ we have finite index of $\Stab_{P_F}(m_1, m_2) \subset \Stab_{P_F}(m_1)$ which holds exactly as in the previous paragraph. \\
Finally, to check condition two we notice that since $ V_F$ is normal in $P_F$ we have \[ P^{\circ} m P^{\circ} n P^{\circ} = P^{\circ} m M_{\OO} n P^{\circ} = V_F \left( M_{\OO} m M_{\OO} n M_{\OO} \right) V_F. \]
Since $M_{\OO} \subset M_F$ is open and compact, we have a finite disjoint union $M_{\OO} m M_{\OO} n M_{\OO} = \bigsqcup_{i=1}^N M_{\OO} m_i M_{\OO}$ and thus \[ P^{\circ} m P^{\circ} n P^{\circ} = V_F \left( \bigsqcup_{i=1}^N M_{\OO} m_i M_{\OO} \right) V_F = \bigcup_{i=1}^N P^{\circ} m_i P^{\circ}. \]
\end{proof}
We mention here a few relevant general results concerning puhsforward and pullback of cohomology classes, whose proofs we delay until appendix \ref{appendice}.
\begin{restatable}{lem}{primaformula} \label{corescupres}
Let \begin{displaymath} \xymatrix{ X \ar[r]^{k} \ar[d]_{i} & Z \\ Y \ar[ru]_{j} & } \end{displaymath} be a commutative triangle of topological groupoid morphisms. Suppose that $i$ is a finite covering morphism and that $j$ induces inclusions between isotropy groups.

Let then $F \in \mathbb H^*(X)$ be $i$-fiberwise compactly supported and $G \in \mathbb H^*(Z)$, we have \[ i_* F \cup j^* G = i_* \left( F \cup k^* G \right) \textnormal{ in } \mathbb H^*(Y). \]
Moreover, if $j$ and $k$ induce finite-fibers maps on connected components, then each operation on cohomology preserves the compact support and thus if $F \in \mathbb H^*_c(X)$ and $G \in \mathbb H^*_c(Z)$, the formula above holds in $\mathbb H^*_c(Y)$.
\end{restatable}
\begin{restatable}{lem}{secondaformula} \label{cohomswitch}
Let \begin{displaymath} \xymatrix{ Z \ar[r]_{\widetilde \pi} \ar[d]_{\widetilde i} & X \ar[d]^i \\ Y \ar[r]_{\pi} & A } \end{displaymath} be a pullback square of topological groupoids, where $i$ is a finite covering morphism. Suppose also that $\pi$ is continuous and induces open injections of isotropy groups at all objects (e.g. $\pi$ is also a finite covering morphism, but in fact it suffices that every morphisms in $A$ has \emph{at most} a unique lift under $\pi$).

Let $F \in \mathbb H^*(X)$ be $i$-fiberwise compactly supported, then $\widetilde \pi^* F$ is $\widetilde i$-fiberwise compactly supported and we have \[ (\widetilde i)_* \left( (\widetilde \pi)^* F \right) = \pi^* ( i_* F). \]
If moreover $\pi$ induces a finite-fibers map on connected components, then given $F \in \mathbb H^*_c(X)$ the same formula above holds in $\mathbb H^*_c(Y)$.
\end{restatable}
As a first application of these lemmas, we can use fact \ref{generaldha} again which shows immediately that convolution on our derived Hecke algebra as defined in proposition \ref{convgroupoid} is associative.
\begin{prop} \label{convassociative}
Let $A, B, C \in \mathbb H^*_c( \underline G)$. Then $(A \circ  B) \circ C = A \circ ( B \circ C)$.
\end{prop}
\begin{proof} We already checked that fact \ref{generaldha} applies to this setup, and hence we obtain an algebra structure.
\end{proof}
\section{Satake map via groupoids} \label{secSatakegroupoids}
In this section we construct our Satake map as a combination of pushforwards and pullbacks in groupoid cohomology. Recall that the classical Satake homomorphism, both over $\C$ (see for instance \cite{gross}) and over $\FFp$ (as in \cite{herzig}), can be described as `integration over the unipotent radical'.
More precisely, in \cite{herzig2}, sections 2.2 and 2.3, Herzig defines the Satake transform between the $\bmod p$ classical Hecke algebra of a reductive $p$-adic group $G$ and that of a standard Levi subgroup $M$. Letting $P_F= M_F \ltimes N_F$ be the Levi decomposition of the standard parabolic $P$ containing $M$, he defines an algebra homomorphism $\mathcal S^M_G: \mathcal H^0_G \lra \mathcal H^0_M$ as \[ f \mapsto \left( m \mapsto \sum_{n \in N_F/ N_{\OO}} f(mn) \right). \]
The notation above takes $f \in \mathcal H^0_G$ to be a bi-$K$-invariant function: if we want to re-interpret this map in terms of the description as $G$-invariant functions on $G / K \times G/K$, we obtain \[ \mathcal S^M_G: f \mapsto \left( \left( M_{\OO}, mM_{\OO} \right) \mapsto \sum_{n \in N_F / N_{\OO}} f( K, mn K ) \right). \]
This is more suited for a generalization to our groupoid cohomology language: noticing that pushforward as in definition \ref{pushforwarddef} corresponds to `integration along the fiber', we define our Satake map to be the result of the pushforward and pullbacks as follows.
Recall that we denote $[M] = M_F / M_{\OO} = P_F / P^{\circ}$. Consider the diagram: \begin{displaymath} \xymatrix{ & \left( G_F \curvearrowright [G] \times G_F / P_{\OO} \right) \ar[r]^i \ar[d]_{\pi} & \underline G = \left( G_F \curvearrowright [G]^2 \right) \\ \underline M = \left( M_F \curvearrowright [M]^2 \right) \ar[r]_{i^G_M} & \left( G_F \curvearrowright [G] \times G_F / P^{\circ} \right) & } \end{displaymath}
The map $i^G_M$ is an inclusion, both at the level of groups and at the level of sets. $i$ is the identity at the group level, and the surjection $G_F / P_{\OO} \twoheadrightarrow G_F / K$ at the level of sets. $\pi$ is a quotient map at the level of sets, and it is the map encoding `integration over the unipotent radical' as discussed above. It is clear that all the morphisms are continuous with respect to the topologies induced on each groupoid as in example \ref{topgroupoidexmp}.
\begin{defn}[Satake map in groupoid cohomology] \label{satakemapdef}
We define the Satake map as the function $\mathcal \delta^G_M: \mathbb H^*_c(\underline G) \lra \mathbb H^*_c( \underline M)$ induced by the sequence of pullbacks and pushforwards in the diagram above.
\end{defn}
\begin{rem}
To prove that this is well-defined, we need to check that $\pi$ is a finite covering morphism and that each inclusion map induces a map on connected components which has finite fibers.

The fact that $\pi$ is a covering is immediate, since in both groupoids the same group $G_F$ is acting.
To check the finiteness assumption, it suffices to check finite index inclusion of isotropy groups on a representative for each connected component of the source groupoid: let then $(K, xP_{\OO})$ be one such representative with $x \in P_F$ by the Iwasawa decomposition.

Then the isotropy group of $(K, x P_{\OO})$ is $K \cap \Ad(x) P_{\OO}$. We also have $\pi \left( (K, x P_{\OO}) \right) = (K, x P^{\circ})$ whose isotropy group is $K \cap \Ad(x) P^{\circ}$.

Since $x \in P_F$, we have $\Ad(x) P^{\circ} \subset P_F$ so that $K \cap \Ad(x) P^{\circ} = K \cap \Ad(x) P^{\circ} \cap P_F = P_{\OO} \cap \Ad(x) P^{\circ}$. Similarly, $K \cap \Ad(x) P_{\OO} = P_{\OO} \cap \Ad(x) P_{\OO}$, which is finite index in $P_{\OO}$ since $P_{\OO} \subset P_F$ is open and closed. A fortiori, $P_{\OO} \cap \Ad(x) P_{\OO}$ will be finite index in $P_{\OO} \cap \Ad(x) P^{\circ}$, and in particular open since $P_{\OO} \cap \Ad(x) P^{\circ}$ is compact. This shows that we can pushforward along $\pi$ while preserving compact support.

It remains to check that each inclusion map has finite fibers on connected components.
For $i^G_M$ we have $\pi_0 \underline M = M_{\OO} \backslash M_F / M_{\OO}$ but also $(G_F \curvearrowright [G] \times G_F / P^{\circ})$ has connected components indexed by $K \backslash G_F / P^{\circ} \equiv P_{\OO} \backslash P_F / P^{\circ} \equiv M_{\OO} \backslash M_F / M_{\OO}$ where the first equivalence follows by the Iwasawa decomposition and the second by the Levi decomposition of $P_F$ and the fact that $P^{\circ} = M_{\OO} V(F)$. Hence in fact $i^G_M$ induces a bijection on connected components.

Finally, we need to show that $i$ induces a finite-fibers map on connected components, which is to say we need to show that each double coset $KgK$ decomposes into finitely many double cosets in $K \backslash G / P_{\OO}$. But by compactness of $K$ we already know that $KgK$ decomposes in finitely many left $K$-cosets, so a fortiori it will split into finitely many double $(K, P_{\OO})$-cosets. This proves that the Satake map is well-defined as a function $\mathbb H^*_c ( \underline G ) \lra \mathbb H^*_c ( \underline M)$.
\end{rem}

We check that this definition coincides with the one to be obtained in section \ref{secSatakeUPS} via an explicit use of the derived universal principal series. Let then $F \in \mathbb H^*_c(\underline G)$, we have that \[ \pi_* (i^*F)(x,y) = \sum_{\Stab_G(x,y) \backslash \backslash \pi^{-1}((x,y)) \ni (x, \widetilde y)} \cores{\Stab_G(x,y)}{\Stab_G(x, \widetilde y)} (i^*F)(x, \widetilde y) = \] \[ = \sum_{\Stab_G(x,y) \backslash \backslash \pi^{-1}((x,y)) \ni (x, \widetilde y)} \cores{\Stab_G(x,y)}{\Stab_G(x, \widetilde y)} \res{\Stab_G(i(x, \widetilde y))}{\Stab_G(x, \widetilde y)} F(i(x, \widetilde y)). \]
To check it coincide with our formula \ref{intermediatesatake} in the next section, we can pick one representative for each orbit, so let's fix $x = K$ and $y = m M_{\OO}V_F$. Then $\Stab_G(x,y) = \mathrm M(\OO)_m \mathrm V(\OO)$, while for each $(K, mv P_{\OO}) \in \pi^{-1}(K,mM_{\OO}V_F)$ we have $\Stab_G(K, mv P_{\OO}) = \mathrm P(\OO)_{mv}$.
The above sum becomes \[ \sum_{\mathrm M(\OO)_m \mathrm V(\OO) \backslash \backslash \left\{ mv \in P_F / P_{\OO} \right\} } \cores{\mathrm M(\OO)_m \mathrm V(\OO)}{\mathrm P(\OO)_{mv}} \res{K_{mv}}{\mathrm P(\OO)_{mv}} F(K,mvK) \] which coincides with formula \ref{intermediatesatake} for $(F.1)$ after the definition \ref{satakedefn} of the Satake homomorphism in section \ref{secSatakeUPS} once we prove that $\mathrm P(\OO)_{mv} = \mathrm M(\OO)_m \mathrm V(\OO) \cap K_{mv}$. \\
Clearly, $\mathrm P(\OO)_{mv} \subset K_{mv}$, and letting $\eta \nu \in \mathrm M(\OO) \ltimes \mathrm V(\OO) = \mathrm P(\OO)$ we have that $\Ad(mv) (\eta \nu) \in \mathrm P(\OO)$ forces $\Ad(m) \eta \in \mathrm M(\OO)$ and hence the Levi component of each element of $\mathrm P(\OO)_{mv}$ is in $\mathrm M(\OO)_m$, which proves that $P(\OO) \subset M(\OO)_m V(\OO)$. \\
For the opposite inclusion, we notice that $K_{mv} \cap \mathrm M(\OO)_m \mathrm V(\OO) = K \cap \Ad(mv) K \cap \mathrm P(\OO) \cap \mathrm M(\OO)_m \mathrm V(\OO) = \Ad(mv) \mathrm P(\OO) \cap \mathrm P(\OO) \cap \mathrm M(\OO)_m \mathrm V(\OO) \subset \mathrm P(\OO)_{mv}$.

The map $i^G_M$ induces a pullback in cohomology which corresponds simply to restriction, using the canonical inclusion of isotropy groups. Hence the map $\mathbb H^*_c(\underline G) \stackrel{\mathcal S}{\lra} \mathbb H^*_c (\underline M)$ results to be \[ \mathcal S^G_M(F) \left( M_{\OO}, m M_{\OO} \right) = \res{\mathrm M(\OO)_m \mathrm V(\OO)}{\mathrm M(\OO)_m} \sum_{\mathrm M(\OO)_m \mathrm V(\OO) \backslash \backslash \left\{ mv \in P_F / P_{\OO} \right\} } \cores{\mathrm M(\OO)_m \mathrm V(\OO)}{\mathrm P(\OO)_{mv}} \res{K_{mv}}{\mathrm P(\OO)_{mv}} F(K,mvK) \] which coincides with formula \ref{onemoresatakeformula} in section \ref{secSatakeUPS}.

\subsection{Transitivity of the Satake map} \label{sectiontransitivesatake}
In the subsection we show that the Satake map is transitive for inclusion among Levi factors, using the groupoid setup described above: this shows the transitivity claims of theorem \ref{mainthm} and corollary \ref{casepinvertible}.

Given an inclusion between standard parabolic subgroups $\mathrm P = \mathrm M \ltimes \mathrm V \supset \mathrm M' \ltimes \mathrm V' = \mathrm P'$, we obtain three groupoid diagrams for $\mathcal S^G_M$, $\mathcal S^M_{M'}$ and $\mathcal S^G_{M'}$ and we show that the composition of the first two induces on cohomology the same map as the third one, via repeated applications of lemmas \ref{corescupres} and \ref{cohomswitch}.
Roughly speaking, the equivalence of the diagrams boils down to the fact that integrating along the unipotent radical $V'_F$ of $P'_F$ is equivalent to integrating along the unipotent radical $V_F$ of $P_F$ first, then along the unipotent radical $(V'_F \cap M_F)$ of $(P'_F \cap M_F)$, since $V'_F = (V'_F \cap M_F) \ltimes V_F$.

Denote $\mathrm P'_M = \mathrm P \cap \mathrm M$, a parabolic subgroup of $\mathrm M$. With the groups defined as above, we have the diagram \begin{displaymath} \xymatrix{ & \left( G_F \curvearrowright [G] \times G_F / P'_{\OO} \right) \ar[r]^i \ar[d]_{\pi} & \underline G = \left( G_F \curvearrowright [G]^2 \right) \\ \underline M' = \left( M'_F \curvearrowright [M']^2 \right) \ar[r]_{i^G_{M'}} & \left( G_F \curvearrowright [G] \times G_F / (P')^{\circ} \right) & } \end{displaymath} which describes the Satake map $\mathcal S^G_{M'}$.
We want to prove that $\mathcal S^G_{M'} = \mathcal S^M_{M'} \circ \mathcal S^G_M$ which is to say that the same function between compactly supported cohomology algebras is obtained by composing the following two diagrams:
\begin{displaymath} \xymatrix{ & \left( G_F \curvearrowright [G] \times G_F / P_{\OO} \right) \ar[r]^i \ar[d]_{\pi} & \underline G = \left( G_F \curvearrowright [G]^2 \right) \\ \underline M = \left( M_F \curvearrowright [M]^2 \right) \ar[r]_{i^G_M} & \left( G_F \curvearrowright [G] \times G_F / P^{\circ} \right) & } \end{displaymath} and then
\begin{displaymath} \xymatrix{ & \left( M_F \curvearrowright [M] \times M_F / (P'_M)_{\OO} \right) \ar[r]^i \ar[d]_{\pi} & \underline M = \left( M_F \curvearrowright [M]^2 \right) \\ \underline M' = \left( M'_F \curvearrowright [M']^2 \right) \ar[r]_{i^M_{M'}} & \left( M_F \curvearrowright [M] \times M_F / (P'_M)^{\circ} \right) & } \end{displaymath}

To prove the statement we will repeatedly apply lemma \ref{cohomswitch}, and replace the diagrams with ones that have the same effect on cohomology, until the two diagrams will coincide.

Let's start with the second one.
\begin{claim} \label{pullback1}
The diagram \begin{displaymath} \xymatrix{ \left( M_F \curvearrowright [M] \times P_F / P_{\OO} \right) \ar[d]_{\pi'} \ar[r]^{i'} & \left( G_F \curvearrowright [G] \times G_F / P_{\OO} \right)  \ar[d]_{\pi}  \\ \underline M = \left( M_F \curvearrowright [M]^2 \right) \ar[r]_i  & \left( G_F \curvearrowright [G] \times G_F / P^{\circ} \right) } \end{displaymath} is a pullback square of groupoids. Moreover, the maps involved satisfy the assumptions of lemma \ref{cohomswitch}: that is to say, $\pi$ is a finite covering morphism and $i$ induces inclusion on isotropy groups, and a finite-fibers map on connected components.
\end{claim}
\begin{proof}
While checking that the Satake map was well-defined we already checked all conditions to make sure that the pushforward $\pi_*$ and the pullback $i^*$ were well-defined on compactly supported groupoid cohomology: the only thing that remains to check is that $i$ induces inclusions of isotropy groups. \\
Letting then $(m_1 M_{\OO} , m_2 M_{\OO}) \in [M]^2$, we have that $i \left( (m_1 M_{\OO} , m_2 M_{\OO}) \right) = (m_1 K , m_2 P^{\circ})$. We have \[ \Stab_{M_F} (m_1 M_{\OO} , m_2 M_{\OO}) = \Ad(m_1) M_{\OO} \cap \Ad(m_2) M_{\OO} \subset \Ad(m_1) K \cap \Ad(m_2) P^{\circ} = \Stab_{G_F} (m_1 K , m_2 P^{\circ}) \] and hence $i$ induces inclusions on isotropy groups.

In the case of groupoids consisting of groups acting on sets with compatible maps, the pullback square is obtained by taking the pullback of the groups acting on the pullback of the sets (proposition 4.4 (ii) in \cite{bhk}).
It's clear that the pullback of the groups is $M_F$ and that the first factor of the pullback of the sets is $[M]$.
For the second factor, by definition of pullback of sets we obtain \[ G_F / P_{\OO} \times_{G_F / P^{\circ}} [M] =  \{ (gP_{\OO}, m M_{\OO} ) \in G_F / P_{\OO} \times [M] \, | \, \pi(gP_{\OO}) = i( m M_{\OO}) \} = \] \[ = \{ (gP_{\OO}, m M_{\OO} ) \in G_F / P_{\OO} \times [M] \, | \, gP^{\circ} = m P^{\circ} \}. \] As $m \in M_F$, we have $m P^{\circ} \subset P_F$ and thus the above condition forces $g \in P_F$ too.
Let then $g = nv$ be its Levi decomposition and notice than that $gP_{\OO} = nv P_{\OO} \stackrel{\pi}{\mapsto} nv P^{\circ} = n P^{\circ}$.
In particular, $m P^{\circ} = n P^{\circ}$ with $m, n \in M_F$ forces then $m M_{\OO} = n M_{\OO}$, i.e. $m = n$ as cosets in $[M]$. We obtain \[ G_F / P_{\OO} \times_{G_F / P^{\circ}} [M] \cong P_F / P_{\OO} \] with maps \[ P_F / P_{\OO} \stackrel{i'}{\lra} G_F / P_{\OO} \textnormal{ and } P_F / P_{\OO} \stackrel{\pi'}{\lra} [M] \] being respectively inclusion and projection onto the $M_F$-component of the Levi factorization $P_F = M_F \ltimes V_F$.
\end{proof}
Applying lemma \ref{cohomswitch} shows that for the purpose of moving around cohomology classes, we can replace the second diagram with the following one: \begin{displaymath} \xymatrix{ \left( M_F \curvearrowright [M] \times P_F / P_{\OO} \right) \ar[d]_{\pi'} \ar[r]^{i'} & \left( G_F \curvearrowright [G] \times G_F / P_{\OO} \right) \ar[r]^i & \underline G = \left( G_F \curvearrowright [G]^2 \right) \\ \underline M = \left( M_F \curvearrowright [M]^2 \right)  } \end{displaymath}
Using again claim \ref{pullback1} but replacing $\mathrm M$ with $\mathrm M'$, we obtain that \begin{displaymath} \xymatrix{ \left( M'_F \curvearrowright [M'] \times P'_F / P'_{\OO} \right) \ar[d]_{\pi'} \ar[r]^{i'} & \left( G_F \curvearrowright [G] \times G_F / P'_{\OO} \right)  \ar[d]_{\pi}  \\ \underline M' = \left( M'_F \curvearrowright [M']^2 \right) \ar[r]_i  & \left( G_F \curvearrowright [G] \times G_F / (P')^{\circ} \right) } \end{displaymath} is also a pullback square of groupoids, with maps satisfying the additional assumptions of lemma \ref{cohomswitch}.
Hence we can replace the first diagram by \begin{displaymath} \xymatrix{ \left( M'_F \curvearrowright [M'] \times P'_F / P'_{\OO} \right) \ar[d]_{\pi'} \ar[r]^{i'} & \left( G_F \curvearrowright [G] \times G_F / P'_{\OO} \right) \ar[r]^i & \underline G = \left( G_F \curvearrowright [G]^2 \right) \\ \underline M' = \left( M'_F \curvearrowright [M']^2 \right)  } \end{displaymath}
Finally claim \ref{pullback1} with $\mathrm G$ replaced by $\mathrm M$ and $\mathrm M$ by $\mathrm M'$ yields the pullback square of groupoids \begin{displaymath} \xymatrix{ \left( M'_F \curvearrowright [M'] \times (P'_M)_F / (P'_M)_{\OO} \right) \ar[d]_{\pi'} \ar[r]^{i'} & \left( M_F \curvearrowright [M] \times M_F / (P'_M)_{\OO} \right)  \ar[d]_{\pi}  \\ \underline M' = \left( M'_F \curvearrowright [M']^2 \right) \ar[r]_i  & \left( M_F \curvearrowright [M] \times M_F / (P'_M)^{\circ} \right) } \end{displaymath} with maps again satisfying the additional assumption of lemma \ref{cohomswitch}.
This allows us to replace the third diagram with \begin{displaymath} \xymatrix{ \left( M'_F \curvearrowright [M'] \times (P'_M)_F / (P'_M)_{\OO} \right) \ar[d]_{\pi'} \ar[r]^{i'} & \left( M_F \curvearrowright [M] \times M_F / (P'_M)_{\OO} \right) \ar[r]^i & \underline M = \left( M_F \curvearrowright [M]^2 \right) \\ \underline M' = \left( M'_F \curvearrowright [M']^2 \right)  } \end{displaymath}
Notice that the map $i: \left( G_F \curvearrowright [G] \times G_F / P'_{\OO} \right)\lra \underline G = \left( G_F \curvearrowright [G]^2 \right)$ involved in the first diagram factors through the similar map $i: \left( G_F \curvearrowright [G] \times G_F / P_{\OO} \right)\lra \underline G = \left( G_F \curvearrowright [G]^2 \right)$ of the second diagram, hence it remains to show that the diagrams \begin{displaymath} \xymatrix{ \left( M'_F \curvearrowright [M'] \times P'_F / P'_{\OO} \right) \ar[d]_{\pi} \ar[r]^{i'} & \left( G_F \curvearrowright [G] \times G_F / P'_{\OO} \right) \ar[r]^i & \left( G_F \curvearrowright [G] \times G_F / P_{\OO} \right) \\ \underline M' = \left( M'_F \curvearrowright [M']^2 \right)  } \end{displaymath} and \begin{displaymath} \xymatrix{ & \left( M_F \curvearrowright [M] \times P_F / P_{\OO} \right) \ar[d]_{\pi} \ar[r]^i & \left( G_F \curvearrowright [G] \times G_F / P_{\OO} \right) \\ \left( M'_F \curvearrowright [M'] \times (P'_M)_F / (P'_M)_{\OO} \right) \ar[d]_{\pi} \ar[r]^i & \underline M = \left( M_F \curvearrowright [M]^2 \right) & \\ \underline M' = \left( M'_F \curvearrowright [M']^2 \right) & & } \end{displaymath} induce the same map in cohomology. This is a consequence of the following
\begin{claim}
The diagram \begin{displaymath} \xymatrix{ \left( M'_F \curvearrowright [M'] \times P'_F / P'_{\OO} \right) \ar[d]_{\pi'} \ar[r]^{i'} & \left( M_F \curvearrowright [M] \times P_F / P_{\OO} \right)  \ar[d]_{\pi}  \\ \left( M'_F \curvearrowright [M'] \times (P'_M)_F / (P'_M)_{\OO} \right) \ar[r]_i  & \underline M = \left( M_F \curvearrowright [M]^2 \right) } \end{displaymath} is a pullback square of groupoids. Moreover, it satisfies the assumption of lemma \ref{cohomswitch}.
\end{claim}
\begin{proof}
First of all, $\pi$ comes from the pullback square of claim \ref{pullback1}, hence it is a finite covering of groupoids as shown in the proof of lemma \ref{cohomswitch}.

Also, $i$ is the composition of two maps $ \left( M'_F \curvearrowright [M'] \times (P'_M)_F / (P'_M)_{\OO} \right) \stackrel{i'}{\lra} \left( M_F \curvearrowright [M] \times M_F / (P'_M)_{\OO} \right) \stackrel{i}{\lra} \underline M$, with $i'$ coming from the pullback square of claim \ref{pullback1} and thus inducing injections on isotropy groups and a finite-fibers map on connected components by the proof of lemma \ref{cohomswitch}. On the other hand, $\left( M_F \curvearrowright [M] \times M_F / (P'_M)_{\OO} \right) \lra \underline M$ induces injections on isotropy groups and a finite-fibers map on connected components by the remark following definition \ref{satakemapdef}.

It remains to check that $ \left( M'_F \curvearrowright [M'] \times P'_F / P'_{\OO} \right)$ with the given maps $\pi'$ and $i'$ is indeed the pullback square.
Just as before, by proposition 4.4(ii) in \cite{bhk} we can compute the pullback of the groups and the sets separately. It is clear that the pullback of the groups is $M'_F$, and that the pullback of the sets - computed separately on each factor - has first factor equal to $[M']$.
Finally, For the second factor by definition of pullback of sets we obtain \[ P_F / P_{\OO} \times_{M_F / M_{\OO}} (P'_M)_F / (P'_M)_{\OO} =  \{ (pP_{\OO}, p' (P'_M)_{\OO} ) \in P_F / P_{\OO} \times (P'_M)_F / (P'_M)_{\OO} \, | \, \pi(pP_{\OO}) = i( p' (P'_M)_{\OO}) \} = \] \[ = \{ (pP_{\OO}, p' (P'_M)_{\OO} ) \in P_F / P_{\OO} \times (P'_M)_F / (P'_M)_{\OO} \, | \, mM_{\OO} = p' M_{\OO} \textnormal{ using the Levi decomposition } p = mv \}. \]
In other words, there exists $m_0 \in M_{\OO}$ such that $m = p' m_0$. Notice that we have then $p P_{\OO} = mv P_{\OO} = p' m_0 v m_0^{-1} P_{\OO} = p' \widetilde v P_{\OO}$, since $M_F$ normalizes $V_F$.
We obtain then \[ \{ (pP_{\OO}, p' (P'_M)_{\OO} ) \in P_F / P_{\OO} \times (P'_M)_F / (P'_M)_{\OO} \, | \, mM_{\OO} = p' M_{\OO} \textnormal{ with } p = mv \} = \] \[ = \{ (p' \widetilde v P_{\OO}, p' (P'_M)_{\OO} ) \in P_F / P_{\OO} \times (P'_M)_F / (P'_M)_{\OO}  \} \equiv P'_F / P'_{\OO} \] with maps \[ P'_F / P'_{\OO} \stackrel{i'}{\lra} P_F / P_{\OO} \textnormal{ and } P'_F / P'_{\OO} \stackrel{\pi'}{\lra} (P'_M)_F / (P'_M)_{\OO} \] the first one being inclusion and the second one induced by the semidirect product decomposition $P'_F = (P'_M)_F \ltimes V_F$, which is just the restriction to $P'_F$ of the Levi decomposition $P_F = M_F \ltimes V_F$.
\end{proof}

\section{Satake homomorphism via the Universal Principal Series} \label{secSatakeUPS}
In this section we prove that the map defined in the previous one is a morphism of graded algebras (at least in degrees $0$ and $1$): notice indeed that the construction of section \ref{secSatakegroupoids} holds in any degree, and provides us with a Satake map in all degrees - but at the present time the only thing we can show in full generality is that $\mathcal S^G_M$ is an algebra homomorphism in degrees $0$ and $1$. We also describe in this section under what conditions $\mathcal S^G_M$ is an algebra homomorphism in all degrees, and we make them explicit for $\mathrm M = \mathrm T$. In particular, we complete in this section the proofs of theorems \ref{mainthm} and \ref{satakealldegrees}.

To show that our Satake map is a morphism, we re-interpret it in a different way, via a generalization to the derived setting of the classical universal principal series.

We start by recalling the classical definition of the Satake homomorphism via the universal principal series (see for instance \cite{HKP}, section 4).
Consider the space of compactly supported, locally constant functions on $G$ which are left-invariant by $K$ and right-invariant by $\mathrm T(\OO) \mathrm U(F)$ - let us temporarily denote this by $C^{\infty}_c \left( K \backslash G / B^{\circ} \right)$.
This module admits a left action of $\mathcal H^0(G,K)$ by convolution on the left, and a right action of the group algebra of $\mathrm T(F) / \mathrm T(\OO) \cong X_*(\mathrm T)$. Moreover, using the Iwasawa and then the Levi decomposition one shows that $C^{\infty}_c \left( K \backslash G / B^{\circ} \right)$ is a free module of rank 1 over this group algebra, with the isomorphism given by $f \mapsto 1.f$ where $1 =1_{K \mathrm T(\OO) \mathrm U(F)}$ is the `spherical vector'.
Therefore, the rule $F.1 =1. \mathcal S^G_T(F)$ yields a well-defined algebra homomorphism from $\mathcal H^0(G)$ to the group algebra of $X_*(\mathrm T)$, which coincide with the classical spherical Hecke algebra of the torus.

In this section we generalize the above construction to our derived setting, at least for the degree $1$ part of the derived Hecke algebra. The degree 1 obstacle is of technical nature, and it does not seem easy to bypass for a general Levi subgroup $\mathrm M$ (even though it disappears if we take $\bmod l$-torsion coefficients, rather than $\bmod p$-torsion). 

Our construction (with groupoids generalizing the spaces involved in the universal principal series description above) gives rise to a homomorphism which we show coincide with the map defined in section \ref{secSatakegroupoids} by comparing explicit formulas.\footnote{This comparison has in fact already been done in section \ref{secSatakegroupoids}.}

Let $\mathrm P = \mathrm M \ltimes \mathrm V$ be a standard parabolic, with $\mathrm M$ the standard Levi factor.
\begin{lem} \label{dhaisom}
Consider the groupoid $\underline P = \left( P_F \curvearrowright [M]^2 \right)$ (where the use that $[M] =M_F / M_{\OO} = P_F / P^{\circ}$) and the continuous morphism $i: \underline M \lra \underline P$ defined as the identity on the spaces, and the inclusion $M_F \hookrightarrow P_F$ at the level of groups.

This induces a pullback map in compactly supported cohomology $i^*: \mathbb H^*_c(\underline P) \lra \mathbb H^*_c(\underline M)$, and each compactly supported cohomology is a derived Hecke algebra (respectively $\mathcal H_P$ and $\mathcal H_M)$ as per propositions \ref{convgroupoidP} and \ref{convgroupoid}.

The pullback $i^*: \mathcal H_P \lra \mathcal H_M$ is an isomorphism of graded algebras.
\end{lem}
\begin{proof}
We start by remarking that the pullback along $i$ indeed preserves compact support, because $\pi_0 \underline P = P^{\circ} \backslash P_F / P^{\circ} \equiv M_{\OO} \backslash M_F / M_{\OO} = \pi_0 \underline M$ by the Levi decomposition. The map induced by $i$ on connected components is thus the identity, and in particular it has finite fibers.

As usual, we interpret the algebra $\mathcal H_P$ as $\mathrm P(F)$-equivariant cohomology classes \[  A: \mathrm P(F) / P^{\circ} \times \mathrm P(F) / P^{\circ} \lra \bigoplus_{x,y \in \mathrm P(F) / P^{\circ}} H^* \left(\mathrm P(F)_{x,y} \right) \] with support on finitely many orbits and such that $A \left( m P^{\circ}, n P^{\circ} \right) \in H^* \left( \mathrm P(F)_{m,n} \right)$.

Denoting by $\mathrm P(F)_{m,n} = \Stab_{\mathrm P(F)}(mP^{\circ}, n P^{\circ})$ and $\mathrm M(F)_{m,n} = \Stab_{\mathrm M(F)}(mM_{\OO}, nM_{\OO})$, we have an explicit formula for the pullback map $i^*: \mathcal H_P \lra \mathcal H_M$, $A \mapsto \widetilde A$ as \begin{equation} \label{pullbackPtoM} \widetilde A \left( mM_{\OO}, nM_{\OO} \right) = \res{\mathrm P(F)_{m,n}}{\mathrm M(F)_{m,n}} A \left( m P^{\circ}, n P^{\circ} \right). \end{equation}

To check that this is an algebra homomorphism, we need to show that given $A, B \in \mathcal H_P$ we have $i^*A \circ i^*B = i^*(A \circ B)$ as elements of $\mathbb H^*_c (\underline M)$. We interpret both sides in terms of the usual groupoids diagrams: \begin{displaymath} \xymatrix{ A \in \underline P & & \underline P \ni B \\ & \left( P_F \curvearrowright [M]^3 \right) \ar[lu]_{i_{1,2}} \ar[ru]^{i_{2,3}} \ar[d]_{i_{1,3}} \\ & \underline P \\ & \underline M  \ar[u]_i } \end{displaymath} is the diagram representing the right hand side, while \begin{displaymath} \xymatrix{ A \in \underline P & & \underline P \ni B \\ i^*A \in \underline M \ar[u]^i & & \underline M \ni i^*B \ar[u]_i \\ & \left( M_F \curvearrowright [M]^3 \right) \ar[lu]_{i_{1,2}} \ar[ru]^{i_{2,3}} \ar[d]_{i_{1,3}} \\ & \underline M   } \end{displaymath} is the diagram representing the left hand side. By composing the arrows on top, this latter diagram simplifies to \begin{displaymath} \xymatrix{ A \in \underline P & & \underline P \ni B \\ & \left( M_F \curvearrowright [M]^3 \right) \ar[lu]_{i \circ i_{1,2}} \ar[ru]^{i \circ i_{2,3}} \ar[d]_{i_{1,3}} \\ & \underline M   } \end{displaymath}
As for the other diagram, we use the following fact:
\begin{claim}
The diagram \begin{displaymath} \xymatrix{ \left( M_F \curvearrowright [M]^3 \right) \ar[d]_i \ar[rr]^{i_{1,3}} & & \underline M \ar[d]^i \\ \left( P_F \curvearrowright [M]^3 \right) \ar[rr]_{i_{1,3}} & & \underline P } \end{displaymath} is a pullback square of groupoids. Moreover, it satisfies the assumption of lemma \ref{cohomswitch}: that is to say, $i_{1,3}$ is a finite covering morphism and $i$ induces inclusion on isotropy groups, and a finite-fibers map on connected components.
\end{claim}
\begin{proof}[Proof of claim] Notice that $i_{1,3}$ is clearly a continuous covering morphism, since we have the same group $P_F$ acting on both the source and the target groupoid.
Moreover, given $(xP^{\circ}, y P^{\circ}, z P^{\circ}) \in \Ob ( [M]^3)$ we have that \[ \Stab_{\underline P} \left( (xP^{\circ}, y P^{\circ}, z P^{\circ}) \right) = \Ad(x) P^{\circ} \cap  \Ad(y) P^{\circ} \cap  \Ad(z) P^{\circ} = \mathrm M(F)_{x,y,z} \mathrm V(F) \] while \[ \Stab_{\underline P} \left( (xP^{\circ}, z P^{\circ}) \right) = \Ad(x) P^{\circ} \cap  \Ad(z) P^{\circ} = \mathrm M(F)_{x,z} \mathrm V(F) \] so that the index of the inclusion of isotropy groups is the index $\left[ \mathrm M(F)_{x,z} : \mathrm M(F)_{x,y,z} \right]$, which is finite since $\mathrm M(\OO)$ is open and compact in $\mathrm M(F)$. This proves that $i_{1,3}$ is a finite covering morphism, since it also implies that $\mathrm M(F)_{x,y,z}$ is open in $\mathrm M(F)_{x,z}$. \\
That $i$ induces an inclusion of isotropy groups is immediate, and as for the connected components we have that $\pi_0 \underline M = M_{\OO} \backslash M_F / M_{\OO} \equiv P^{\circ} \backslash P_F / P^{\circ} = \pi_0 \underline P$, so that the map induced by $i$ is in fact a bijection.

It remains to show that the pullback square is indeed $\left( M_F \curvearrowright [M]^3 \right)$, with the maps as given in the diagram. As mentioned before, the pullback square of a diagram of groupoids given by groups acting on sets with compatible maps is the pullback of the groups acting on the pullback of the sets. \\
The pullback groups is obviously $M_F$, while for the sets we get $[M]^3$. The maps are the ones given.
\end{proof}
The usual application of lemma \ref{cohomswitch} shows then that the diagram for the right hand side is replaced by \begin{displaymath} \xymatrix{ A \in \underline P & & \underline P \ni B \\ & \left( P_F \curvearrowright [M]^3 \right) \ar[lu]_{i_{1,2}} \ar[ru]^{i_{2,3}}  \\ & \left( M_F \curvearrowright [M]^3 \right) \ar[u]_i \ar[d]_{i_{1,3}} \\ & \underline M  } \end{displaymath} and after composing the upper arrows the latter diagram becomes identical to the one for $i^*A \circ i^*B$.

This proves that the pullback is an algebra homomorphism. It remains to show that it is bijective: this amounts to showing that the two spaces of compactly supported cohomology classes are identified as $S$-modules under pullback.

Given that we have an equality of double coset spaces $P^{\circ} \backslash P_F / P^{\circ} \equiv M_{\OO} \backslash M_F / M_{\OO}$ and in light of the explicit formula \ref{pullbackPtoM}, it suffices to check that \[ H^i \left( \Stab_{\mathrm P(F)} \left( P^{\circ}, x P^{\circ} \right) , S \right) \stackrel{\res{}{}}{\lra} H^i \left( \Stab_{\mathrm M(F)} \left( M_{\OO}, x M_{\OO}  \right) , S \right) \] is an isomorphism for all $i \ge 0$. We find \[ \Stab_{\mathrm P(F)} \left( P^{\circ}, x P^{\circ} \right) = P^{\circ} \cap x P^{\circ} x^{-1} = \] \[ =  \left( M_{\OO} \cap x M_{\OO} x^{-1} \right) V_F = \mathrm M(\OO)_x \mathrm V(F) \] and \[ \Stab_{\mathrm M(F)} \left( M_{\OO}, x M_{\OO}  \right) = M_{\OO} \cap x M_{\OO} x^{-1} = \mathrm M(\OO)_x . \]
Hence we need to show that the restriction map \[ \res{}{} :H^i \left( \mathrm M(\OO)_x \mathrm V(F) , S \right) \lra H^i \left( \mathrm M(\OO)_x , S \right) \] is an isomorphism in all degrees $0$ and $1$. In degree $0$ this is obvious, as both cohomology groups are copies of $S$ and the map is the identity map.

In degree $1$, both cohomology groups are $\Hom$-groups, so it suffices to show that $\mathrm V(F) \subset \left[ \mathrm M(\OO)_x \mathrm V(F), \mathrm M(\OO)_x \mathrm V(F) \right] $.
The equality $\pi_0 \underline P \equiv \pi_0 \underline M$ implies that $x$ can be chosen among a set of coset representatives for $M_{\OO} \backslash M_F / M_{\OO}$, and the Cartan decomposition for $M$ says that we can then choose $x \in T_F$.
Therefore, $\mathrm T(\OO) \subset \mathrm M(\OO)_x$ and hence we can show that $\mathrm V(F)$ is in the commutator subgroup of $\mathrm T(\OO) \mathrm V(F) \subset \mathrm M(\OO)_x \mathrm V(F)$ by working on each root space $\mathrm U_{\alpha}(F) \subset \mathrm V(F)$ - here we use crucially that the cardinality of the residue field of $F$ is at least $5$.

The following argument that finishes the proof for all degrees $i > 1$ (and in fact works for $i=1$ as well) was suggested to us by Florian Herzig, who we are very thankful to. \\
Consider the short exact sequence $1 \lra \mathrm V(F) \lra \mathrm M(\OO)_x \mathrm V(F) \lra \mathrm M(\OO)_x \lra 1$ where we identify the quotient with the Levi factor. The Hochschild-Serre spectral sequence for continuous cohomology with trivial $S$-coefficients gives \[ E_2^{p,q} = H^p \left( \mathrm M(\OO)_x, H^q(\mathrm V(F), S) \right) \Rightarrow H^{p+q} \left( \mathrm M(\OO)_x \mathrm V(F), S \right). \]
It suffices to prove that $H^q( \mathrm V(F), S) = 0$ for all $q \ge 1$. Indeed, it will follow that the spectral sequence collapses on the second page, and the inflation map $H^p(\mathrm M(\OO)_x, S) \stackrel{\mathrm{infl}}{\lra} H^p( \mathrm M(\OO)_x \mathrm V(F), S)$ is an isomorphism in all degrees $p \ge 0$.
A consideration on the cochain complex shows easily that inflation is a right inverse to restriction, and thus we will obtain that the restriction map $H^p( \mathrm M(\OO)_x \mathrm V(F), S ) \lra H^p(\mathrm M(\OO)_x , S)$ is an isomorphism as well.

We prove that $H^q(\mathrm V(F), S)=0$ for any $q \ge 1$ by induction on $\dim \mathrm V$. The inductive step is immediate: as $\mathrm V$ is nilpotent, if $\dim \mathrm V \ge 2$ we can choose a short exact sequence $1 \lra \mathrm V' \lra \mathrm V \lra \mathrm V'' \lra 1$ having $\dim \mathrm V > \dim \mathrm V', \dim \mathrm V \ge 1$ and then Hochschild-Serre yields \[ E_2^{p,q} = H^p \left( \mathrm V''(F), H^q(\mathrm V'(F), S) \right) \Rightarrow H^{p+q} \left( \mathrm \mathrm V(F), S \right). \]
By induction, $H^q(\mathrm V'(F),S) = 0$ for all $q \ge 1$ and thus the Hochschild-Serre spectral sequence collapses on the second page, to yield that the inflation map $H^p(\mathrm V''(F), S) \lra H^p(\mathrm V(F),S)$ is an isomorphism in all degrees $p \ge 0$. By repeating the argument replacing $\mathrm V$ with $\mathrm V''$, we can assume that $\dim \mathrm V''=1$, and then the claim will follow from the base case of the induction.

It remains thus to show that $H^q(F, S) = 0$ for all $q \ge 1$. For $F = \Qp$ this has been proved by Emerton in the course of the proof of lemma 4.3.7 in \cite{emerton2}.
By induction on $d = \dim_{\Qp} F$ (the base case being Emerton's result), the general case follows by using again the Hochschild-Serre spectral sequence for $1 \lra \Qp^{\oplus d-1} \lra F \cong \Qp^{\oplus d} \lra \Qp \lra 1$ - or alternatively the K\"unneth formula for continuous cohomology.
\end{proof}
\begin{rem}
Notice that the last part of the proof holds verbatim even for coefficient rings $S$ where $p$ is invertible. Indeed, the only place where we used that $S = \Z / p^a \Z$ is when invoking Emerton's result but his proof describes $H^i (\Qp, S)$ as an inverse limit of copies of $H^i (p^k \Zp, S)$ - all these are trivially zero if $p$ is invertible in $S$, for any $i \ge 1$.
\end{rem}
We now give a brief description of the content of the rest of this section. Consider the topological groupoid $\left( P_F \curvearrowright [P] \times [M] \right)$ whose compactly supported cohomology $\mathbb H^*_c \left( P_F \curvearrowright [P] \times [M] \right)$ is the space $\mathcal H_P ( P_{\OO} , P^{\circ} )$ of $\mathrm P(F)$-invariant cohomology classes $A$ on $P_F / P_{\OO} \times P_F / P^{\circ}$, supported on finitely many orbits and such that \[ A \left( x P_{\OO}, y P^{\circ} \right) \in H^* \left( \Stab_{\mathrm P(F)} \left( x P_{\OO}, y P^{\circ} \right) , S \right). \]
We have a continuous morphism of topological groupoids \[ \left( P_F \curvearrowright [P] \times [M] \right) \stackrel{i}{\lra} \left( G_F \curvearrowright [G] \times G/ P^{\circ} \right) \] given by inclusion $P_F \hookrightarrow G_F$ at the level of groups, `identity' on the first component $[P] = P_F / P_{\OO} \equiv G / K = [G]$ and injection on the second component $P_F / P^{\circ} \hookrightarrow G_F / P^{\circ}$. This turns out to be a homotopy equivalence (see fact \ref{homequivPtoG}) and thus gives an isomorphism on cohomology via pullback.

We also have a convolution action of $\mathcal H_G$ on the left of $\mathbb H^*_c \left( G_F \curvearrowright [G] \times G/ P^{\circ} \right)$ (denoted $ \mathcal H_G(K, P^{\circ})$ from now on) as well as a convolution action of $\mathcal H_P$ on the right of $\mathcal H_P(P_{\OO}, P^{\circ})$.

The diagram \begin{displaymath} \xymatrix{ \mathcal H_G = \mathcal H_G(K,K) \ar@/^2.0pc/[r] & \mathcal H_G(K, P^{\circ}) \ar[d]^{i^*}_{\cong} \\ & \mathcal H_P(P_{\OO}, P^{\circ}) & \mathcal H_P(P^{\circ}, P^{\circ}) = \mathcal H_P \ar@/^2.0pc/[l] } \end{displaymath} \\ where the curved arrows denote the aforementioned actions, allows us to view the middle column $\mathcal H_G(K, P^{\circ}) \cong \mathcal H_P(P_{\OO}, P^{\circ})$ as a $(\mathcal H_G, \mathcal H_P)$-bimodule: this is proved in fact \ref{bimodaction}.

It turns out that by choosing a specific\footnote{This is the spherical vector.} element $1 \in \mathcal H_P(P_{\OO}, P^{\circ})$, the convolution action $\mathcal H_P(P_{\OO}, P^{\circ}) \curvearrowleft \mathcal H_P$ provides a morphism of (right) $\mathcal H_P$-modules via $A \mapsto 1.A$ which is an isomorphism in degrees $0$ and $1$ and in all degrees if $\mathrm P = \mathrm B$ and $p$ is large enough - see proposition \ref{modisom} for explicit bounds.

This allows us to define an algebra morphism $\mathcal S: \mathcal H_G^{\le 1} \lra \mathcal H_P^{\le 1}$ (respectively, $\mathcal S: \mathcal H_G \lra \mathcal H_P$ if the last part of proposition \ref{modisom} holds) intrinsically via the rule $F.1 = 1. \mathcal S(F)$, as in definition \ref{preliminarysatake}. Composing with the isomorphism given by lemma \ref{dhaisom} yields an algebra morphism $\mathcal H_G^{\le 1} \lra \mathcal H_M^{\le 1}$ (respectively, $\mathcal S: \mathcal H_G \lra \mathcal H_M$ if the last part of proposition \ref{modisom} holds) as in definition \ref{satakedefn} and the final part of this section proves - by obtaining explicit formulas - that this morphism coincides with the Satake map defined in section \ref{secSatakegroupoids}.
\begin{prop} \label{modisom}
Consider the following push-pull diagram of groupoids \begin{displaymath} \xymatrix{ \left( P_F \curvearrowright [P] \times P_F / P^{\circ} \right) & & \left( P_F \curvearrowright P_F / P^{\circ} \times P_F / P^{\circ} \right) = \underline P \\ & \left( P_F \curvearrowright [P] \times P_F / P^{\circ} \times P_F / P^{\circ} \right) \ar[lu]_{i_{1,2}} \ar[ru]^{ i_{2,3}} \ar[d]_{i_{1,3}} \\ & \left( P_F \curvearrowright [P] \times P_F / P^{\circ} \right) } \end{displaymath}

The above diagram yields a well-posed action by convolution of $\mathcal H_P$ on $\mathcal H_P \left( P_{\OO}, P^{\circ} \right)$, giving the latter the structure of a graded right $\mathcal H_P$-module.

In particular, let $1$ be the element of $\mathcal H_P \left( P_{\OO}, P^{\circ} \right)$ supported on the $P_F$-orbit of $\left( P_{\OO}, P^{\circ} \right)$ and such that $1 \left( P_{\OO}, P^{\circ} \right) \in H^0 \left( P_{\OO}, S \right)$ is the element $1$ of $S$.
Then the map \[ \mathcal H_P \lra \mathcal H_P \left( P_{\OO}, P^{\circ} \right) \quad F \mapsto 1.F \] is a degree-preserving morphism of graded $\mathcal H_P$-modules, and when restricted to $\mathcal H^{\le 1}_P$ it is an isomorphism onto its image $\mathcal H_P^{\le 1} (P_{\OO}, P^{\circ})$.

If $\mathrm P = \mathrm B = \mathrm T \ltimes \mathrm U$, $ p \ge h$ and\footnote{Recall that $h$ is the Coxeter number of $\mathrm G$, defined as the maximum of $\langle \rho, \alpha^{\vee} \rangle +1$, where $\rho$ is the half-sum of positive roots and $\alpha$ varies among the positive roots $\Phi^+(\mathrm G, \mathrm T)$.}  either \begin{itemize}
    \item $F = \Qp$ or
    \item $F / \Qp$ is a Galois extension of degree $d = ef$ and\footnote{Recall that $f$ denotes the residue field degree.} $p^f - 1 > 2d(h-1)$,
\end{itemize} then the map \[ \mathcal H_B \lra \mathcal H_B \left( B_{\OO}, B^{\circ} \right) \quad F \mapsto 1.F \] is an isomorphism in every degree.
\end{prop}
\begin{proof}
The first part of the proposition is an application of fact \ref{generalaction}, so we need to check the three conditions in the statement of that result to obtain a well-posed algebra action of $\mathcal H_P$ on $\mathcal H_P(P_{\OO}, P^{\circ})$.

To show condition 3, i.e. that $i_{1,3}$ is a finite covering morphism, it suffices to check that the injection of isotropy groups is an open inclusion of finite index for one object in each connected component of the source groupoid. Let then $(P_{\OO}, m P^{\circ}, n P^{\circ}) \in [P] \times [M]^2$, whose stabilizer is $\mathrm P(\OO) \cap \mathrm M(F)_{m,n} \mathrm V(F) = \mathrm M(\OO)_{m,n} \mathrm V(\OO)$. We have $i_{1,3} \left( (P_{\OO}, m P^{\circ}, n P^{\circ}) \right) = (P_{\OO}, n P^{\circ})$ whose stabilizer is $\mathrm P(\OO) \cap \mathrm M(F)_{n} \mathrm V(F) = \mathrm M(\OO)_n \mathrm V(\OO)$ and hence the index is $[\mathrm M(\OO)_n : \mathrm M(\OO)_{m,n}]$ which is finite as $\mathrm M(\OO) \subset \mathrm M(F)$ is open and compact; moreover $\mathrm M(\OO)_n \mathrm V(\OO)$ is compact, so the inclusion is open.

Next we check condition 1: given $A \in \mathbb H^*_c \left( P_F \curvearrowright [P] \times [M] \right)$ and $B \in \mathcal H_P$, we show that the cohomology class $i_{1,2}^* A \cup i_{2,3}^* B$ is $i_{1,3}$-fiberwise compactly supported. This boils down to the index $[\Stab_{P_F} (xP_{\OO}) : \Stab_{P_F} (x P_{\OO}, m P^{\circ})]$ being finite.
Conjugating by $x^{-1}$ yields that this index is the same as $ [\Stab_{P_F} (P_{\OO}) : \Stab_{P_F} ( P_{\OO}, x^{-1}m P^{\circ})]$, and since $x^{-1} m P^{\circ} = m' P^{\circ}$ for some $m' \in M_F$, we obtain that the index is $[ \mathrm P(\OO) : \mathrm P(\OO) \cap \mathrm M(F)_{m'} \mathrm V(F) ] = [ \mathrm P(\OO) : \mathrm M(\OO)_{m'} \mathrm V(\OO)] = [\mathrm M(\OO) : \mathrm M(\OO)_{m'}]$ which is finite.

Finally, we need to check condition 2: since $V_F$ is normal in $P_F$, we have \[ P_{\OO} m P^{\circ} n P^{\circ} = P_{\OO} m M_{\OO} n P^{\circ} = V_{\OO} \left( M_{\OO} m M_{\OO} n M_{\OO} \right) V_F. \]
Since $M_{\OO} \subset M_F$ is open and compact, we have a finite disjoint union $M_{\OO} m M_{\OO} n M_{\OO} = \bigsqcup_{i=1}^N M_{\OO} m_i M_{\OO}$ and thus \[ P_{\OO} m P^{\circ} n P^{\circ} = V_{\OO} \left( \bigsqcup_{i=1}^N M_{\OO} m_i M_{\OO} \right) V_F = \bigcup_{i=1}^N P_{\OO} m_i P^{\circ}. \]
This completes the proof of the first part of the proposition.
It is then an immediate consequence of the structure of right $\mathcal H_P$-module on $\mathcal H_P ( P_{\OO} , P^{\circ})$ that the map $F \mapsto 1.F$ is a morphism of right $\mathcal H_P$-modules. It is also clearly degree preserving, since $1 \in \mathcal H_P ( P_{\OO}, P^{\circ} )$ is supported in degree 0.

It remains to show that $F \mapsto 1.F$ is bijective in degree at most $1$. Both $\mathcal H_P$ and $\mathcal H_P \left( P_{\OO}, P^{\circ} \right)$ can be broken down into direct sums of cohomology algebras of groups indexed by the same double coset set $M_{\OO} \backslash M_F / M_{\OO}$: more precisely, $ \mathcal H_P \cong \bigoplus_{m \in M_{\OO} \backslash M_F / M_{\OO}} H^* \left( \mathrm M(\OO)_m \mathrm V(F), S \right)$ and $ \mathcal H_P \left( P_{\OO}, P^{\circ} \right) \cong \bigoplus_{m \in M_{\OO} \backslash M_F / M_{\OO}} H^* \left( \mathrm M(\OO)_m \mathrm V(\OO), S \right)$.

By definition of the convolution action, we have \[ (1.F) \left( P_{\OO}, m P^{\circ} \right) = \sum_{xP^{\circ} \in P_F / P^{\circ}} 1 \left( P_{\OO}, x P^{\circ} \right) \cup F \left( x P^{\circ}, m P^{\circ} \right), \] but since $1$ is supported on the $P_F$-orbit of $(P_{\OO}, P^{\circ})$, for a summand to be nonzero we need $x \in P_{\OO} \subset P^{\circ}$ and hence we can pick $x P^{\circ} = P^{\circ}$ and get \[ (1.F) \left( P_{\OO}, m P^{\circ} \right) = 1 \left( P_{\OO}, P^{\circ} \right) \cup F \left( P^{\circ}, m P^{\circ} \right). \]
More precisely, since $\Stab \left( P_{\OO}, P^{\circ} \right) = P_{\OO}$, $\Stab \left( P^{\circ}, mP^{\circ} \right) = \left( M_{\OO} \cap m M_{\OO} m^{-1} \right) V_F$ and $\Stab \left( P_{\OO}, m P^{\circ} \right) = \left( M_{\OO} \cap m M_{\OO} m^{-1} \right) V_{\OO}$ we obtain \[ (1.F) \left( P_{\OO}, m P^{\circ} \right) = \cores{\mathrm M(\OO)_m \mathrm V(\OO)}{\mathrm M(\OO)_m \mathrm V(\OO)} \left( \res{\mathrm P(\OO)}{\mathrm M(\OO)_m \mathrm V(\OO)} 1 \left( P_{\OO}, P^{\circ} \right) \cup \res{\mathrm M(\OO)_m \mathrm V(F)}{\mathrm M(\OO)_m \mathrm V(\OO)} F \left( P^{\circ}, m P^{\circ} \right) \right) \] and since cupping with $1$ does not do anything, even after restriction, \begin{equation} \label{Pisom} (1.F) \left( P_{\OO}, m P^{\circ} \right) = \res{\mathrm M(\OO)_m \mathrm V(F)}{\mathrm M(\OO)_m \mathrm V(\OO)} F \left( P^{\circ}, m P^{\circ} \right). \end{equation}
This explicit formula shows that bijectivity amounts to the restriction map \[ \res{\mathrm M(\OO)_m \mathrm V(F)}{\mathrm M(\OO)_m \mathrm V(\OO)} :H^i \left( \mathrm M(\OO)_m \mathrm V(F) , S \right) \lra H^i \left( \mathrm M(\OO)_m \mathrm V(\OO) , S \right) \] being an isomorphism for all $i=0,1$ and for all $m$ in a choice of representatives for $P_{\OO} \backslash P_F / P^{\circ} \equiv M_{\OO} \backslash M_F / M_{\OO}$. Then choosing $m \in X^M_*(\mathrm T)_-$, the antidominant cocharacters for $\mathrm M$, allows us to show that $\mathrm T(\OO) \subset \mathrm M(\OO)_m$ and then run the same argument as in lemma \ref{dhaisom}, using again that $|k| \ge 5$.

\begin{rem}
As in the lemma \ref{dhaisom}, it is immediate from the proof that the map $F \mapsto 1.F$ is an isomorphism in degree less than $k$, where $k$ is the greatest integer such that the restriction map $H^i \left( \mathrm M(\OO)_m \mathrm V(F) , S \right) \lra H^i \left( \mathrm  M(\OO)_m \mathrm V(\OO) , S \right)$ is an isomorphism for all $m \in M_{\OO} \backslash M_F / M_{\OO}$ and for all $i \le k$.

In the course of the proof of lemma \ref{dhaisom}, we showed that the restriction map \[ \res{\mathrm M(\OO)_m \mathrm V(F)}{\mathrm M(\OO)_m} :H^i \left( \mathrm M(\OO)_m \mathrm V(F) , S \right) \lra H^i \left( \mathrm M(\OO)_m, S \right) \] is an isomorphism for all $m \in M_F$ and for all $i \ge 0$.
Since the restriction maps factor for multiple inclusions, showing that $H^i \left( \mathrm M(\OO)_m \mathrm V(F) , S \right) \stackrel{\res{}{}} {\lra} H^i \left( \mathrm M(\OO)_m \mathrm V(\OO) , S \right)$ is an isomorphism is equivalent to showing that \[ \res{\mathrm M(\OO)_m \mathrm V(\OO)}{\mathrm M(\OO)_m} :H^i \left( \mathrm M(\OO)_m \mathrm V(\OO) , S \right) \lra H^i \left( \mathrm M(\OO)_m, S \right) \] is an isomorphism.

If $S$ is a ring where $p$ is invertible, using the Hochschild-Serre spectral sequence as in the proof of lemma \ref{dhaisom} allows one to show that $H^q \left( V(\OO), S \right) =0$ for all $q \ge 1$, and hence this restriction map is an isomorphism in all degrees, proving that the map $\mathcal H_P \lra \mathcal H_P \left( P_{\OO}, P^{\circ} \right)$ is an isomorphism of right $\mathcal H_P$-modules in all degrees. This observation and the remark after the proof of lemma \ref{dhaisom} are the crucial steps to prove corollary \ref{casepinvertible}.
\end{rem}

Assume now that $\mathrm P = \mathrm B = \mathrm T \ltimes \mathrm U$ and that $p \ge h$. Thanks to the remark, to finish the proof of the proposition it remains to show that the restriction map \[ \res{\mathrm T(\OO) \mathrm U(\OO)}{\mathrm T(\OO)} :H^i \left( \mathrm T(\OO) \mathrm U(\OO) , S \right) \lra H^i \left( \mathrm T(\OO), S \right) \] is an isomorphism in all degrees as long as one of the additional conditions in the statement holds.

To show that the restriction map is an isomorphism, it suffices to check that the inflation map induced by $\mathrm B(\OO) \twoheadrightarrow \mathrm B(\OO) / \mathrm U(\OO) \cong \mathrm T(\OO)$ is also an isomorphism (one can check that the two isomorphism statements are equivalent by working on cochain representatives).

Now, the inflation map is the edge map in the Hochschild-Serre spectral sequence induced by $1 \lra \mathrm U(\OO) \lra \mathrm B(\OO) \lra \mathrm T(\OO) \lra 1$; this spectral sequence reads \[ E_2^{ij} = H^i \left( \mathrm T(\OO), H^j \left( \mathrm U(\OO), S \right) \right) \Rightarrow H^{i+j} \left( \mathrm B(\OO), S \right), \] thus to show that the inflation map is an isomorphism it suffices to prove that $E_2^{ij} = 0$ for all $i \ge 0, j \ge 1$.

Consider the short exact sequence of trivial $\mathrm B(\OO)$-representations: \[ 0 \lra \Zp \stackrel{\cdot p^n} {\lra} \Zp \lra \Z / p^n \lra 0 \] which induces a long exact sequence in $\mathrm U(\OO)$-cohomology from which for every $j \ge 1$ we extract the following short exact sequence: \[ 0 \lra H^j \left( \mathrm U(\OO), \Zp \right) / p^n H^j \left( \mathrm U(\OO), \Zp \right) \lra H^j \left( \mathrm U(\OO), \Z / p^n \right) \lra H^{j+1} \left( \mathrm U(\OO), \Zp \right) [ p^n ] \lra 1. \]
An immediate consequence of fact \ref{kostantcomputations} below is that $H^* \left ( \mathrm U(\OO), \Zp \right)$ is torsion-free.
This torsion-freeness implies that the third term of the last short exact sequence is zero, and we conclude that $H^j \left( \mathrm U(\OO), \Z / p^n \right)$ is the cokernel of the multiplication by $p^n$ on $H^j \left( \mathrm U(\OO), \Zp \right)$: \[ 1 \lra H^j \left( \mathrm U(\OO), \Zp \right) \stackrel{\cdot p^n}{\lra} H^j \left( \mathrm U(\OO), \Zp \right) \lra H^j \left( \mathrm U(\OO), \Z / p^n \right) \lra 1. \]
The above sequence is in fact a short exact sequence of $\mathrm T(\OO)$-representations, so we obtain a long exact sequence in $T(\OO)$-cohomology: \[ \ldots \lra H^i \left( \mathrm T(\OO), H^j \left( \mathrm U(\OO), \Zp \right) \right) \stackrel{\cdot p^n}{\lra} H^i \left( \mathrm T(\OO), H^j \left( \mathrm U(\OO), \Zp \right) \right) \lra \] \[ \lra H^i \left( \mathrm T(\OO), H^j \left( \mathrm U(\OO), \Z / p^n \right) \right)  \lra H^{i+1} \left( \mathrm T(\OO), H^j \left( \mathrm U(\OO), \Zp \right) \right) \stackrel{\cdot p^n} {\lra} \ldots \]
and hence it suffices to show that $H^i \left( \mathrm T(\OO), H^j \left( \mathrm U(\OO), \Zp \right) \right)$ is zero for all $i \ge 0$ and all $j \ge 1$.

We now need the following result, which is an immediate consequence of corollary 2 and theorem 4 from \cite{ronchetti}.
\begin{fact} \label{kostantcomputations} Let $F / \Qp$ be a finite Galois extension. Let $\mathrm G$ be a connected, reductive, split group over $F$. Fix a maximal split torus $\mathrm T$ and a Borel subgroup containing it $\mathrm B = \mathrm T \ltimes \mathrm U$. Assume also that compatible, smooth $\OO_F$-models of all these have been fixed. Let $W$ be the Weyl group.

Then there exists an exhaustive and discrete $\mathrm T(\OO_F)$-equivariant filtration on $ H^j \left( \mathrm U(\OO_F), \Zp \right) $ whose associated graded is a direct sum of modules\footnote{We denote by $V_{\OO_F} (\lambda)$ a free, rank $1$ $\OO_F$-module where $\mathrm T(\OO_F)$ acts via the character $\lambda \in X^*(\mathrm T)$.} $ V_{\OO_F} \left( \sum_{\sigma \in \Gal(F/\Qp)} w_{\sigma} \cdot 0 \right)$ as we vary $w_{\sigma} \in W$ subject to the condition $\sum_{\sigma \in \Gal(F/\Qp)} l(w_{\sigma})=j$.\footnote{These modules come with multiplicity, but for our current purposes that's not important.}
\end{fact}
The $\mathrm T(\OO_F)$-filtration on $H^j \left( \mathrm U(\OO_F), \Zp \right)$ being exhaustive and discrete means that there are integers $N, M$ such that $\mathrm{Fil}^N H^j \left( \mathrm U(\OO_F), \Zp \right) = H^j \left( \mathrm U(\OO_F), \Zp \right)$ and $\mathrm{Fil}^M H^j \left( \mathrm U(\OO_F), \Zp \right) = 0$.
For every $N \le r \le M$, the short exact sequence \[ 0 \lra \mathrm{FIl}^{r+1} H^j \left( \mathrm U(\OO_F), \Zp \right) \lra \mathrm{Fil}^r H^j \left( \mathrm U(\OO_F), \Zp \right) \lra \mathrm{gr}^r H^j \left( \mathrm U(\OO_F), \Zp \right) \lra 0 \] induces a long exact sequence in $\mathrm T(\OO_F)$-cohomology.

By induction on $r$, to prove that $H^i \left( \mathrm T(\OO_F), H^j \left( \mathrm U(\OO_F), \Zp \right) \right) = 0$ it suffices thus to prove that $H^i \left( \mathrm T(\OO_F), \mathrm{gr}^r H^j \left( \mathrm U(\OO_F), \Zp \right) \right) = 0$ for all $i \ge 0$, $j \ge 1$ and $r$.
In light of fact \ref{kostantcomputations}, it suffices thus to show that $H^i \left( \mathrm T(\OO_F), V_{\OO} \left( \sum_{\sigma \in \Gal(F/ \Qp)} w_{\sigma} \cdot 0 \right) \right) = 0$ for fixed choices of $w_{\sigma} \in W$ such that $\sum_{\sigma \in \Gal(F/\Qp)} l(w_{\sigma})=j$.

Fix one such module $V_{\OO} \left( \sum_{\sigma \in \Gal(F/ \Qp)} w_{\sigma} \cdot 0 \right)$. Notice that $w_{\sigma} \cdot 0 = w_{\sigma} \left( 0 + \rho \right) - \rho = w_{\sigma}(\rho) - \rho$.
In particular, \[ \sum_{\sigma \in \Gal(F/ \Qp)} w_{\sigma} \cdot 0 = \sum_{\sigma \in \Gal(F/ \Qp)} \left( w_{\sigma} \rho - \rho \right) = \sum_{\sigma \in \Gal(F/ \Qp)} w_{\sigma}(\rho) - d \rho. \]
Since the Weyl group action on the character group preserves the canonical inner product (see Bourbaki \cite{bourbaki}, chapter VII, section 1, subsection 1, proposition 3), the only way we could have $\sum_{\sigma \in \Gal(F/ \Qp)} w_{\sigma} \cdot 0$ equal to zero is if each $w_{\sigma} (\rho) = \rho$.
The Weyl group acts with trivial stabilizer on the half-sum of positive roots $\rho$, so that would impliy that $w_{\sigma} = \id $ for each $\sigma \in \Gal(F/\Qp)$, but then $\sum l(w_{\sigma}) = 0$, which is not the case.

This shows that $\sum_{\sigma \in \Gal(F/ \Qp)} w_{\sigma} \cdot 0 \in X^* ( \mathrm T )$ is not the zero character. 
\begin{itemize}
    \item Suppose first that $F = \Qp$, so there's only one $w = w_{\sigma}$. For each cocharacter $\mu$, notice that we have \[ \langle \mu, w \rho - \rho \rangle = \langle \mu, w \rho \rangle - \langle \mu, \rho \rangle = \langle w^{-1} \mu, \rho \rangle - \langle \mu, \rho \rangle = \langle (w^{-1} - \id) \mu , \rho \rangle. \]
Since $w$ has finite order $c$, the transformation $(w^{-1} - \id)$ on the lattice of cocharacters $X_* \left( \mathrm T \right)$ is invertible, with inverse $- \sum_{i=0}^{c-1} (w^{-1})^i$. In particular, we pick a simple coroot $\alpha^{\vee}$ and we can find $\mu \in X_*\left( \mathrm T \right)$ such that $w^{-1} \mu - \mu = \alpha^{\vee}$. Fix such a $\mu$.

We obtain then \[ \langle \mu, w \rho - \rho \rangle = \langle (w^{-1} - \id) \mu, \rho \rangle = \langle \alpha^{\vee}, \rho \rangle = 1 \] where the last equality is due to the fact that $\rho$ is the sum of the fundamental weights, and hence pairs to $1$ with any simple coroot.
    \item Suppose instead that $F / \Qp$ is a degree $d$ Galois extension with residue field degree $d$, and assume $p^f-1 > 2d(h-1)$.

By definition of the Coxeter number, we have $\langle \rho, \alpha^{\vee} \rangle \le h-1$ for each positive root $\alpha \in \Phi^+ (\mathrm G ,\mathrm T )$ and also $\langle \rho, \alpha^{\vee} \rangle \ge 0$ for each positive root $\alpha$, since $\rho$ pairs to $1$ with each simple coroot.
On the other hand, if $\beta \in \Phi^- (\mathrm G, \mathrm T)$ is a negative root, then we have $- \left( \beta^{\vee} \right) = \left( - \beta \right)^{\vee}$ (as one can verify by pairing against all roots).
In particular, $ - \langle \rho, \beta^{\vee} \rangle = \langle \rho, - \left( \beta^{\vee} \right) \rangle = \langle \rho, \left( - \beta \right)^{\vee} \rangle $ and thus $0 \ge \langle \rho, \beta^{\vee} \rangle \ge - (h-1)$ for all $\beta \in \Phi^- \left( \mathrm G, \mathrm T \right)$.

We conclude that \begin{equation} \label{estimaterhopairing} \left| \langle \rho, \alpha^{\vee} \rangle \right| \le h-1 \qquad \forall \alpha \in \Phi(\mathrm G, \mathrm T). \end{equation}
Since the Weyl group permutes the coroots, we obtain that \[ \left| \langle w \cdot 0, \alpha^{\vee} \rangle \right| = \left| \langle w\rho - \rho, \alpha^{\vee} \rangle \right| = \left| \langle w \rho, \alpha^{\vee} \rangle - \langle \rho, \alpha^{\vee} \rangle \right| \le \] \[ \le  \left| \langle w \rho, \alpha^{\vee} \rangle \right| + \left| \langle \rho, \alpha^{\vee} \rangle \right| = \left| \langle \rho, w^{-1} \alpha^{\vee} \rangle \right| + \left| \langle \rho, \alpha^{\vee} \rangle \right| \le (h-1) + (h-1) = 2(h-1). \]
We thus obtain the (rough) estimate that \[ \left| \langle \sum_{\sigma} w_{\sigma} \cdot 0, \alpha^{\vee} \rangle \right| \le \sum_{\sigma} \left| \langle w_{\sigma} \cdot 0 , \alpha^{\vee} \rangle \right| \le d \cdot 2 (h-1). \]
So if \begin{equation} \label{boundtrivialcohomology} p^f-1 > 2d(h-1) \end{equation} we obtain that there exists a cocharacter $\mu \in X_* (\mathrm T)$ (in fact, a coroot) such that through the perfect pairing $X^*(\mathrm T) \times X_*(\mathrm T) \lra \Z$ we have $\langle \mu, \sum_{j=1}^d w_j \cdot 0 \rangle \not\in (p^f-1) \Z$. Fix one such $\mu$.
\end{itemize}
The rest of the proof is identical for both cases.

Let $k$ be the residue field of $F$. We fix a splitting $\OO_F^* \cong k^* \times (1 + \varpi \OO_F)$ by using the Teichmuller lift.
Let then $C = \mu (k) \subset T(\OO_F)$; this is a quotient of $k^*$, and hence is cyclic. Fixing a generator $\varepsilon$ of $k^*$ we have that $\mu(\varepsilon)$ generates $C$.

The usual trick with the Hochschild-Serre spectral sequence allows us to reduce to showing that $H^a \left( C, V_{\OO_F} (\sum_{\sigma} w_{\sigma} \cdot 0 ) \right) =0$ for all $a \ge 0$, where now $x \in C$ acts on $\OO_F$ by multiplication by $(\sum_{\sigma} w_{\sigma} \cdot 0 ) (x)$. Since $C$ is cyclic these cohomology groups can be easily computed.

The cohomology groups in even degrees are zero: indeed they are a quotient of the $C$-invariants, which are zero since the generator $\mu(\varepsilon)$ of $C$ acts by multiplication by $\varepsilon^{\langle \mu, \sum_{\sigma} w_{\sigma} \cdot 0 \rangle } \neq 1$, thanks to our assumptions on $\mu$.

To show that the cohomology groups in odd degrees are $0$, it suffices to show that multiplication by $\mu(\varepsilon) - \id$ is invertible on $V_{\OO_F} \left( \sum_{\sigma} w_{\sigma} \cdot 0 \right)$. Since $\mu(\varepsilon)- \id$ acts as multiplication by $\varepsilon^{\langle \mu, \sum_{\sigma} w_{\sigma} \cdot 0 \rangle } -1$, it's enough to prove that $\varepsilon^{\langle \mu, \sum_{\sigma} w_{\sigma} \cdot 0 \rangle } \neq 1 \bmod \varpi \OO_F$, which is again clear thanks to our assumption that ${\langle \mu, \sum_{\sigma} w_{\sigma} \cdot 0 \rangle }$ is not divisible by $p^f-1$. This finishes the proof of proposition \ref{modisom}.
\end{proof}
Recall that the compactly supported cohomology of the groupoid $\left( G_F \curvearrowright [G] \times G/P^{\circ} \right)$ is the space $\mathcal H_G ( K, P^{\circ} )$ of $G$-invariant cohomology classes \[ f: G/K \times G/P^{\circ} \lra \bigoplus H^* \left( \Stab_G (xK, y P^{\circ}) , S \right) \] supported on finitely many orbits and such that $f(xK, yP^{\circ}) \in H^* \left( \Stab_G (xK, y P^{\circ}) , S \right)$.
\begin{fact} \label{homequivPtoG} We have a degree-preserving isomorphism $\mathcal H_G(K, P^{\circ}) \lra \mathcal H_P(P_{\OO}, P^{\circ})$ induced as pullback on cohomology by the homotopy equivalence of topological groupoids \[ \left( P_F \curvearrowright [P] \times [M] \right) \stackrel{i}{\lra} \left( G_F \curvearrowright [G] \times G/ P^{\circ} \right) \] given by inclusion $P_F \hookrightarrow G_F$ at the level of groups, `identity' on the first component $[P] = P_F / P_{\OO} \equiv G / K = [G]$ and injection on the second component $P_F / P^{\circ} \hookrightarrow G_F / P^{\circ}$. \\
Explicitly, the map is realized as \[ \mathcal H_G ( K, P^{\circ} ) \lra \mathcal H_P (P_{\OO}, P^{\circ}) \quad F \mapsto \widetilde F \textnormal{ with } \widetilde F(P_{\OO}, mP^{\circ}) = F (K, m P^{\circ}) \] for all $m \in M_{\OO} \backslash M_F / M_{\OO}$, which are representatives for the $G$-orbits on $G/K \times G/ P^{\circ}$ as well as for the $P_F$-orbits on $P_F / P_{\OO} \times P_F / P^{\circ}$.
\end{fact}
Notice that the map $F \mapsto \widetilde F$ indeed coincide with the pullback $i^* : \mathbb H^*_c \left( G_F \curvearrowright [G] \times G/ P^{\circ} \right) \lra \mathbb H^*_c \left( P_F \curvearrowright [P] \times P_F / P^{\circ} \right)$. In particular, $i$ induces injections on isotropy groups which are in fact bijections, so that no restriction map is needed: we have \[ \Stab_G(K, m P^{\circ}) = K \cap m P^{\circ} m^{-1} = K \cap P_F \cap m P^{\circ} m^{-1} = P_{\OO} \cap m P^{\circ} m^{-1} = \Stab_{\mathrm P(F)}(P_{\OO}, m P^{\circ}), \] showing that $i$ is indeed a homotopy equivalence.
\begin{fact} \label{bimodaction} The above isomorphism allows us then to view $\mathcal H_G(K, P^{\circ}) \cong \mathcal H_P(P_{\OO}, P^{\circ})$ as a bimodule for the algebras $\mathcal H_G$ and $\mathcal H_P$, the former acting by convolution on the left factor of $\mathcal H_G(K, P^{\circ})$ and the latter acting by convolution on the right factor of $\mathcal H_P(P_{\OO}, P^{\circ})$.
\end{fact}
We show that we have indeed a bi-module structure - that is to say, the two convolution actions commute with each other. This resembles strongly fact \ref{dhabimodule} from appendix \ref{appendice}, but is not quite the same since we have the map $i$ inducing an isomorphism in cohomology `in the middle'. We give then all details, even though the proof boils down to a diagram chasing very similar to that of fact \ref{dhabimodule}.
\begin{proof} Let $A \in \mathcal H_G$, $B \in \mathcal H_P$ and $F \in \mathbb H^* \left( G_F \curvearrowright [G] \times G / P^{\circ} \right)$. The convolution $(A.F).B$ is represented by the following diagram: \begin{displaymath} \xymatrix{ A \in \underline G & & \left( G_F \curvearrowright [G] \times G_F / P^{\circ} \right) \ni F   \\ & \left( G_F \curvearrowright [G]^2 \times G_F / P^{\circ} \right) \ar[lu]_{i_{1,2}} \ar[ru]^{ i_{2,3}} \ar[d]_{i_{1,3}} & \\ & \left( G_F \curvearrowright[G] \times G_F/ P^{\circ} \right) & \left( P_F \curvearrowright [P] \times [M] \right) \ar[l]^i &  \left( P_F \curvearrowright [M]^2 \right) \ni B \\ & &  \left( P_F \curvearrowright [P] \times [M]^2 \right) \ar[u]_{i_{1,2}} \ar[ru]^{ i_{2,3}} \ar[d]_{i_{1,3}} &  \\ & & ((A.F).B) \in \left( P_F \curvearrowright [P] \times [M] \right) & } \end{displaymath}
\begin{claim}
The diagram \begin{displaymath} \xymatrix{ \left( G_F \curvearrowright [G]^2 \times G_F/ P^{\circ} \right) \ar[d]_{i_{1,3}}  & & \left( P_F \curvearrowright [P]^2 \times [M]^2 \right) \ar[d]_{i_{1,3,4}} \ar[ll]^{i \circ i_{1,2,3}} \\ \left( G_F \curvearrowright [G] \times G_F / P^{\circ} \right) & & \left( P_F \curvearrowright [P] \times [M]^2 \right) \ar[ll]_{i \circ i_{1,2}} } \end{displaymath} is a pullback square of groupoids. Moreover, it satisfies the assumption of lemma \ref{cohomswitch}: that is to say, $i_{1,3}$ is a finite covering morphism and $i \circ i_{1,2}$ induces inclusion on isotropy groups.
\end{claim}
\begin{proof}[Proof of claim] We have shown above that $i$ induces a bijection between isotropy groups, and since $i_{1,2}: \left( P_F \curvearrowright [P] \times [M]^2 \right) \lra \left( P_F \curvearrowright [P] \times [M] \right)$ is clearly a covering morphism, the composition $i \circ i_{1,2}$ also induces inclusion of isotropy groups. \\
Clearly $i_{1,3}$ is a covering morphism, since $G_F$ acts on both groupoids. To prove that the inclusion of isotropy groups is open and finite index, we take a representative $(p_1 K, p_2 K, P^{\circ})$ for a fixed connected component of $\left( G_F \curvearrowright [G]^2 \times G_F/ P^{\circ} \right)$.
Then $\Stab_{P_F} \left(p_1 K, p_2 K, P^{\circ} \right) = \Ad(p_1) K \cap \Ad(p_2)K \cap P^{\circ} = \Ad(p_1)P_{\OO} \cap \Ad(p_2) P_{\OO} \cap P^{\circ}$, while $\Stab_{P_F} \left( i_{1,3} \left(p_1 K, p_2 K, P^{\circ} \right) \right) = \Stab_{P_F} \left( p_1 K, P^{\circ} \right) = \Ad(p_1) K \cap P^{\circ} = \Ad(p_1) P_{\OO} \cap P^{\circ}$.
Since $P^{\circ}$ is closed inside $P_F$, we obtain that $\Ad(p_1)P_{\OO} \cap P^{\circ}$ is closed inside the compact $\Ad(p_1)P_{\OO}$, hence compact.
But $\Ad(p_2)P_{\OO}$ is open in $P_F$, thus the intersection $\Ad(p_2) P_{\OO} \cap \left( \Ad(p_1)P_{\OO} \cap P^{\circ} \right)$ is open in the compact $\Ad(p_1)P_{\OO} \cap P^{\circ}$, and thus finite index.

It remains to show that the pullback is the given one. The group acting is obviously $P_F$, as for the set we have that \[ \left( [G]^2 \times G_F / P^{\circ} \right) \times_{[G] \times G_F / P^{\circ}} \left( [P] \times [M]^2 \right) = \] \[ = \left\{ \left( (p_1 K, p_2K, g_3 P^{\circ}), (p_1' K, p_2'K, p_3' P^{\circ} ) \right) \, | \, (p_1K, g_3 P^{\circ}) = (p_1'K, p_2' P^{\circ}) \right\} \] which forces $g_3 \in P_F$, and completes the proof.
\end{proof}
Thus lemma \ref{cohomswitch} allows us to replace the original diagram with the following: \begin{displaymath} \xymatrix{ \underline G & & \left( G_F \curvearrowright [G] \times G / P^{\circ} \right)   \\ & \left( G_F \curvearrowright [G]^2 \times G / P^{\circ} \right) \ar[lu]_{i_{1,2}} \ar[ru]^{ i_{2,3}} & \left( P_F \curvearrowright [P]^2 \times  [M]^2 \right) \ar[l]^{i \circ i_{1,2,3}} \ar[d]_{i_{1,2,4}} &  \left( P_F \curvearrowright [M]^2 \right) \\ & &  \left( P_F \curvearrowright [P] \times [M]^2 \right) \ar[ru]^{ i_{2,3}} \ar[d]_{i_{1,3}} &  \\ & & \left( P_F \curvearrowright [P] \times [M] \right) & } \end{displaymath}
Applying lemma \ref{corescupres} to the subdiagram \begin{displaymath} \xymatrix{ \left( P_F \curvearrowright [P]^2 \times [M]^2 \right)  \ar[dr]_{i_{1,2,4}} & & \left( P_F \curvearrowright [M]^2 \right) \\ &   \left( P_F \curvearrowright [P] \times [M]^2 \right) \ar[ru]^{ i_{2,3}} \ar[d]_{i_{1,3}} & \\ & \left( P_F \curvearrowright [P] \times  [M]  \right)  } \end{displaymath} makes it equivalent to \begin{displaymath} \xymatrix{ \left( P_F \curvearrowright [P]^2 \times [M]^2 \right)  \ar[dr]_{i_{1,2,4}} \ar[rr]^{i_{3,4}} & & \left( P_F \curvearrowright [M]^2 \right) \\ &   \left( P_F \curvearrowright [P] \times [M]^2 \right) \ar[d]_{i_{1,3}} & \\ & \left( P_F \curvearrowright [P] \times  [M]  \right)  } \end{displaymath} and replacing this into the last diagram and composing pullbacks we obtain that $(A.F).B$ is given by \begin{displaymath} \xymatrix{ A \in \underline G & F \in \left( G_F \curvearrowright [G] \times G / P^{\circ} \right) & \left( P_F \curvearrowright [M]^2 \right) \ni B \\ & \left( P_F \curvearrowright [P]^2 \times  [M]^2 \right) \ar[ul]^{i_{1,2}} \ar[u]^{i \circ i_{2,3}} \ar[ur]_{i_{3,4}} \ar[d]_{i_{1,4}} &   \\ & ((A.F).B) \in \left( P_F \curvearrowright [P] \times [M]  \right)  &  } \end{displaymath}

Consider now $A.(F.B)$: the relevant diagram is \begin{displaymath} \xymatrix{ F \in \left( G_F \curvearrowright [G] \times G / P^{\circ} \right) & \left( P_F \curvearrowright [P] \times [M] \right) \ar[l]_i & & \left( P_F \curvearrowright [M]^2 \right) \ni B \\ & & \left( P_F \curvearrowright [P] \times [M]^2 \right) \ar[ur]_{i_{2,3}} \ar[ul]_{i_{1,2}} \ar[d]^{i_{1,3}} \\ A \in \underline G & \left( G_F \curvearrowright [G] \times G / P^{\circ} \right) & \left( P_F \curvearrowright [P] \times [M] \right) \ar[l]_i^{\spadesuit} &  \\ & \left( G_F \curvearrowright [G]^2 \times G / P^{\circ} \right) \ar[ul]^{i_{1,2}} \ar[u]_{i_{2,3}} \ar[d]_{i_{1,3}} \\ & \left( G_F \curvearrowright [G] \times G / P^{\circ} \right) & \left( P_F \curvearrowright [P] \times [M] \right) \ar[l]_i } \end{displaymath} where the symbol $\spadesuit$ is to remember that although the pullback in cohomology goes as $i^* : \mathbb H^* \left( G_F \curvearrowright [G] \times G / P^{\circ} \right) \lra  \mathbb H^* \left( P_F \curvearrowright [P] \times [M] \right)$, we use the fact that it is an isomorphism and instead move the cohomology class $F.B$ in the opposite direction.
\begin{claim}
The diagram \begin{displaymath} \xymatrix{ \left( G_F \curvearrowright [G]^2 \times G/ P^{\circ} \right) \ar[d]_{i_{1,3}}  & & \left( P_F \curvearrowright [P]^2 \times [M]  \right) \ar[d]_{i_{1,3}} \ar[ll]^i \\ \left( G_F \curvearrowright [G] \times G / P^{\circ} \right) & & \left( P_F \curvearrowright [P] \times [M]  \right) \ar[ll]_i } \end{displaymath} is a pullback square of groupoids. Moreover, it satisfies the assumption of lemma \ref{cohomswitch}: that is to say, $i_{1,3}$ is a finite covering morphism and $i$ induces inclusion on isotropy groups.
\end{claim}
\begin{proof}[Proof of claim] We have already shown that $i_{1,3}$ is a finite covering morphism, and that $i$ induces bijection on isotropy groups. It remains to show that the pullback is the given one: it is clear that the group acting is $P_F$, and as for the sets we have \[ \left( [G]^2 \times G_F / P^{\circ} \right) \times_{[G] \times G_F / P^{\circ}} \left( [P] \times [M] \right) = \] \[ = \left\{ \left( (p_1 K, p_2 K, gP^{\circ}) , (p_1' P_{\OO}, p_2' P^{\circ}) \right) \, | \, (p_1 K, g P^{\circ}) = (p_1' P_{\OO}, p_2' P^{\circ}) \right\} \] which forces $g \in P_F$ and concludes the proof.
\end{proof}
We apply lemma \ref{cohomswitch} and then distribute and compose pullbacks, to obtain that the diagram for $A.(F.B)$ is equivalent to \begin{displaymath} \xymatrix{ \left( G_F \curvearrowright [G] \times G / P^{\circ} \right) & \left( P_F \curvearrowright [P] \times [M] \right) \ar[l]_i & & \left( P_F \curvearrowright [M]^2 \right) \\ & & \left( P_F \curvearrowright [P] \times [M]^2 \right) \ar[ur]_{i_{2,3}} \ar[ul]_{i_{1,2}} \ar[d]^{i_{1,3}} \\ \underline G & \left( G_F \curvearrowright [G] \times G / P^{\circ} \right) & \left( P_F \curvearrowright [P] \times [M] \right) \ar[l]_i^{\spadesuit} &  \\ & \left( P_F \curvearrowright [P]^2 \times [M] \right) \ar[ul]^{i_{1,2}} \ar[u]_{i \circ i_{2,3}} \ar[d]_{i_{1,3}} \\ & \left( P_F \curvearrowright [P] \times [M] \right) &  } \end{displaymath}
The two occurences of $i$ cancel out since $i^*$ is an isomorphism, and we obtain then \begin{displaymath} \xymatrix{ \left( G_F \curvearrowright [G] \times G / P^{\circ} \right) & \left( P_F \curvearrowright [P] \times [M] \right) \ar[l]_i & & \left( P_F \curvearrowright [M]^2 \right) \\ & & \left( P_F \curvearrowright [P] \times [M]^2 \right) \ar[ur]_{i_{2,3}} \ar[ul]_{i_{1,2}} \ar[d]^{i_{1,3}} \\ \underline G & & \left( P_F \curvearrowright [P] \times [M] \right)  &  \\ & \left( P_F \curvearrowright [P]^2 \times [M] \right) \ar[ul]^{i_{1,2}} \ar[ur]_{i_{2,3}} \ar[d]_{i_{1,3}} \\ & \left( P_F \curvearrowright [P] \times [M] \right) &  } \end{displaymath}
\begin{claim}
The diagram \begin{displaymath} \xymatrix{ \left( P_F \curvearrowright [P]^2 \times [M]^2 \right) \ar[d]_{i_{1,2,4}} \ar[rr]^{i_{2,3,4}}  & & \left( P_F \curvearrowright [P] \times [M]^2  \right) \ar[d]_{i_{1,3}} \\ \left( P_F \curvearrowright [P]^2 \times [M] \right) \ar[rr]_{i_{2,3}} & & \left( P_F \curvearrowright [P] \times [M]  \right) } \end{displaymath} is a pullback square of groupoids. Moreover, it satisfies the assumption of lemma \ref{cohomswitch}: that is to say, $i_{1,3}$ is a finite covering morphism and $i_{2,3}$ induces inclusion on isotropy groups.
\end{claim}
\begin{proof}[Proof of claim] Both $i_{1,3}$ and $i_{2,3}$ are covering morphisms, so in particular $i_{2,3}$ induces inclusion on isotropy groups. We have already seen that $i_{1,3}$ is a finite covering morphism, so it remains to prove that the pullback is the given one. It is obvious that the group acting is $P_F$, and that the set is the given one.
\end{proof}
Applying lemma \ref{cohomswitch} yields then \begin{displaymath} \xymatrix{ \left( G_F \curvearrowright [G] \times G / P^{\circ} \right) & \left( P_F \curvearrowright [P] \times [M] \right) \ar[l]_i & & \left( P_F \curvearrowright [M]^2 \right) \\ \underline G & \left( P_F \curvearrowright [P]^2 \times [M]^2  \right) \ar[d]_{1,2,4} \ar[r]^{i_{2,3,4}} & \left( P_F \curvearrowright [P] \times [M]^2 \right) \ar[ur]_{i_{2,3}} \ar[ul]_{i_{1,2}}  &  \\ & \left( P_F \curvearrowright [P]^2 \times [M] \right) \ar[ul]^{i_{1,2}}  \ar[d]_{i_{1,3}} \\ & \left( P_F \curvearrowright [P] \times [M] \right) &  } \end{displaymath}
Applying now lemma \ref{corescupres} to the subdiagram \begin{displaymath} \xymatrix{ \underline G & & \left( P_F \curvearrowright [P]^2 \times [M]^2  \right) \ar[dl]_{1,2,4} \\ & \left( P_F \curvearrowright [P]^2 \times [M] \right) \ar[ul]^{i_{1,2}}  \ar[d]_{i_{1,3}} \\ & \left( P_F \curvearrowright [P] \times [M] \right)  } \end{displaymath} and composing pullbacks yields the same diagram that $(A.F).B$ was equivalent to, and completes the proof.
\end{proof}
\begin{defn} \label{preliminarysatake}
We define the homomorphism: $\mathcal S^G_P: \mathcal H_G^{\le 1} \lra \mathcal H_P^{\le 1} $ implicitly by the following formula: \[ F.1 = 1. \mathcal S^G_P(F) \quad \forall F \in \mathcal H_G^{\le 1} \] where recall that $1 \in \mathcal H_G(K, P^{\circ}) \cong \mathcal H_P(P_{\OO}, P^{\circ})$ is supported on the orbit of $(K, P^{\circ})$ and takes value $1 \in H^0 (P_{\OO}, S)$ there.

If $p \ge h$ and either $F = \Qp$ or $F / \Qp$ is a Galois extension of degree $d = ef$ with $2d(h-1) < p^f-1$, we define the homomorphism $\mathcal S^G_B: \mathcal H_G \lra \mathcal H_B $ by the same formula.
\end{defn}
This is a degree-preserving morphism of algebras: indeed letting $F_1, F_2 \in \mathcal H_G$ be such that $\deg(F_1 \circ F_2) \le 1$ we have \[ (F_1 \circ F_2).1 = 1. \mathcal S^G_P (F_1 \circ F_2) \] but also \[ (F_1 \circ F_2).1 = F_1. (F_2.1) = F_1. \left( 1. \mathcal S^G_P(F_2) \right) = \left( F_1. 1 \right) . \mathcal S^G_P(F_2) = \left( 1. \mathcal S^G_P(F_1) \right). \mathcal S^G_P(F_2) =1 . \left( \mathcal S^G_P(F_1) \circ \mathcal S^G_P(F_2) \right) \] and finally since by proposition \ref{modisom} the map $G \mapsto 1.G$ is an isomorphism of $\mathcal H_P$-modules between $\mathcal H_P^{\le 1}$ and $\mathcal H_P^{\le 1} \left( P_{\OO}, P^{\circ} \right)$ (or $\mathcal H_B$ and $\mathcal H_B \left( B_{\OO}, B^{\circ} \right)$, for the second part of the definition), we must have \[ \mathcal S^G_P (F_1 \circ F_2) = \mathcal S^G_P(F_1) \circ \mathcal S^G_P(F_2). \]
\begin{defn}[Satake homomorphism] \label{satakedefn}
We define the Satake homomorphism: $\mathcal S^G_M: \mathcal H_G^{\le 1} \lra \mathcal H_M^{\le 1}$ by composing the previous homomorphism and the isomorphism $\mathcal H_P \cong \mathcal H_M$ from lemma \ref{dhaisom}. Explicitly, \[ \mathcal S^G_M F \left( mM_{\OO}, n M_{\OO} \right) = \res{\mathrm P(F)_{m,n}}{\mathrm M(F)_{m,n}} \mathcal S^G_P F \left( m P^{\circ}, n P^{\circ} \right). \]
Again, if $p \ge h$ and either $F = \Qp$ or $F / \Qp$ is a Galois extension of degree $d = ef$ with $2d(h-1) < p^f-1$, we can define $\mathcal S^G_T: \mathcal H_G \lra \mathcal H_T$ in all degrees by the same formula.
\end{defn}
\begin{rem}
Even for a general parabolic subgroup $\mathrm P = \mathrm M \ltimes \mathrm V$, following the remarks after lemma \ref{dhaisom} and proposition \ref{modisom}, it follows that the same maps $\mathcal S^G_P$ and $\mathcal S^G_M$ are degree-preserving algebra homomorphisms on $\mathcal H_G^{\le k}$, where $k$ is the largest integer such that the restriction maps \[ \res{}{}:H^i \left( \mathrm M(\OO)_m \mathrm V(F) , S \right) \lra H^i \left( \mathrm M(\OO)_m \mathrm V(\OO) , S \right) \textnormal{ and } \res{}{}:H^i \left( \mathrm M(\OO)_m \mathrm V(F) , S \right) \lra H^i \left( \mathrm M(\OO)_m , S \right) \] are isomorphisms for all $m \in M_{\OO} \backslash M_F / M_{\OO}$ and for all $i \le k$.
\end{rem}
We want then to get a very explicit formula for $\mathcal S^G_M$. In formula \ref{Pisom} we obtained that $(1.G)(P_{\OO}, m P^{\circ}) = \res{\mathrm M(\OO)_m \mathrm V(F)}{\mathrm M(\OO)_m \mathrm V(\OO)} G(P^{\circ}, m P^{\circ})$ for all $G \in \mathcal H_P$.

Let then $F \in \mathcal H_G$ and consider \[ (F.1) (K, m P^{\circ}) = \sum_{x \in G/K \cong P_F / P_{\OO}} F(K, xK) \cup 1(xK, mP^{\circ}). \]
We now restrict the choice of possible nonzero summands. We can write $x=nv \in M_F \ltimes V_F$, and since $1$ is supported on the $P_F$-orbit of $(P_{\OO}, P^{\circ})$, we need $(nv P_{\OO}, m P^{\circ}) \sim (P_{\OO}, P^{\circ})$. The left pair is in the same orbit of $(P_{\OO}, v^{-1}n^{-1}mP^{\circ}) = (P_{\OO}, n^{-1}mP^{\circ})$ where the last equality holds since $P^{\circ} \supset V_F$.
For the summand to be nonzero we need then $n^{-1}m \in P^{\circ}$ which amounts to saying that $n M_{\OO} = m M_{\OO}$.
We conclude that \[ (F.1) (K, m P^{\circ}) = \sum_{mv \in P_F / P_{\OO}} F(K, mvK) \cup 1(mvK, mP^{\circ}) \] with the usual understanding about taking $\Stab_G(K, mP^{\circ}) = \mathrm M(\OO)_m \mathrm V(\OO)$-orbits in the sum, and using the appropriate restriction and corestriction maps: explicitly \begin{equation} \label{intermediatesatake} (F.1) (K, mP^{\circ}) = \sum_{\mathrm M(\OO)_m \mathrm V(\OO) \backslash \backslash \left\{ mv \in P_F / P_{\OO} \right\} } \cores{\mathrm M(\OO)_m \mathrm V(\OO)}{\mathrm M(\OO)_m \mathrm V(\OO) \cap K_{mv}} \res{K_{mv}}{\mathrm M(\OO)_m \mathrm V(\OO) \cap K_{mv}} F(K,mvK) \end{equation} where as usual we have noted that restricting the element $1$ and then cupping with it has no effect.

The last term is then equal to \[ (1. \mathcal S^G_P F)(P_{\OO}, m P^{\circ}) = \res{\mathrm M(\OO)_m \mathrm V(F)} {\mathrm M(\OO)_m \mathrm V(\OO)} \mathcal S^G_P F (P^{\circ} , m P^{\circ} ). \]
In particular, \[ \mathcal S^G_M F(M_{\OO}, m M_{\OO} ) = \res{\mathrm M(\OO)_m \mathrm V(F)}{\mathrm M(\OO)_m} \mathcal S^G_P F (P^{\circ}, m P^{\circ}) = \res{\mathrm M(\OO)_m \mathrm V(\OO)}{\mathrm M(\OO)_m} \circ \res{\mathrm M(\OO)_m \mathrm V(F)} {\mathrm M(\OO)_m \mathrm V(\OO)} S^G_P F (P^{\circ}, m P^{\circ}) = \] \[ = \res{\mathrm M(\OO)_m \mathrm V(\OO)} {\mathrm M(\OO)_m} (1. \mathcal S^G_P F)(P_{\OO}, m P^{\circ}) = \res{\mathrm M(\OO)_m \mathrm V(\OO)}{\mathrm M(\OO)_m} (F.1) (K,mP^{\circ}) = \] \begin{equation} \label{onemoresatakeformula} = \res{\mathrm M(\OO)_m \mathrm V(\OO)} {\mathrm M(\OO)_m} \sum_{\mathrm M(\OO)_m \mathrm V(\OO) \backslash \backslash \left\{ mv \in P_F / P_{\OO} \right\} } \cores{\mathrm M(\OO)_m \mathrm V(\OO)}{\mathrm M(\OO)_m \mathrm V(\OO) \cap K_{mv}} \res{K_{mv}}{\mathrm M(\OO)_m \mathrm V(\OO) \cap K_{mv}} F(K,mvK). \end{equation}
We want to apply the double coset formula: notice that for each $x \in \mathrm M(\OO)_m \backslash \mathrm M(\OO)_m \mathrm V(\OO) / \mathrm M(\OO)_m \mathrm V(\OO) \cap K_{mv}$ we have that $x \left( \mathrm M(\OO)_m \mathrm V(\OO) \cap K_{mv} \right) x^{-1} = \mathrm M(\OO)_m \mathrm V(\OO) \cap K_{xmv}$.
We obtain that \[ \res{\mathrm M(\OO)_m \mathrm V(\OO)}{\mathrm M(\OO)_m}  \cores{\mathrm M(\OO)_m \mathrm V(\OO)} {\mathrm M(\OO)_m \mathrm V(\OO) \cap K_{mv}} = \sum_{x \in \mathrm M(\OO)_m \backslash \mathrm M(\OO)_m \mathrm V(\OO) / \mathrm M(\OO)_m \mathrm V(\OO) \cap K_{mv}} \cores{\mathrm M(\OO)_m}{\mathrm M(\OO)_m \cap K_{xmv}}  \res{\mathrm M(\OO)_m \mathrm V(\OO) \cap K_{xmv}}{\mathrm M(\OO)_m \cap K_{xmv}} c_x^*. \]
Now clearly each $x$ can be taken in $\mathrm V_{\OO}$, in particular $xm = m \hat x$ for another $\hat x \in V_F$, then we also have $\hat x v = \hat v $: this is saying that $xmv K = m \hat v K$ for a change of variables $v \mapsto \hat v$ in $V_F / V_{\OO}$.

Also, notice that in the formula for $\mathcal S^G_M F$ we are summing over $\mathrm M(\OO)_m \mathrm V(\OO)$-orbits on classes $mv \in P_F / P_{\OO}$: the new inner sum $\sum_{x \in \mathrm M(\OO)_m \backslash \mathrm M(\OO)_m \mathrm V(\OO) / \mathrm M(\OO)_m \mathrm V(\OO) \cap K_{mv}}$ correspond to splitting each $\mathrm M(\OO)_m \mathrm V(\OO)$-orbit into $\mathrm M(\OO)_m$-orbits, and then considering one contribution for each of the latter.
It is clear that this amounts to considering $\mathrm M(\OO)_m$-orbits on $mv \in P_F / P_{\OO}$ to start with, so we get \[ \mathcal S^G_M F(M_{\OO}, m M_{\OO} ) = \sum_{\mathrm M(\OO)_m \backslash \backslash \left\{ m \hat v \in P_F / P_{\OO} \right\} } \cores{\mathrm M(\OO)_m}{\mathrm M(\OO)_m \cap K_{m \hat v}}  \res{\mathrm M(\OO)_m \mathrm V(\OO) \cap K_{m \hat v}}{\mathrm M(\OO)_m \cap K_{m\hat v}} \res{K_{m \hat v}}{\mathrm M(\OO)_m \mathrm V(\OO) \cap K_{m \hat v}} F(K,m \hat vK) \] \begin{equation}\label{satakeform} = \sum_{\mathrm M(\OO)_m \backslash \backslash \left\{ m v \in P_F / P_{\OO} \right\} } \cores{\mathrm M(\OO)_m}{\mathrm M(\OO)_m \cap K_{mv}}  \res{K_{mv}}{\mathrm M(\OO)_m \cap K_{m v}} F(K,mvK) \end{equation}
This will be our formula for the derived Satake homomorphism.

\section{The image of the Satake homomorphism} \label{secSatakeimage}
We finally prove theorem \ref{satakeimagerevisited}: we first show injectivity of the Satake homomorphism by an adaptation of the standard `unipotence' argument, and then we describe the upper bound on the image.

The description of the image is where our work on the Satake homomorphism in the previous sections comes in handy: recall that we want to show that a certain power of $p$ divides $(\mathcal  S^G_T F)(T_{\OO}, \lambda(\varpi) T_{\OO})$ if $\lambda \in X_*(\mathrm T)$ pairs positively with some simple root $\alpha$.
We briefly summarize the argument: first of all, we restrict ourselves to the semisimple rank 1 case by factoring $\mathcal S^G_T$ through the standard Levi factor having root system $\Phi(\mathrm M, \mathrm T) = \left\{ \pm \alpha \right\}$. Then, a careful analysis of the possible summands involved in the explicit formula for $(\mathcal S^M_T F) (T_{\OO}, \lambda T_{\OO})$ and some $\mathrm{SL}_2$-computations yield that each such summand is divisible by the required power of $p$.

From now on, we will sometimes use the notation $\mathcal S^G_T F (\mu) = \mathcal S^G_T F ( T_{\OO}, \mu (\varpi) T_{\OO} )$ whenever $\mu \in X_*(\mathrm T)$ is a cocharacter. We also denote $\varpi^{\mu} = \mu(\varpi) \in \mathrm T(F)$.
\begin{prop}
Let $\lambda \in X_*(\mathrm T)_-$ be an antidominant cocharacter and $f_{\lambda} \in \mathcal H_G$ be supported on $H^*(K_{\lambda}, S)$. Then $ \mathcal S^G_T (f_{\lambda}) ( \mu ) = 0$ unless $\mu \ge \lambda$, and $\mathcal S^G_T(f_{\lambda})(\lambda) = \res{K_{\lambda}}{\mathrm T(\OO)} f_{\lambda}$.
\end{prop}
\begin{proof} By formula \ref{satakeform}, in order for $\mathcal S^G_T(f_{\lambda})(\mu)$ to be nonzero we must have $ \varpi^{\mu} u \in K \varpi^{\lambda}K$ for some $u \in \mathrm U(F)$. Then lemma 3.6(i) in \cite{herzig} shows that we must have $\mu \ge \lambda$. \\
Similarly, lemma 3.6(ii) in \cite{herzig} says that $K \varpi^{\lambda} K \cap \varpi^{\lambda} \mathrm U(F) = \varpi^{\lambda} \mathrm U(\OO)$, so that in particular there is only one summand contributing to $\mathcal S^G_T(f_{\lambda})(\lambda)$, and we can take $u = \id$ to obtain \[ \mathcal S^G_T(f_{\lambda})(\lambda) = \cores{\mathrm T(\OO)}{\mathrm T(\OO) \cap K_{\lambda}} \res{K_{\lambda}}{\mathrm T(\OO) \cap K_{\lambda}} f(K, \varpi^{\lambda} K ) = \res{K_{\lambda}}{\mathrm T(\OO)} f_{\lambda} \] since $K_{\lambda} \supset \mathrm T(\OO)$.
\end{proof}
\begin{cor} \label{satakeinj}
The derived Satake homomorphism is injective, as a map from $\mathcal H^1_G$ to $\mathcal H^1_T$.
\end{cor}
\begin{proof}
The previous proposition is the main ingredient of the usual `unipotence' argument for the injectivity of the Satake homomorphism. Since $\res{\mathrm B(\OO)}{\mathrm T(\OO)}:H^1(\mathrm B(\OO), S) \lra H^1(\mathrm T(\OO), S)$ is an isomorphism, what remains to check is that $\res{K_{\lambda}}{\mathrm B(\OO)}$ is injective: we will prove this by showing that $\mathrm B(\OO) \left[ K_{\lambda}, K_{\lambda} \right] = K_{\lambda}$.\footnote{Here $[K_{\lambda}, K_{\lambda}]$ is the commutator subgroup of $K_{\lambda}$.} Notice that since $[K_{\lambda}, K_{\lambda}] $ is normal in $K_{\lambda}$, this is equivalent to showing that the subgroup $\langle \mathrm B(\OO), [K_{\lambda}, K_{\lambda}] \rangle$ of $K_{\lambda}$ generated by $\mathrm B(\OO)$ and $[K_{\lambda}, K_{\lambda}]$ is the whole $K_{\lambda}$.

Bruhat-Tits theory provides a factorization of $K_{\lambda}$: we follow \cite{SS} (section 1.1) in explaining it. Let $x_0 \in A(T,F) \subset \mathcal B(G,F)$ be the hyperspecial vertex corresponding to the maximal compact $K$: we turn this into the origin of the apartment $A(T,F)$ for $T$ (so that $x_0 = 1 \in X_*(\mathrm T)$ is the trivial cocharacter) and set $\Omega = \{ x_0, \lambda(\varpi).x_0 \} \subset X_*(\mathrm T) \subset X_*(\mathrm T) \otimes \R = A(T,F) \subset \mathcal B(G,F)$, so that we have \[ K_{\lambda} = K \cap \lambda(\varpi) K \lambda(\varpi)^{-1} = \Stab_G(x_0) \cap \lambda(\varpi) \Stab_G(x_0) \lambda(\varpi)^{-1} = \] \[ = \Stab_G(x_0) \cap \Stab_G \left( \lambda(\varpi).x_0 \right) = \Stab_G ( \Omega ). \] 
Define now $f_{\Omega}: \Phi(\mathrm G, \mathrm T) \lra \R$ as \[ f(\alpha) = -  \inf_{x \in \Omega} \alpha(x), \] where $\alpha(x)$ is the action of the root $\alpha$ as an affine function on $A(T,F)$.
In the setup of \cite{SS}, the translation action of $\mathrm T(F)$ on $A(T,F)$ is defined as $g.x = x + \nu(g)$, where $\nu: \mathrm T(F) \lra X_*(\mathrm T) \otimes \R$ is defined implicitly via the perfect pairing $\langle, \rangle: X_*(\mathrm T)_{\R} \times X^*(\mathrm T)_{\R} \lra \R$ as \[ \langle \nu(g), \chi \rangle = - \val_F \left( \chi(g) \right) \qquad \forall \chi \in X^*(\mathrm T). \]
In particular, we obtain that $\nu(\lambda(\varpi)) = - \lambda \in X_*(\mathrm T)_+$ is a dominant cocharacter.
We also have that $\lambda(\varpi).x_0 = x_0 - \lambda$, and hence $\alpha(\lambda(\varpi).x_0) = \alpha(x_0 - \lambda) = \langle - \lambda, \alpha \rangle$, since $\alpha(x_0) = 0$ for all $\alpha \in \Phi(\mathrm G, \mathrm T)$.
By dominance of $-\lambda$, this yields $\alpha(\lambda(\varpi).x_0) \ge 0$ for all $\alpha \in \Phi^+$ and $\alpha(\lambda(\varpi).x_0) \le 0$ for all $\alpha \in \Phi^-$.
We conclude \[ f(\alpha) = - \inf_{x \in \Omega} \alpha(x) = - \inf \left\{ 0, \langle - \lambda, \alpha \rangle \right\} = \left\{ \begin{array}{lc} -0 & \textnormal{ if } \alpha \in \Phi^+  \\ - \langle - \lambda, \alpha \rangle & \textnormal{ if } \alpha \in \Phi^- \end{array} \right. = \left\{ \begin{array}{lc} 0 & \textnormal{ if } \alpha \in \Phi^+  \\ \langle \lambda, \alpha \rangle & \textnormal{ if } \alpha \in \Phi^- \end{array}. \right.  \]
Following through with section 1.1 in \cite{SS}, we obtain that \[ U_{\alpha, f_{\Omega}(\alpha)} = \left\{ \begin{array}{lc} \mathrm U_{\alpha}(\OO) & \textnormal{ if } \alpha \in \Phi^+  \\ \mathrm U_{\alpha} \left( \varpi^{ \langle \lambda, \alpha \rangle} \OO \right) & \textnormal{ if } \alpha \in \Phi^- \end{array} \right. \]
and denote by $U_{\Omega}$ the subgroup of $G$ generated by the $U_{\alpha, f_{\Omega}(\alpha)}$'s as $\alpha$ varies across the roots $\Phi(\mathrm G, \mathrm T)$. \\
The subgroup $N_{\Omega} = \left\{ n \in N_G(T) \, | \, nx=x \, \forall x \in \Omega \right\}$ normalizes $U_{\Omega}$, and so we can consider the subgroup $P_{\Omega} = N_{\Omega}.U_{\Omega}$, which has $U_{\Omega}$ as a normal subgroup. A crucial fact is that $P_{\Omega} = \Stab_G(\Omega)$ - see \cite{SS} page 104.

In particular, each $U_{\alpha, f_{\Omega}(\alpha)}$ is contained in $\Stab_G(\Omega) = K_{\lambda}$, so using the $\mathrm T(\OO)$-action by conjugation (and the fact that clearly $\mathrm T(\OO) \subset K_{\lambda}$) we obtain that $U_{\alpha, f_{\Omega}(\alpha)} \subset [K_{\lambda}, K_{\lambda}]$ for all $\alpha \in \Phi(\mathrm G, \mathrm T)$, so that $U_{\Omega} \subset [K_{\lambda}, K_{\lambda}]$.

It remains to show that $N_{\Omega} \subset \mathrm B(\OO) [K_{\lambda}, K_{\lambda}]$. Notice that this is immediate whenever $\lambda$ is regular, because then $N_{\Omega} = \left\{ n \in K \cap N_G(T) \, | \, n. \lambda = \lambda \right\}$ is simply $\mathrm T(\OO)$, so this concludes the proof of corollary \ref{satakeinj} in the case of regular cocharacter.

Recall that the action of $N_G(T) / T_{\OO}$ on the apartment $A(T,F)$ via affine transformations is defined by matching up the translation action of $X_*(\mathrm T)$ and the Weyl-action of $N_G(T) / T_F$ on $A(T,F)$ under the short exact sequence \[ 1 \lra T_F / T_{\OO} = X_*(\mathrm T) \lra N_G(T) / T_{\OO} \stackrel{\pr}{\lra} W = N_G(T) / T_F \lra 1, \] as explained for example in \cite{landvogt}, section 1 of chapter 1.

Since every $n \in N_{\Omega} \subset N_G(T)$ fixes $x_0 = 1$ and $\lambda$, we obtain that $\pr(n) = w_n \in W$ must fix $\lambda$ under the Weyl group action, that is to say that $w_n \in W(M) = N_G(T) \cap \mathrm M(F)$ where $\mathrm M = Z_{\mathrm G}(\lambda)$ is a standard Levi subgroup (see for instance lemma 2.17 of \cite{steinberg} for this well-known fact). \\
This implies that $N_{\Omega} \subset \mathrm M(F)$: now lemma 2.14(a) and (b) in \cite{steinberg} say that since $\mathrm M = Z_{\mathrm G}(\lambda)$ is connected, it is generated by $\mathrm T$ and the $\mathrm U_{\alpha}$'s such that $\langle \lambda, \alpha \rangle =0$. Upon intersecting its $F$-points with $K$, we get that \[ N_{\Omega} \subset K \cap M_F \cap N_G(T) \subset \Big\langle \mathrm T(\OO), \mathrm U_{\alpha}(\OO) \, | \, \langle \lambda, \alpha \rangle = 0 \Big\rangle. \]
We have $\mathrm T(\OO) \subset \mathrm B(\OO)$ and for each root $\alpha$ such that $\langle \lambda, \alpha \rangle =0$, we have already shown that $\mathrm U_{\alpha}(\OO) \subset [ K_{\lambda}, K_{\lambda} ]$, so this concludes the proof.
\end{proof}

We are finally ready to prove theorem \ref{satakeimagerevisited}, which follows from the next result - we will spend the rest of the section proving it.
\begin{thm} \label{imagesatake} Suppose that the residue field of $F$ has size $|k_F| = p^f \ge 5$. Let $\alpha$ be a simple root, and suppose that $\mu \in X_*(\mathrm T)$ is such that $h= \langle \mu, \alpha \rangle \cdot f \ge 1$. Then $\mathcal S^G_T F (\mu) \equiv 0 \bmod p^h$ for all $F \in \mathcal H_G^1$.

In particular, if our ring of coefficients $S$ is $p^a$-torsion and $\langle \mu, \alpha \rangle \cdot f \ge a$, then $\mathcal S^G_T F (\mu) = 0$ for all $F \in \mathcal H_G^1$.
\end{thm}
\begin{proof}
Let $\mathrm M = Z_{\mathrm G} \left( (\ker \alpha)^0 \right)$ be a standard Levi for $\mathrm G$, with maximal torus $\mathrm T$ and root system $\Phi_M = \left\{ \alpha, - \alpha \right\}$. In particular $\mathrm M$ has semisimple rank $1$ and we want to adapt the $\mathrm{PGL}_2$-computations outlined in the introduction to this setup. Notice that $\mathrm M = \mathrm M^0 = Z_{\mathrm G}((\ker \alpha)^0 )^0 = Z_{\mathrm G}(\ker \alpha)^0$ where the last equality holds since both groups are smooth, connected, and have the same Lie algebra - and one is contained in the other. From now on, we assume $\mathrm M = Z_{\mathrm G}(\ker \alpha)^0$, so that $\ker \alpha$ is central in $\mathrm M$.

By the transitivity results in subsection \ref{sectiontransitivesatake}, we have that \[ \mathcal S^G_T f = \mathcal S^M_T \circ \mathcal S^G_M f, \] hence it suffices to show that $\mathcal S^M_T F (\mu) \equiv 0 \bmod p^h$ for all $F \in \mathcal H^1_M$. \\
For ease of notation, for the rest of the proof we put ourselves in the setting where $\mathrm G = \mathrm M$, with maximal torus $\mathrm T$ and Borel subgroup $\mathrm B \cap \mathrm M$ which we denote $\mathrm B$. The unipotent radical of $\mathrm B$ is $\mathrm U \cap \mathrm M = \mathrm U_{\alpha}$. Similar adjustment to the notation are made for $F$- and $\OO$-points, so that for instance $K = M(\OO)$ in the following.

Notice then that $\mathrm B = \mathrm T \ltimes \mathrm U_{\alpha}$ as $\OO$-group schemes. We also fix a pinning: isomorphism of $\OO$-group schemes $i_{\alpha} : \mathbf G_a \lra \mathrm U_{\alpha}$ and $i_{- \alpha} : \mathbf G_a \lra \mathrm U_{ - \alpha}$ which forms an $\mathrm{SL}_2$-triple with the coroot $\alpha^{\vee}: \mathbf G_m \lra \mathrm T$, in the sense of \cite{BT2}, section 3.2.1.

By linearity of the Satake map, it suffices to show that $\mathcal S^G_T F (\mu) \equiv 0 \bmod p^h$ for $F$ in a basis of $\mathcal H^1_G$: assume then that $F$ is supported on the $G$-orbit of $(K, \lambda K)$ for $\lambda$ an antidominant cocharacter for $\mathrm G$. 

By lemma \ref{dhaisom}, we have that $\mathcal H_T^{\le 1} \cong \mathcal H_B^{\le 1}$, hence it suffices to show that $\mathcal S^G_B F (\mu) \equiv 0 \bmod p^h$. As $\res{B^{\circ}}{\mathrm B(\OO)} :H^1 (B^{\circ}, S) \lra H^1 (\mathrm B(\OO), S)$ is an isomorphism, by formula \ref{intermediatesatake} it suffices to show that $\res{B^{\circ}}{\mathrm B(\OO)} \mathcal S^G_B F(\mu) = 0$.
We obtain \begin{equation} \label{borelsatakeformula} \res{B^{\circ}}{\mathrm B(\OO)}  \mathcal S^G_B F (B^{\circ}, \mu(\varpi) B^{\circ}) = \sum_{\mathrm B(\OO) \backslash \backslash \left\{ \varpi^{\mu} u \in B_F / B_{\OO} \right\} } \cores{\mathrm B(\OO)}{\mathrm B(\OO) \cap K_{\varpi^{\mu}u}} \res{K_{\varpi^{\mu}u}}{\mathrm B(\OO)\cap K_{\varpi^{\mu}u}} F(K, \varpi^{\mu}u K ), \end{equation} so it suffices to show that the right hand side is zero $\bmod p^h$. Notice that $\mathrm B(\OO) \cap K_{\varpi^{\mu}u} = \mathrm B(\OO) \cap \Ad(\varpi^{\mu}u) K = \mathrm B(\OO) \cap \Ad(\varpi^{\mu}u) \mathrm B(\OO)$, and from now on we will denote this by $\mathrm B(\OO)_{\varpi^{\mu}u}$.

We describe more explicitly the possible $u \in U_F= \mathrm U_{\alpha}(F)$ appearing in the sum: since $\varpi^{\mu} u$ is a coset representative for $B_F / B_{\OO}$, multiplying by $\mathrm U_{\alpha}(\OO)$ on the right shows that we can assume that $u \in \mathrm U_{\alpha}(F) / \mathrm U_{\alpha}(\OO)$.
On the other hand, we consider $\mathrm B(\OO)$-orbits (under left multiplication) on the cosets $\varpi^{\mu} u \mathrm B(\OO)$, hence multiplying on the left by $\mathrm U_{\alpha}(\OO)$, and using that $\mathrm U_{\alpha}(\OO) \varpi^{\mu} = \varpi^{\mu} \mathrm U_{\alpha}(\varpi^{- \langle \mu, \alpha \rangle} \OO)$, we can assume that $u \in \mathrm U_{\alpha}(\varpi^{- \langle \mu, \alpha \rangle} \OO) \backslash \mathrm U_{\alpha}(F)$.

By assumption $\langle \mu, \alpha \rangle \ge 1$, and since $\mathrm U_{\alpha}$ is abelian, this second condition means we can forget the first one we found, and we end up with $u \in \mathrm U_{\alpha}(F) / \mathrm U_{\alpha}(\varpi^{- \langle \mu, \alpha \rangle} \OO)$.

Fix now one $u \in \mathrm U_{\alpha}(F)$ appearing in the convolution sum in formula \ref{borelsatakeformula}, we aim to show that its contribution to the sum is divisible by $p^h$.

We want to obtain some necessary conditions for an element $y = t v \in \mathrm T(\OO) \ltimes \mathrm U_{\alpha} (\OO) = \mathrm B(\OO)$ to belong to $\mathrm B(\OO)_{\varpi^{\mu}u}$. We have that $y \in \mathrm B(\OO)_{\varpi^{\mu}u}$ if and only if $\Ad\left( \varpi^{\mu} u \right)^{-1} y \in \mathrm B(\OO)$: explicitly \[ u^{-1} \varpi^{ - \mu} t v \varpi^{\mu} u = t \underbrace{\left( t^{-1} u^{-1} t \right)}_{\mathrm U_{\alpha}(F)} \underbrace{\left( \varpi^{- \mu} v \varpi^{\mu} \right)}_{\mathrm U_{\alpha}(F)}  u  \]
As conjugation by $\mathrm T$ preserves $\mathrm U_{\alpha}$, the above element is in $\mathrm B(\OO) = \mathrm T(\OO) \ltimes \mathrm U_{\alpha}(\OO)$ if and only if \[ \left( t^{-1} u^{-1} t \right) \left( \varpi^{- \mu} v \varpi^{\mu} \right) u \in \mathrm U_{\alpha}(\OO). \]
Using implicitly the isomorphism $i_{\alpha}(\OO): \OO \cong \mathrm U_{\alpha}(\OO)$ and passing to additive notation, the above is equivalent to \begin{equation} \label{neccond} u \left( 1 - \alpha(t^{-1}) \right) + \varpi^{ - \langle \mu, \alpha \rangle } v \in \OO \end{equation} which is equivalent to \begin{equation} \label{neccond2} \varpi^{\langle \mu, \alpha \rangle} u \left( 1 - \alpha(t^{-1}) \right) + v \in \varpi^{\langle \mu, \alpha \rangle} \OO. \end{equation}
Since $v \in \OO \cong \mathrm U_{\alpha}(\OO)$, we must have $\varpi^{\langle \mu, \alpha \rangle} u \left( 1 - \alpha(t^{-1}) \right) \in \OO$ as well.

Denote now $B_{\alpha, \mu} = \mathrm T(\OO) \ltimes \mathrm U_{\alpha} (\varpi^{\langle \mu, \alpha \rangle} \OO) $.
\begin{lem} \label{easycorzero} $\cores{\mathrm B(\OO)}{ B_{\alpha, \mu}} : H^1 \left( B_{\alpha, \mu} , S \right) \lra H^1 \left( \mathrm B(\OO), S \right)$ is multiplication by $p^h$.
\end{lem}
\begin{proof}
The index is \[ [ \mathrm B(\OO) : B_{\alpha, \mu}] = [\mathrm U (\OO) : \mathrm U(\varpi^{\langle \mu, \alpha \rangle} \OO)] = [\OO : \varpi^{\langle \mu, \alpha \rangle} \OO ] = |k_F|^{\langle \mu, \alpha \rangle} = p^{f \cdot \langle \mu , \alpha \rangle} = p^h \] and on the other hand the restriction map $\res{}{}: H^1 \left( \mathrm B(\OO), S \right) \lra H^1 \left( B_{\alpha, \mu} , S \right)$ is an isomorphism in degree 1, since the derived subgroup of $B_{\alpha, \mu}$ contains the unipotent part $\mathrm U_{\alpha}(\varpi^{\langle \mu, \alpha \rangle} \OO)$, by using the action of the torus as usual. \\
Therefore, the corestriction map $\cores{\mathrm B(\OO)}{ B_{\alpha, \mu}} : H^1 \left( B_{\alpha, \mu} , S \right) \lra H^1 \left( \mathrm B(\OO), S \right)$ is simply multiplication by the index $p^h$.
\end{proof}
We now consider a special case for the possible summand $\varpi^{\mu}u$: suppose first that $u \in \mathrm U(F)$ is in fact in $ \mathrm U(\varpi^{- \langle \mu, \alpha \rangle} \OO)$. Then as a representative for the $\mathrm B(\OO)$-orbit of $\varpi^{\mu} u \mathrm B(\OO)$ we can take $u = \id$, which corresponds via $i_{\alpha}$ to $0$.

Equation \ref{neccond2} yields then $v \in \mathrm U_{\alpha}(\varpi^{ \langle \mu, \alpha \rangle} \OO)$, and since this holds for all $y \in \mathrm B(\OO)_{\varpi^{\mu}u}$, we have $\mathrm B(\OO)_{\varpi^{\mu}u} = \mathrm T(\OO) \ltimes \mathrm U_{\alpha} ( \varpi^{\langle \mu, \alpha \rangle} \OO) = B_{\alpha, \mu}$.

The lemma above showed that $\cores{\mathrm B(\OO)}{B_{\alpha, \mu}}$ is multiplication by $p^h$ - which proves that this summand yields a contribution to $\mathcal S^G_T F (\mu)$ that is zero $\bmod p^h$.

From now on, we assume that $u = u_{\alpha} \not\in \mathrm U_{\alpha} (\varpi^{- \langle \mu, \alpha \rangle} \OO)$ - or equivalently in additive notation that $\varpi^{\langle \mu, \alpha \rangle} u_{\alpha} \not\in \OO$. Fix then $-n = \val \left( i_{\alpha}^{-1}(u) \right) < - \langle \mu, \alpha \rangle$ for some $n > \langle \mu , \alpha \rangle \ge 1$, i.e. $n \ge 2$.

\begin{claim}\label{littleclaim} Suppose that $ - n = \val(i_{\alpha}^{-1}(u)) < - \langle \mu, \alpha \rangle$, then for each $v = i_{\alpha}(z) \in \mathrm U_{\alpha}(\OO)$ we can find $t \in \mathrm T(\OO)$ such that $tv = y$ satisfies equation \ref{neccond2} and hence $tv \in \mathrm B(\OO)_{\varpi^{\mu}u}$.
\end{claim}
\begin{proof}[Proof of claim] Let $u = i_{\alpha}(x)$. It suffices to find a solution $t \in \mathrm T(\OO)$ to the equation \[ \varpi^{\langle \mu, \alpha \rangle} x \left( 1 - \alpha(t^{-1}) \right) + z =0 \] which is equivalent to \[ \alpha(t^{-1}) = 1 + z x^{-1} \varpi^{ - \langle \mu, \alpha \rangle}. \]
As $z \in \OO$ and $\val(x^{-1}) = n > \langle \mu, \alpha \rangle$ by assumption, it suffices to show that $\alpha: \mathrm T(\OO) \lra \OO^*$ surjects onto $1 +\varpi \OO$.

For the coroot $\alpha^{\vee}: \mathbf G_m \lra \mathrm T$ we have $\alpha^{\vee}(\OO): \OO^* \lra \mathrm T(\OO)$ and moreover the composition $\OO^* \stackrel{\alpha^{\vee}}{\lra} \mathrm T(\OO) \stackrel{\alpha}{\lra} \OO^*$ is just squaring, so it suffices to show that the squaring homomorphism $\OO^* \lra \OO^*$, $a \mapsto a^2$ surjects onto $1 + \varpi \OO$.

Notice that the squaring morphism respects the canonical filtration $1 + \varpi^k \OO$ of $\OO^*$, and that for $k \ge 1$ on each quotient $ \left( 1 + \varpi^k \OO \right) / \left( 1 + \varpi^{k+1} \OO \right)$ it induces an isomorphism, because each such quotient is a finite $p$-group with $p \neq 2$. Then lemma 2 in chapter V of \cite{serre} proves that the squaring morphism is an isomorphism from $1 + \varpi \OO$ to itself, and hence finishes the proof of this claim.
\end{proof}
We can then assume that for our fixed $u = i_{\alpha}(x)$, the mapping $\mathrm B(\OO)_{\varpi^{\mu}u} \lra \mathrm U_{\alpha}(\OO)$ sending $tu \mapsto u$ is surjective.
In particular, we have that $\mathrm T(\OO) \mathrm B(\OO)_{\varpi^{\mu}u} = \mathrm B(\OO)$, hence applying the isomorphism $\res{\mathrm B(\OO)}{\mathrm T(\OO)}: H^1(\mathrm B(\OO), S) \lra H^1(\mathrm T(\OO), S)$ followed by the restriction/corestriction formula we obtain that \[ \res{\mathrm B(\OO)}{\mathrm T(\OO)} \cores{\mathrm B(\OO)}{\mathrm B(\OO)_{\varpi^{\mu} u}} \res{K_{\varpi^{\mu}u}}{\mathrm B(\OO)_{\varpi^{\mu}u}} F(K, \varpi^{\mu}u K ) = \cores{\mathrm T(\OO)}{\mathrm T(\OO) \cap \mathrm B(\OO)_{\varpi^{\mu}u}} \res{K_{\varpi^{\mu}u}}{\mathrm T(\OO) \cap \mathrm B(\OO)_{\varpi^{\mu}u}} F(K, \varpi^{\mu} u K ) \]
We notice that \[ \mathrm T(\OO) \cap \mathrm B(\OO)_{\varpi^{\mu}u} = \left\{ t \cdot 1 \, | \, u (1 - \alpha(t^{-1})) + \varpi^{ - \langle \mu, \alpha \rangle} \cdot 0 \in \OO \right\} = \left\{ \mathrm t \in T(\OO) \, | \, \alpha(t) -1 \in u^{-1} \OO = \varpi^n \OO \right\} \]
Denote this last subgroup by $T_{\alpha, n}$ and notice that $\mathrm T(\OO) \supset T_{\alpha, n} \supset Z(\mathrm G) \subseteq \ker \alpha$.

Recall that $F$ is supported on the double coset $K \lambda(\varpi) K$. Now lemma 3.6i) in \cite{herzig} says that $F(K, \varpi^{\mu} u K )$ is only nonzero if $\mu \ge_{\R} \lambda$, which means that $(\mu - \lambda) \in \R_{\ge 0} \alpha^{\vee}$ is a non-negative real linear combination of simple coroots. 
\begin{claim} Recall that $u  = i_{\alpha}(x)$ so that $n = - \val(x)$. Letting \[ k = \Ad(\varpi^{\mu}) i_{-\alpha}(x^{-1}) = \mu(\varpi) i_{-\alpha} \left( x^{-1} \right) \mu(\varpi)^{-1}, \] we have $k \in \mathrm U_{- \alpha} (\OO)$ and $\varpi^{\mu} u \in k \varpi^{\lambda} K$. In particular we also obtain $\mu - \lambda = n \alpha^{\vee}$.
\end{claim}
\begin{proof}
First of all, notice that by assumption $\val(x^{-1}) = - \val(x) = n > \langle \mu, \alpha \rangle$ so that $\val (x^{-1}) + \langle \mu, - \alpha \rangle > 0$, and hence \[ k= \Ad (\varpi^{\mu}) i_{- \alpha}(x^{-1}) = i_{-\alpha} \left( \varpi^{\langle \mu, - \alpha \rangle} x^{-1} \right) \in i_{-\alpha} (\OO) = \mathrm U_{- \alpha}(\OO) \subset K. \]
To show the rest of the claim, we prove that $\varpi^{\mu} u \in k \left( \mu - n \alpha^{\vee} \right)(\varpi) K$. 
Notice that $\langle \mu - n \alpha^{\vee} , \alpha \rangle = \langle \mu, \alpha \rangle - 2n < n - 2n = -n \le -2$; in particular $\mu - n \alpha^{\vee}$ is an antidominant weight, and hence by uniqueness of the Cartan decomposition for $\varpi^{\mu}u$ it must coincide with $\lambda$, proving the last statement of the claim.

We only need to prove that $\left( \varpi^{\mu} u \right)^{-1} \cdot k \left(\mu - n \alpha^{\vee} \right) (\varpi) \in K$. That is the same as \[ u^{-1} i_{ - \alpha}(x^{-1}) \varpi^{- \mu} \left(\mu - n \alpha^{\vee} \right) (\varpi) = i_{\alpha}(-x) i_{- \alpha}(x^{-1}) \left( n \alpha^{\vee} \right) (\varpi^{-1}) = i_{\alpha}(-x) i_{- \alpha}(x^{-1}) \alpha^{\vee} (\varpi^{-n}). \]
Recall that the coroot $\alpha^{\vee}$ is defined `in the same way' both for $\mathrm{SL}_2$ and for $\mathrm{PGL}_2$ as $t \mapsto \smat{t}{0}{0}{t^{-1}}$. We can compute the above element (using the previously fixed pinning) inside our $\mathrm{SL}_2$-triple, and we obtain \[ i_{\alpha}(-x) i_{- \alpha}(x^{-1}) \alpha^{\vee} (\varpi^{-n}) = \smat{1}{-x}{0}{1} \smat{1}{0}{x^{-1}}{1} \smat{\varpi^{-n}}{0}{0}{\varpi^n} = \smat{0}{-x}{x^{-1}}{1} \smat{\varpi^{-n}}{0}{0}{\varpi^n} = \smat{0}{-x \varpi^n}{x^{-1}\varpi^{-n}}{\varpi^n} \] and as $n = \val(x^{-1}) \ge 2 $, the last element is indeed in $\mathrm{SL}_2(\OO)$ or $\mathrm{PGL}_2(\OO)$ depending on the case.
\end{proof}
We have then $F(K, \varpi^{\mu}u K) = F(K, k \varpi^{\lambda} K) = (c_k^*)^{-1} F(K, \lambda K)$. We now analyze the effect of conjugation by $k$ and restriction to $T_{\alpha,n}$, to finish the proof of the theorem.

Recall that we are studying the following summand:
\[ \cores{\mathrm T(\OO)}{T_{\alpha,n}} \res{K_{\varpi^{\mu}u}}{T_{\alpha,n}} F(K, \varpi^{\mu} u K ) = \cores{\mathrm T(\OO)}{T_{\alpha,n}} \res{K_{\varpi^{\mu}u}}{T_{\alpha,n}} F(K, k \varpi^{\lambda} K ) = \cores{\mathrm T(\OO)} {T_{\alpha,n}} \res{K_{\varpi^{\mu}u}}{T_{\alpha,n}} c_{k^{-1}}^* F(K, \varpi^{\lambda} K ). \]
We focus on $ \res{K_{\varpi^{\mu}u}}{T_{\alpha,n}} c_{k^{-1}}^* F(K, \varpi^{\lambda} K ) \in H^1 \left( T_{\alpha, n } , S \right)$. Denote $f = F(K , \lambda K) \in H^1( K_{\lambda}, S)$.
\begin{claim} The subgroup $\alpha^{\vee}(\OO^*) \cdot (\ker \alpha) (\OO)$ of $\mathrm T(\OO)$ has index coprime to $p$.
\end{claim}
\begin{proof}[Proof of claim]
Consider the exact sequence \[ 1 \lra (\ker \alpha)(\OO) \hookrightarrow \mathrm T(\OO) \stackrel{\alpha}{\lra} \OO^* \lra 1. \] Then we have for the quotient \[ \mathrm T(\OO) / \alpha^{\vee}(\OO) \cdot (\ker \alpha)(\OO) \cong \im (\alpha) / \alpha(\alpha^{\vee}(\OO)) \subset \OO^* / (\OO^*)^2. \] So it suffices to show that the cokernel of the squaring morphism $\OO^* \lra \OO^*$, $a \mapsto a^2$ has size coprime to $p$. \\
In the proof of claim \ref{littleclaim} we have shown that the squaring morphism is an isomorphism when restricted to the pro-$p$ Sylow $1+ \varpi \OO$ of $\OO^*$, so the size of the cokernel is coprime to $p$.
\end{proof}
We show how this claim finishes the proof of theorem \ref{imagesatake}. As the index is coprime to $p$, the behavior of a morphism $f : \mathrm T(\OO) \lra S$ is completely determined by its restriction to $\alpha^{\vee}(\OO^*) \cdot (\ker \alpha) (\OO)$.
Let then $t = \alpha^{\vee}(z) t_0 \in \alpha^{\vee}(\OO^*) \cdot (\ker \alpha) (\OO)$. Since $(\ker \alpha)$ is central in $\mathrm G$ by construction, conjugation by $k^{-1}$ acts trivially on it while on the other hand \[ \Ad(k^{-1}) \left( \alpha^{\vee} (z) \right) = \smat{1}{0}{- \varpi^{- \langle \mu, \alpha \rangle} x^{-1}}{1} \smat{z}{0}{0}{z^{-1}} \smat{1}{0}{ \varpi^{- \langle \mu, \alpha \rangle} x^{-1}}{1} = \] \[ = \smat{z}{0}{- \varpi^{- \langle \mu, \alpha \rangle} x^{-1} z}{z^{-1}} \smat{1}{0}{ \varpi^{- \langle \mu, \alpha \rangle} x^{-1}}{1} = \smat{z}{0}{ \varpi^{- \langle \mu, \alpha \rangle} x^{-1} \left( z^{-1} - z \right)}{z^{-1}} \] so that we conclude \[ \Ad(k^{-1}) t = \Ad(k^{-1}) (\alpha^{\vee} (z) t_0) = \Ad(k^{-1}) \left( \alpha^{\vee}(z) \right) \cdot t_0 = \smat{z}{0}{\varpi^{ - \langle \mu, \alpha \rangle} x^{-1} (z^{-1} - z)}{z^{-1}} \cdot t_0, \]
We notice that this latter element is not just on $K_{\lambda}$, but also in $K_{\lambda} \cap \mathrm B^-(\OO) = \mathrm U_{-\alpha}(\varpi^{ - \langle \lambda, \alpha \rangle} \OO) \rtimes \mathrm T(\OO) $.
Therefore, when computing \[ \res{K_{\varpi^{\mu}u}}{T_{\alpha,n}} c_{k^{-1}}^* F(K, \varpi^{\lambda} K ) (t) = \res{K_{\varpi^{\mu}u}}{T_{\alpha,n}} (c_{k^{-1}}^* f)(t) = (c_{k^{-1}}^* f)|_{T_{\alpha, n}} (t) = f (k^{-1} t k) \] we can replace $f$ by its image under the restriction map $\res{K_{\lambda}}{K_{\lambda} \cap \mathrm B^-(\OO)} : H^1 ( K_{\lambda}, S) \lra H^1(K_{\lambda} \cap \mathrm B^-(\OO), S)$. 
Now since $\res{\mathrm U_{-\alpha}(\varpi^{ - \langle \lambda, \alpha \rangle} \OO) \rtimes \mathrm T(\OO)}{\mathrm T(\OO)} : H^1 (\mathrm U_{-\alpha}(\varpi^{ - \langle \lambda, \alpha \rangle} \OO) \rtimes \mathrm T(\OO) , S ) \lra H^1(\mathrm T(\OO), S)$ is also an isomorphism, we can assume $f$ is just the \emph{unique} extension to $\mathrm U_{-\alpha}(\varpi^{ - \langle \lambda, \alpha \rangle} \OO) \rtimes \mathrm T(\OO)$ of a homomorphism $\mathrm T(\OO) \lra S$.
This unique extension is obtained via the quotient map $\mathrm U_{-\alpha}(\varpi^{ - \langle \lambda, \alpha \rangle} \OO) \rtimes \mathrm T(\OO) \twoheadrightarrow \mathrm T(\OO)$ - in other words, $f(ut) = f(t)$ if $ut \in \mathrm U_{- \alpha}(\varpi^{- \langle \lambda, \alpha \rangle} \OO) \rtimes \mathrm T(\OO)$.

Notice that \[ \smat{z}{0}{ \varpi^{- \langle \mu, \alpha \rangle} x^{-1} \left( z^{-1} - z \right)}{z^{-1}} = \underbrace{\smat{1}{0}{ \varpi^{- \langle \mu, \alpha \rangle} x^{-1} \left( z^{-2} - 1 \right)}{1}}_{\mathrm U_{- \alpha}(\varpi^{- \langle \lambda, \alpha \rangle} \OO)} \cdot \smat{z}{0}{0}{z^{-1}} \] where the unipotent element belongs to the particular filtration subgroup $\mathrm U_{- \alpha}(\varpi^{- \langle \lambda, \alpha \rangle} \OO)$ since $t = \alpha^{\vee}(z) t_0 \in T_{\alpha,n}$ implies that $z^2 = \alpha(t) \in 1 + \varpi^n \OO$ and hence $z^{-2} -1 \in \varpi^n \OO$.

We conclude that \[ f(k^{-1}tk) = f (\Ad k^{-1}(\alpha^{\vee}(z)) \cdot t_0) = f \left( \smat{z}{0}{\varpi^{ - \langle \mu, \alpha \rangle} x^{-1} (z^{-1} - z)}{z^{-1}} \cdot t_0 \right) = \] \[ = f \left( \smat{1}{0}{ \varpi^{- \langle \mu, \alpha \rangle} x^{-1} \left( z^{-2} - 1 \right)}{1} \cdot \alpha^{\vee}(z) \cdot t_0 \right) = f (\alpha^{\vee}(z) \cdot t_0 ) = \left( \res{T(\OO)}{T_{\alpha,n}} f \right) (t) \] since $\alpha^{\vee}(z) \cdot t_0 = t \in T_{\alpha ,n}$.

Applying the corestriction map $\cores{\mathrm T(\OO)}{T_{\alpha,n}}$ to $\res{\mathrm T(\OO)}{T_{\alpha,n}} f$ yields then $[ \mathrm T(\OO) : T_{\alpha,n}] \cdot f$. Recall that $T_{\alpha,n} = \left\{ t \in \mathrm T(\OO) \, | \, \alpha(t) \in 1 + \varpi^n \OO \right\}$ is the kernel of the composition $\alpha_n: \mathrm T(\OO) \stackrel{\alpha}{\lra} \OO^* \twoheadrightarrow \OO^* / (1 + \varpi^n \OO)$, hence the index $[\mathrm T(\OO): T_{\alpha,n}]$ is equal to $ | \im (\alpha_n) |$. We proved before that $\alpha$ surjects onto $1+ \varpi \OO$, hence the size of the image is divisible by $ \left| (1 + \varpi \OO) / (1 + \varpi^n \OO) \right| = |k_F|^{n-1} = p^{f \cdot (n-1)}$.

Our standing assumption that $n > \langle \mu, \alpha \rangle$ means that $f \cdot (n-1) \ge f \cdot \langle \mu, \alpha \rangle =h$, where the last equality is the assumption on $\mu$ in the statement of the theorem. Therefore, multiplication by $[\mathrm T(\OO) : T_{\alpha,n}]$ lands into the $p^h$-divisible subgroup of the cohomology with $S$-coefficients and this completes the proof of the theorem.
\end{proof}

\appendix
\section{Appendix on groupoids} \label{appendice}
We collect here the proofs of several facts on groupoids and groupoid cohomology used in the main body of the paper.
\begin{prop}[Well-posedness of pullback] \label{pullbackwellposed} Let $i: G \lra H$ be a continuous morphism of topological groupoids. Then $i^*F$ satisfies conditions 1 and 2 of definition \ref{groupoidcoh}. Moreover, if $\pi_0 i: \pi_0 G \lra \pi_0 H$ has finite fibers and $F \in \mathbb H^*_c(H)$, then $i^*F \in \mathbb H^*_c(G)$.
\end{prop}
\begin{proof}
That $i^*F$ satisfies condition 1 of definition \ref{groupoidcoh} is obvious, so we focus on condition 2.
Fix then $x ,y \in \Ob(G)$ and $\phi \in \Hom_G(x,y)$. As $i$ is a natural transformation, this induces $i (\phi) \in \Hom_H (i(x), i(y))$ and thus an isomorphism $\Stab_H(i(x)) \stackrel{i(\phi)_x}{\lra} \Stab_H(i(y))$ mapping $h \mapsto i(\phi) \circ h \circ i(\phi)^{-1}$. This yields a map in cohomology in the opposite direction: \[ (i(\phi))^* : H^* \left( \Stab_H(i(y)) ,S \right) \lra H^* \left( \Stab_H(i(x)) , S \right) \] and since $F \in \mathbb H^*(H)$ we have \[ F(i(x)) = (i(\phi))^* F(i(y)). \]
Moreover, the functorial properties of $i$ give group homomorphisms \[ \Stab_G(x) = \Hom_G(x,x) \stackrel{i_x}{\lra} \Hom_H(i(x), i(x)) = \Stab_H(i(x)) \] and \[ \Stab_G(y) = \Hom_G(y,y) \stackrel{i_y}{\lra} \Hom_H(i(y), i(y)) = \Stab_H(i(y)) \] and via the maps in cohomology in the opposite directions \[ i^x  : H^* \left( \Stab_H(i(x)) \right) \lra H^* \left( \Stab_G(x) \right) \] and \[  i^y  : H^* \left( \Stab_H(i(y)) \right) \lra H^* \left( \Stab_G(y) \right) \] we defined $i^*F(x) = i^x F(i(x))$ and $i^*F(y) = i^y F(i(y))$.

Condition 2 of definition \ref{groupoidcoh} for $i^*F$ amounts to showing that the map induced by the isomorphism $ \Stab_G(x) \stackrel{\phi_*}{\lra} \Stab_G(y)$ in cohomology, $\phi^*: H^* \left( \Stab_G(y) ,S \right) \lra H^* \left( \Stab_G(x) ,S \right)$, gives $i^*F(x) = \phi^* \left( i^*F(y) \right)$.

Using the identities above, the last equality boils down to showing that \[ i^x \left( i(\phi)^* F(i(y)) \right) = \phi^* \left( i^y F(i(y)) \right) \] i.e. that the following diagram commutes \begin{displaymath} \xymatrix{ H^* \left( \Stab_H (i(y)) ,S \right) \ar[r]^{i^y} \ar[d]_{i(\phi)^*} & H^* \left( \Stab_G (y) ,S \right) \ar[d]^{\phi^*} \\ H^* \left( \Stab_H (i(x)) ,S \right) \ar[r]_{i^x} &H^* \left( \Stab_G (x) ,S \right)} \end{displaymath}

Since all the maps are induced from group homomorphism, it suffices (by functoriality of cohomology w.r.t. group homomorphisms) to show that the following diagram commutes: \begin{displaymath} \xymatrix{  \Stab_H (i(y))  &  \Stab_G (y) \ar[l]_{i^y}  \\\Stab_H (i(x)) \ar[u]^{i(\phi)_*}  & \Stab_G (x) \ar[l]^{i^x} \ar[u]_{\phi_*} } \end{displaymath}

Let's check this latter condition. Given $g \in \Stab_G(x) = \Hom_G(x,x)$, applying $\phi_*$ yields $\phi \circ g \circ \phi^{-1} \in \Hom_G(y,y) = \Stab_G(y)$, then applying $i^y$ gives $i \left( \phi \circ g \circ \phi^{-1} \right) = i(\phi) \circ i(g) \circ i(\phi)^{-1}$ by functoriality of $i$.

On the other hand, we can start by applying $i^x$ to $g$, getting $i(g) \in \Hom_H(i(x), i(x)) = \Stab_H(i(x))$. Then applying $i(\phi)_*$ yields $i(\phi) \circ i(g) \circ i(\phi)^{-1}$, which proves the claim.

Finally, it remains to check that if $F$ is supported on finitely many connected components of $H$ and $\pi_0 i$ has finite fibers, then also $i^*F$ is supported on finitely many connected components of $G$. We can obviously assume that $F$ is supported on a unique connected component $\mathcal C$ of $H$; then $i^*F(x) \neq 0$ implies that $F(i(x)) \neq 0$ and thus $i(x) \in \mathcal C$. Hence, the connected components $\mathcal C_x$ of $x$ varies among the finite set $(\pi_o i )^{-1} ( \mathcal C)$, which concludes the proof.
%
\end{proof}
\begin{prop}[Well-posedness of pushforward] \label{pushforwardwellposed} Let $i: G \lra H$ be a finite covering morphism of groupoids. Then pushforward is well-defined as a map $\mathbb H^*_c(G) \lra \mathbb H^*_c(H)$.
\end{prop}
\begin{proof}
Recall that for $F \in \mathbb H^*_c(G)$ and $y \in \Ob(H)$ we defined \[ (i_*F)(y) = \sum_{H_y \backslash \backslash x \in i^{-1}(y)} \cores{\Stab_H(y)}{\Stab_G(x)} F(x). \]
We notice that corestriction is well-defined because $i$ is a finite covering, and also that the sum is well-defined, i.e. that the summand corresponding to a particular $H_y$-orbit on the fiber $i^{-1}(y)$ does not depend on the particular representative $x$ chosen.

Indeed, replacing $x$ by $hx$ for $h \in H_y$ (we denote $h \in \Hom_G(x,hx)$ the unique lift given by the finite covering property) we obtain \[ \cores{\Stab_H(y)}{\Stab_G(hx)} F(hx) = \cores{\Stab_H(y)}{h \Stab_G(x)h^{-1}} F(hx) = \cores{h^{-1} \Stab_H(y)h}{\Stab_G(x)} c_{h^{-1}}^* F(hx) = \cores{\Stab_H(y)}{\Stab_G(x)} c_{h^{-1}}^* F(hx)\] where $c_{h^{-1}}^* : H^* (h \Stab_G(x) h^{-1}, S) \lra H^* (\Stab_G(x), S )$ is the isomorphism in cohomology induced by the morphism $h \in \Hom_G(x, hx)$. By invariance of $F$, we obtain $c_{h^{-1}}^* F(hx) = F(x)$, which shows that the summand is independent of the representative.

We also remark that by proposition 2.1(b) in \cite{bhk}, two objects $x, x'$ of the fiber above $y$ are in the same connected component of $G$ if and only if they are in the same $H_y$-orbit: this shows that the sum has only finitely many nonzero summands, since $F$ is supported on finitely many connected components of $G$ (above $y$, if we only assume that $F$ is $i$-fiberwise compactly supported).

To prove that $i_*F$ satisfies the invariance condition of the definition of groupoid cohomology, fix $\phi \in \Hom_H(y,y')$, and notice that this induces a bijection between the $H_y$-orbits on $i^{-1}(y)$ and the $H_{y'}$-orbits on $i^{-1}(y')$ via $i^{-1}(y) \ni x \mapsto \widetilde \phi(x)$ where $\widetilde \phi$ is the \emph{unique} morphism of $G$ having source $x$ and mapping to $\phi$ under $i$.
Therefore, we have \[ (i_*F)(y' = \phi(y)) = \sum_{H_y \backslash \backslash x \in i^{-1}(y)} \cores{\Stab_H(\phi y)}{\Stab_G(\widetilde \phi x)} F(\widetilde \phi x) = \sum_{H_y \backslash \backslash x \in i^{-1}(y)} \cores{\Stab_H(\phi y)}{\Stab_G(\widetilde \phi x)} \widetilde \phi^* F(x) \] where the last equality is due to invariance of $F$. To prove that $(i_*F)(\phi y) = \phi^* (i_*F)(y)$ it suffices then to show that the diagram \begin{displaymath} \xymatrix{ H^*(G_x) \ar[r]^{\widetilde \phi^*} \ar[d]_{\cores{}{}} & H^*(G_{\widetilde \phi x}) \ar[d]^{\cores{}{}} \\ H^*(H_y) \ar[r]_{\phi^*} & H^*(H_{\phi y}) } \end{displaymath} commutes.
This follows from the fact that the group isomorphisms induced by $\phi$ and $\widetilde \phi$ via `conjugation' respect the inclusions $G_x \hookrightarrow H_y$ and $G_{\widetilde \phi x} \hookrightarrow H_{\phi y}$, in the sense that the diagram \begin{displaymath} \xymatrix{ G_x \ar[r]^{\widetilde \phi^*} \ar@{^{(}->}[d] & G_{\widetilde \phi x} \ar@{^{(}->}[d] \\ H_y \ar[r]_{\phi^*} & H_{\phi y} } \end{displaymath} is commutative.

Finally, the fact that $i_*F$ is supported on finitely many connected components of $H$ is a consequence of the same fact for $F$ and the fact that morphisms of groupoids preserve connectedness: more in details, if $i_*F$ is nonzero at $y$ and $y'$ it means that there exists objects $x \in i^{-1}(y)$ and $x' \in i^{-1}(y')$ where $F$ is nonzero. Assuming by linearity that $F$ is supported on only one connected component yields that there exists an element $g: x \lra x'$, but then $i(g): y \lra y'$.
\end{proof}
\primaformula*
\begin{proof}
This is obviously an equivariant version of the well-known formula for group cohomology \begin{equation} \label{coresintocup} (\cores{G}{H} \alpha) \cup \beta = \cores{G}{H} \left( \alpha \cup \res{G}{H} \beta \right). \end{equation}
We give all details since we use this result repeatedly in the paper.

First of all, notice that since $i$ is a covering morphism, it induces continuous inclusions on isotropy groups. Letting $x \in \Ob(X)$, $y  = i(x) \in \Ob(Y)$ and $z = j(y) = k(x) \in \Ob(Z)$ we have thus the commutative triangle of group homomorphisms \begin{displaymath} \xymatrix{ \Stab_X(x) \ar[r]^{k_x} \ar[d]_{i_x} & \Stab_Z(z) \\ \Stab_Y(y) \ar[ru]_{j_y} & } \end{displaymath} and since $i_x$ and $j_y$ are inclusions by assumption, so is $k_x$: this shows that $k$ also induces inclusions between isotropy groups. We also denote $j_y^*: H^*(Z_{j(y)}) \lra H^*(Y_y)$ and $k_x^*: H^*(Z_{k(x)}) \lra H^*(X_x)$ the induced maps on group cohomology.

Notice also that since $F \in \mathbb H^*(X)$ is $i$-fiberwise compactly supported - and the space of such cohomology classes is an ideal of $\mathbb H^*(X)$ under pointwise cup product - we obtain that $F \cup k^*G$ is also $i$-fiberwise compactly supported, thus making sense of the right hand side of the formula we aim to prove.

Let $y \in \Ob(Y)$, then we have \[ (i_* F \cup j^* G)(y) = i_*F(y) \cup j^*G (y) = \left( \sum_{Y_y \backslash \backslash i^{-1}(y) \ni x} \cores{Y_y}{X_x} F(x)  \right) \cup \left( j^*_y G(j(y)) \right) = \] \[ = \sum_{Y_y \backslash \backslash i^{-1}(y) \ni x} \left( \cores{Y_y}{X_x} F(x) \cup j^*_y G(j(y)) \right). \]
The assumptions on the maps guarantees that for each $y \in \Ob(Y)$ and each $x \in i^{-1}(x)$ we have $X_x \subset Y_y \subset Z_{j(y)}$. We apply the abovementioned formula \ref{coresintocup} for group cohomology for each single summand, where $G= Y_y \supset X_x = H$, $\alpha = F(x)$ and $\beta = j^*_y G(j(y))$.
We obtain \[ \sum_{Y_y \backslash \backslash i^{-1}(y) \ni x} \cores{Y_y}{X_x} \left( F(x) \cup \res{Y_y}{X_x} j^*_y G(j(y)) \right) = \sum_{Y_y \backslash \backslash i^{-1}(y) \ni x} \cores{Y_y}{X_x} \left( F(x) \cup k_x^* G(k(x)) \right) = \] \[ = \sum_{Y_y \backslash \backslash i^{-1}(y) \ni x} \cores{Y_y}{X_x} \left( (F \cup k^*G)(x)) \right) = \left( i_* (F \cup k^*G ) \right) (y). \]
The final assertion of the lemma is immediate, since pushforward along $i$ always preserves the compactly supported cohomology, and under the additional assumption that $j$ and $k$ induce finite-fibers map in cohomology, the pullbacks along $j$ and $k$ preserve that too.
\end{proof}
\secondaformula*
\begin{proof}
This is clearly an equivariant version of the restriction-corestriction formula for group cohomology. We give all details, since we use this lemma several times.

Firstly notice that by proposition 2.8 in \cite{brown}, $\widetilde i$ is also a covering morphism, and is continuous as seen in the definition \ref{defpullback} of pullback and homotopy pullback.

We start by proving that $\widetilde \pi$ also induces an injection on isotropy groups. Let $z \in \Ob(Z)$, and let $x = \widetilde \pi(z)$, $y = \widetilde i(z)$ and $\pi(\widetilde i(z)) = a = i(\widetilde \pi(z))$.
By theorem 2.2(v) in \cite{bhk2}, the isotropy group $\Stab_Z(z)$ is the pullback of the diagram \begin{displaymath} \xymatrix{ & \Stab_X(x) \ar[d]^{i_x} \\ \Stab_Y(y) \ar[r]_{\pi_y} & \Stab_A(a) } \end{displaymath} and on the other hand $i_x: \Stab_X(x) \lra \Stab_A(a)$ is an injection since $i$ is a covering morphism, while $\pi_y: \Stab_Y(y) \lra \Stab_A(a)$ is an injection by assumption on $\pi$.

This proves that the pullback of the diagram above is simply the intersection $\Stab_Y(y) \cap \Stab_X(x) \subset \Stab_A(a)$, where we identify $\Stab_Y(Y)$ with its image under $\pi_y$ inside $\Stab_A(a)$, and same for $\Stab_X(x)$.
In particular, $\Stab_Z(z) \cong \Stab_Y(y) \cap \Stab_X(x) \subset \Stab_X(x)$ inside $\Stab_A(a)$, so that $\widetilde \pi$ also induces injection of isotropy groups.
Moreover, $\Stab_Z(z) = \Stab_Y(y) \cap \Stab_X(x)$ is open in $\Stab_Y(y)$, since $\Stab_X(x) \subset \Stab_A(a)$ is open by assumption on $i$ being a finite covering morphism, and the property of being an open map is preserved under pullback by the continuous map $\pi_y: \Stab_Y(y) \lra \Stab_A(a)$.

The above reasoning also shows that $\widetilde i$ is a \emph{finite} covering morphism: indeed we have shown that $\Stab_Z(z) \cong \Stab_Y(y) \cap \Stab_X(x)$, and on the other hand the map of coset spaces \[ \Stab_Y(y) / \left( \Stab_Y(y) \cap \Stab_X(x) \right) \lra \Stab_A(a) / \Stab_X(x) \quad g \left( \Stab_Y(y) \cap \Stab_X(x) \right) \mapsto g \Stab_X(x) \] is easily seen to be injective, so the target $\Stab_A(a) / \Stab_X(x)$ being finite by assumption on $i$ proves the same finiteness result for $\widetilde i$.

The same argument when we switch the roles of $\pi$ and $i$ shows that, as $\pi$ is assumed to inducing open injections on stabilizers, so does $\widetilde \pi$.

We also show that under the additional assumption that $\pi$ induces a finite fibers map on connected components so does $\widetilde \pi$, proving that the pullback $\widetilde \pi^*$ preserves compactly supported cohomology. The pullback square induces a commutative diagram on connected components: \begin{displaymath} \xymatrix{ \pi_0 Z \ar[r]_{\widetilde \pi_0} \ar[d]_{\widetilde i_0} & \pi_0 X \ar[d]^{i_0} \\ \pi_0 Y \ar[r]_{\pi_0} & \pi_0 A } \end{displaymath}
Fix a component $\mathcal X \in \pi_0 X$. Then $i(\mathcal X) \in \pi_0 A$ has finitely many connected components of $Y$ above it, by the assumption on $\pi$: denote them by $\{ \mathcal Y_1, \ldots, \mathcal Y_m \}$. Then clearly \[ \widetilde \pi_0^{-1} ( \mathcal X ) = \coprod_{i=1}^m \widetilde \pi_0^{-1} (\mathcal X) \cap \widetilde i_0 ^{-1} (\mathcal Y_i) \] and hence it suffices to prove that each of $\widetilde \pi_0^{-1} (\mathcal X) \cap \widetilde i_0 ^{-1} (\mathcal Y_i)$ is finite.

We fix then one such $\mathcal Y_i$ and denote it by $\mathcal Y$ from now on. Let $z \in \Ob(Z)$ be in one of the components of $\widetilde \pi_0^{-1} (\mathcal X) \cap \widetilde i_0 ^{-1} (\mathcal Y)$, and denote $x = \widetilde \pi(z) \in \mathcal X$, $y = \widetilde i(z) \in \mathcal Y$ and $i(x) = a = \pi(y)$.
The `Mayer-Vietoris' sequence of theorem 2.2 in \cite{brown} is \begin{displaymath} \xymatrix{ & \Stab_X(x) \ar[rd]^i & & & \pi_0 X \ar[rd]^{i_0} & \\ \Stab_Z(z) \ar[rd]_{\widetilde i} \ar[ru]^{\widetilde \pi} & & \Stab_A(a) \ar[r]^{\Delta} & \pi_0 Z \ar[rd]_{\widetilde i_0} \ar[ru]^{\widetilde \pi_0} & & \pi_0 A \\ & \Stab_Y(y) \ar[ru]_{\pi} & & & \pi_0 Y \ar[ru]_{\pi_0} & \\ } \end{displaymath} and part (ii) of the statement of the theorem is that the image of $\Delta$ inside $\pi_0 Z$ is the intersection $\widetilde \pi_0^{-1} (\mathcal X) \cap \widetilde i_0^{-1} (\mathcal Y)$.

On the other hand, part (iv) of the same theorem says that the image of $\Delta$ is in bijection with the double coset space $\pi \left( \Stab_Y(y) \right) \backslash \Stab_A(a) / i \left( \Stab_X(x) \right)$. The assumption on $i$ inducing finite-index inclusions on isotropy groups implies that this double coset space is finite, and this concludes the proof that $\widetilde \pi_0$ also has finite fibers.

Before proving the main formula, it remains to show that if $F$ is $i$-fiberwise compactly supported, then $\widetilde \pi^* F$ is $\widetilde i$-fiberwise compactly supported, so that the left hand side of the formula is well-defined. Fix then $y \in \Ob(Y)$: we want to show that $\widetilde \pi^*F$ is supported on finitely many connected components of $Z$ above $y$.

Proposition 2.1 in \cite{bhk} says that we have a bijection between $\Stab_Y(y)$-orbits on $\widetilde i^{-1}(y)$ and connected components of $Z$ intersecting the fiber $\widetilde i^{-1}(y)$, so it suffices to show that $\widetilde \pi^* F$ is only nonzero on finitely many $\Stab_Y(y)$-orbits on $\widetilde i^{-1}(y)$.

Let now $a = \pi(y) \in \Ob(A)$. An explicit model for the pullback groupoid $Z$ gives that the objects $\Ob(Z)$ are pairs in $\Ob(Y) \times \Ob(X)$ mapping to the same object of $A$.

In particular, the objects of the fiber $\widetilde i^{-1}(y)$ consists of pairs $\{ (y,x), \, | \, a = \pi(y) = i(x) \}$ - this is obviously in bijection with the fiber $i^{-1}(a)$ via the map $(y, x) \mapsto x$.
\begin{claim} \label{sameaction}
This bijection respects the $\Stab_Y(y)$-actions on each side: the action on $\widetilde i^{-1}(y)$ is the natural one, while the action on $i^{-1}(a)$ is the restriction to $\Stab_Y(y)$ of the natural action of $\Stab_A(a)$ (recall that by assumption $\pi$ induces an injection $\Stab_Y(y) \hookrightarrow \Stab_A(a)$). 
\end{claim}
\begin{proof}[Proof of claim] Fix $(y, x) \in \Ob(\widetilde i^{-1}(y))$ and let $g \in \Stab_Y(y)$. As $\widetilde i$ is a covering, $g$ lifts to a unique morphism in $Z$: denote it $\hat g:(y, x) \mapsto (y, x')$. We thus want to show that the $\Stab_A(a)$-action of $\pi(g)$ on $x \in i^{-1}(a)$ maps it to $x'$.

Under the map $\pi$, we have that $\pi(g) \in \Stab_A(a)$ and as such it acts on the fiber $i^{-1}(a)$: its unique lift under the covering $i$ having source $x$ is the element $\widehat{\pi(g)}: x \mapsto \widehat x'$, and thus the $\Stab_Y(y)$-action of $g$ sends $x$ to $\widehat x'$.

Now notice that \[ \pi(g) = \pi( \widetilde i(\hat g) ) = i \left( \widetilde \pi (\hat g) \right) \] so that $\widetilde \pi(\hat g)$ is another lift under $i$ of $\pi(g)$: the uniqueness statement of the covering means that we must have $\widetilde \pi(\hat g) = \widehat{\pi(g)}$. But $\widetilde \pi(\hat g): x \mapsto x'$, so that $ x' = \widehat x'$. This proves that the two $\Stab_Y(y)$-actions coincide under the given bijection, and completes the proof of the claim.
\end{proof}
Finally, notice that as $F$ is $i$-fiberwise compactly supported, $F(x) \neq 0$ only on finitely many connected components of $X$ in the fiber $i^{-1}(a)$. By proposition 2.1 of \cite{bhk}, these are in bijection with $\Stab_A(a)$-orbits on the fiber $i^{-1}(a)$, and since $\Stab_Y(y) \hookrightarrow \Stab_A(a)$ is a finite index inclusion by assumption on $\pi$, we obtain that $F(x) \neq 0$ only on finitely many $\Stab_Y(y)$-orbits.

Since $F(x) \neq 0$ is a necessary condition for $\widetilde \pi^* F ((y,x)) \neq 0$, the statement of the claim completes the proof that $\widetilde \pi^* F$ is $\widetilde i$-fiberwise compactly supported.

We are finally ready to prove the main formula of the lemma. Notice that the assumption on $\pi$ and $\widetilde \pi$ inducing open injections of isotropy groups at all objects means that the formula for the pullback becomes \[ (\widetilde \pi)^* F (z) = \left( \widetilde \pi_*^{(z)} \right)^* F(\widetilde \pi (z)) = \res{\Stab_X(\widetilde \pi(z))}{\Stab_Z(z)} F(\widetilde \pi (z)) \] and similarly for $\pi$.

Fix $y \in \Ob(Y)$, then by definitions we have \[ (\widetilde i)_* \left( (\widetilde \pi)^* F \right) (y) = \sum_{\Stab_Y(y) \backslash \backslash z \in \widetilde i^{-1}(y)} \cores{\Stab_Y(y)}{\Stab_Z(z)} (\widetilde \pi^* F) (z) = \] \[ = \sum_{\Stab_Y(y) \backslash \backslash z \in \widetilde i^{-1}(y)} \cores{\Stab_Y(y)}{\Stab_Z(z)} \res{\Stab_X(\widetilde \pi(z))}{\Stab_Z(z)} F(\widetilde \pi(z)). \]
On the other hand, we have \[ \pi^* ( i_* F) (y) = \res{\Stab_A(\pi (y))}{\Stab_Y(y)} (i_*F)(\pi(y)) = \res{\Stab_A(\pi (y))}{\Stab_Y(y)} \sum_{\Stab_A(\pi(y)) \backslash \backslash x \in i^{-1}(\pi(y))} \cores{\Stab_A(\pi(y))}{\Stab_X(x)} F(x). \]
Moving the restriction map inside the summation and applying the restriction-corestriction formula for group cohomology yields \[ \sum_{\Stab_A(\pi(y)) \backslash \backslash x \in i^{-1}(\pi(y))} \sum_{\Stab_Y(y) \backslash \Stab_A(\pi(y)) / \Stab_X(x) \ni g} \cores{\Stab_Y(y)}{\Stab_Y(y) \cap g \Stab_X(x)g^{-1}} \res{g \Stab_X(x) g^{-1}}{\Stab_Y(y) \cap g \Stab_X(x)g^{-1}} c_g^* F(x). \]
Since $\pi$ is a covering morphism, each double coset representative $g$, which is an element of $\Stab_A(\pi(y))$, has a well-defined lift to a morphism in $X$ with source $x$; in particular $gx \in \Ob(X)$ is well-defined, and its stabilizer is obviously $\Stab_X(gx) = g \Stab_X(x) g^{-1}$.
Invariance of $F$ also gives that the single summand is $\cores{\Stab_Y(y)}{\Stab_Y(y) \cap \Stab_X(gx)} \res{\Stab_X(gx)}{\Stab_Y(y) \cap \Stab_X(gx)} F(gx)$.
Notice that the intersection $\Stab_Y(y) \cap \Stab_X(gx)$ is taken inside $\Stab_A(\pi(y))$, and since both maps between isotropy groups are inclusions (by assumption on $\pi$ and $i$), theorem 2.2(v) in \cite{bhk2} yields that this intersection is precisely $\Stab_Z(z)$.

We can wrap up the two summation in a single one by defining $\widetilde x = gx$ to be representatives for the $\Stab_Y(y)$-orbits on $i^{-1}(\pi(y))$, where the action is the restriction to $\Stab_Y(y)$ of the natural action of $\Stab_A(\pi(y))$: \[ \pi^* ( i_* F) (y) = \sum_{\Stab_Y(y) \backslash \backslash i^{-1}(\pi(y)) \ni \widetilde x} \cores{\Stab_Y(y)}{\Stab_Z((y,\widetilde x))} \res{\Stab_X(\widetilde x)}{\Stab_Z((y,\widetilde x))} F( \widetilde x)  \] where we notice that an explicit model of the objects of $Z$ is given by pairs in $\Ob (Y) \times \Ob(X)$ which map to the same object in $A$ - and each choice of $(y, \widetilde x)$ in the sum satisfies this requirement.

To finish the proof of the lemma, it remains to show that the last summation coincides with$ \sum_{\Stab_Y(y) \backslash \backslash z \in \widetilde i^{-1}(y)} \cores{\Stab_Y(y)}{\Stab_Z(z)} \res{\Stab_X(\widetilde \pi(z))}{\Stab_Z(z)} F(\widetilde \pi(z))$.

It is immediate - again by the description of the objects of the pullback $Z$ as pairs $\Ob(X) \times \Ob(Y)$ mapping to the same object of $A$ - that we have $\Ob \left( \widetilde i^{-1}(y) \right) = \left\{ (y,\widetilde x), \, | \, \pi(y) = i(\widetilde x) \right\}$ in bijection with $\Ob \left( i^{-1} (\pi(y)) \right)$ via $(y, \widetilde x) \mapsto \widetilde x$.

It remains to check that the two actions of $\Stab_Y(y)$ correspond under this identification, as then it will follow that the orbits correspond as well: the two summations will be indexed by the same set and will coincide summand by summand. This is precisely the result of claim \ref{sameaction}.
\end{proof}
We now describe a very general setup for derived Hecke algebras. Let $G$ be a topological group acting transitively on a set $X$, for instance $X = G / K$ with the action by left multiplication.

Consider the topological groupoid $\left( G \curvearrowright X^2 \right)$ where $g.(x_1, x_2) = (gx_1, gx_2)$ whose topology is defined as in example \ref{topgroupoidexmp}. We denote this groupoid by $\underline G_X$. Similarly, let $\left( G \curvearrowright X^3 \right)$ be another topological groupoid, where $G$ acts diagonally as above. Let $i_{1,2}: \left( G \curvearrowright X^3 \right) \lra \underline G_X$ be the projection onto the first and second factors, and similarly define $i_{2,3}$ and $i_{1,3}$. Fix a base point $x_0 \in X$ and denote $H = \Stab_G(x_0)$.
\begin{fact} \label{generaldha} Suppose that the following conditions hold: \begin{enumerate}
\item For each $F_1, F_2 \in \mathbb H^*_c (\underline G_X)$, the cohomology class $i_{1,2}^* F_1 \cup i_{2,3}^* F_2$ is $i_{1,3}$-fiberwise compactly supported;
\item For all $x, z \in H \backslash G / H$, the subset $HxHzH$ is a finite union of $(H,H)$-cosets.
\item $i_{1,3}$ is a finite covering morphism.
\end{enumerate}
Then the compactly supported cohomology $\mathbb H^*_c(\underline G_X)$ admits the structure of a derived Hecke algebra under convolution, as follows: given $F_1, F_2 \in \mathbb H^*_c(\underline G_X)$ we define their convolution to be the cohomology class resulting from the following diagram \begin{displaymath} \xymatrix{ \underline G_X^{F_1} & & \underline G_X^{F_2} \\ & \left( G \curvearrowright X^3 \right) \ar[lu]^{i_{1,2}} \ar[ru]_{i_{2,3}} \ar[d]_{i_{1,3}} & \\ & \underline G_X & } \end{displaymath}
where we pullback $F_1$ and $F_2$ to $\left( G \curvearrowright X^3 \right)$, cup them, then pushforward along $i_{1,3}$ to $\underline G_X$. We denote this algebra by $\mathcal H(G,X)$.
\end{fact}
\begin{rem}
Before proving the fact, we notice that if $H$ is compact and open in $G$ then all conditions are automatically satisfied. Indeed, $i_{1,3}$ is always a covering, so it is a finite covering as long as the index $[G_{x,z}: G_{x,y,z}]$ is finite, which holds once $H$ is compact open. Similarly, condition 2 follows immediately since $HxHzH$ is compact if $H$ is. Finally, condition 1 is implied by $G_{x,z}$ having finitely many orbits on the set of $y \in X$ such that $F_1(x,y) \neq 0 \neq F_2(y,z)$: since $F_1$ is compactly supported, $G_x$ has finitely many orbits on that set, so finiteness of the index $[G_x : G_{x,z}]$ implies condition 1.
\end{rem}
\begin{proof}
Conditions 1 guarantees that we can pushforward along $i_{1,3}$. To show that this pushforward is compactly supported, we notice that by linearity we can assume that $F_1$ is supported on the $G$-orbit of $(x_0, xx_0)$ and $F_2$ on the $G$-orbit of $(x_0, zx_0)$ - or in other words, $F_1$ is supported on the connected component $HxH$ and $F_2$ on $HzH$. Then unraveling the formulas for pullback, cup product and pushforward shows that the final cohomology class is supported on $HxHzH$: condition 2 assures that this is a finite union of $(H,H)$-cosets - or in other words, it consists of finitely many connected components.

It remains to prove that this operation is associative and defines thus an algebra structure: notice that the unit under this operation is the cohomology class supported on the $G$-orbit of $(x_0, x_0)$ and taking value $1(x_0, x_0) = 1 \in S \cong H^0(H,S)$, the unit element of the ring $S$.

It remains to show that this operation is associative: given $A, B, C \in \mathbb H^*_c(\underline G_X)$ we need to check that the following two diagrams output the same cohomology class:
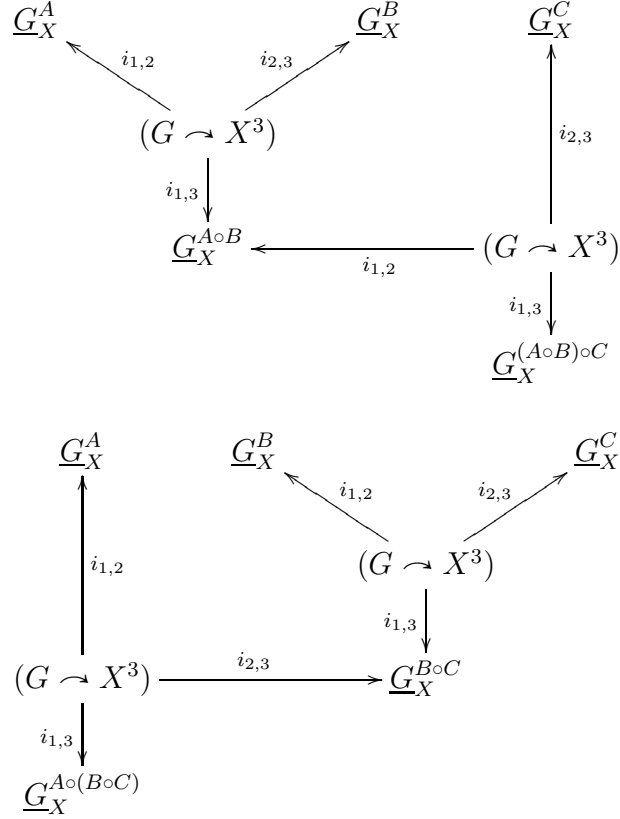
\begin{displaymath} \xymatrix{ \underline G_X^A & & \underline G_X^B & \underline G_X^C \\ & \left( G \curvearrowright X^3 \right) \ar[lu]_{i_{1,2}} \ar[ru]^{i_{2,3}} \ar[d]_{i_{1,3}} & & \\ & \underline G_X^{A \circ B} & &  \left( G \curvearrowright X^3 \right) \ar[uu]_{i_{2,3}} \ar[ll]^{i_{1,2}} \ar[d]_{i_{1,3}} \\ & & & \underline G_X^{(A \circ B) \circ C} } \end{displaymath}
\begin{displaymath} \xymatrix{ \underline G_X^A & \underline G_X^B & & \underline G_X^C \\ & & \left( G \curvearrowright X^3 \right) \ar[lu]_{i_{1,2}} \ar[ru]^{i_{2,3}} \ar[d]_{i_{1,3}} & \\ \left( G \curvearrowright X^3 \right) \ar[uu]_{i_{1,2}} \ar[rr]^{i_{2,3}} \ar[d]_{i_{1,3}} & & \underline G_X^{B \circ C} &  \\ \underline G_X^{A \circ (B \circ C)} & & & } \end{displaymath}
We start with the following fact.
\begin{claim}
The square \begin{displaymath} \xymatrix{ \left( G \curvearrowright X^3 \right) \ar[d]_{i_{1,3}} & \left( G \curvearrowright X^4 \right) \ar[l]_{i_{1,2,3}} \ar[d]^{i_{1,3,4}}  \\ \underline G_X & \left( G \curvearrowright X^3 \right) \ar[l]^{i_{1,2}} } \end{displaymath} is a pullback square of groupoids. Moreover, it satisfies the additional assumptions of lemma \ref{cohomswitch}, that is to say: $i_{1,3}$ is a finite covering morphism and $i_{1,2}$ induces an injection of isotropy groups at all objects.
\end{claim}
\begin{proof}[Proof of claim] By assumption $i_{1,3}$  is a finite covering morphisms, and it is immediate to see that $i_{1,2}$ induces an injection of isotropy groups at all objects.

By proposition 4.4 in \cite{bhk}, the pullback square of a diagram of groups acting on sets is constructed by taking the pullback of the groups acting on the pullback of the sets. It is immediate that the pullback group is again $G$. Moreover, since $i_{1,3}$ is a covering, the pullback and the homotopy pullback coincide, hence a model for the object set of the pullback is \[ \left\{ \left( (x,y,z) , (a,b,c) \right) \in X^3 \times X^3 \, | \, i_{1,3}(x,y,z) = i_{1,2}(a,b,c) \right\} = \] \[ = \left\{ \left( (x,y,z) , (a,b,c) \right) \in X^3 \times X^3 \, | \, (x,z) = (a,b) \right\} = \left\{ \left( x =a, y, z=b, c \right) \right\} = X^4 \] with the maps $i_{1,3,4}$ and $i_{1,2,3}$ as above. 
\end{proof}
%
By assumption on $G \curvearrowright X$ the cohomology class $F = i_{1,2}^*A \cup i_{2,3}^* B$ is $i_{1,3}$-fiberwise compactly supported, thus lemma \ref{cohomswitch} applies and we can use the above pullback squares to modify the two original diagrams without impacting the final result in cohomology. The first diagram becomes \begin{displaymath} \xymatrix{ \underline G_X^A & & \underline G_X^B & \underline G_X^C \\ & \left( G \curvearrowright X^3 \right) \ar[lu]_{i_{1,2}} \ar[ru]^{i_{2,3}} &  \left( G \curvearrowright X^4 \right) \ar[l]_{i_{1,2,3}} \ar[dr]^{i_{1,3,4}} & \\ & & &  \left( G \curvearrowright X^3 \right) \ar[uu]_{i_{2,3}} \ar[d]_{i_{1,3}} \\ & & & \underline G_X^{(A \circ B) \circ C} } \end{displaymath}
A biproduct of lemma \ref{cohomswitch} was that $F_1 = i_{1,2,3}^* F$ is $i_{1,3,4}$-fiberwise compactly supported, thus we can apply lemma \ref{corescupres} to the subdiagram \begin{displaymath} \xymatrix{ \left( G \curvearrowright X^4 \right) \ar[dr]^{i_{1,3,4}} \ar[r]^{i_{3,4}} & \underline G_X \\ &  \left( G \curvearrowright X^3 \right) \ar[u]_{i_{2,3}} } \end{displaymath} to obtain that for each $F_1 \in \mathbb H^* \left( \left( G \curvearrowright X^4 \right) \right)$ that is $i_{1,3,4}$-fiberwise compactly supported and each $F_2 \in \mathbb H^* (\underline G_X)$, we have \[ (i_{1,3,4})_* F_1 \cup i_{2,3}^* F_2 = (i_{1,3,4})_* \left( F_1 \cup i_{3,4}^* F_2 \right) \] and hence we can replace the first diagram with \begin{displaymath} \xymatrix{ \underline G_X^A & & \underline G_X^B & \underline G_X^C \\ & \left( G \curvearrowright X^3 \right) \ar[lu]_{i_{1,2}} \ar[ru]^{i_{2,3}} &  \left( G \curvearrowright X^4 \right) \ar[l]_{i_{1,2,3}} \ar[dr]^{i_{1,3,4}} \ar[ru]^{i_{3,4}} & \\ & & &  \left( G \curvearrowright X^3 \right) \ar[d]_{i_{1,3}} \\ & & & \underline G^{(A \circ B) \circ C} } \end{displaymath} which is equivalent to \begin{displaymath} \xymatrix{ \underline G^A & \underline G_X^B & \underline G_X^C \\ &  \left( G \curvearrowright X^4 \right) \ar[lu]_{i_{1,2}} \ar[u]_{i_{2,3}} \ar[d]^{i_{1,4}} \ar[ru]_{i_{3,4}} & \\ & \underline G_X^{(A \circ B) \circ C} & } \end{displaymath}
This last diagram is obviously symmetric in $A, B, C$, which proves associativity.
\end{proof}

Suppose now we have another set $Y$ on which $G$ acts transitively. Fix a base point $y_0 \in Y$ and denote $L = \Stab_G(y_0)$. As before, we have the projection maps $i_{1,2}: \left( G \curvearrowright X^2 \times Y \right) \lra \underline G_X$, $i_{1,3}, i_{2,3}: \left( G \curvearrowright X^2 \times Y \right) \lra \left( G \curvearrowright X \times Y \right)$.
\begin{fact} \label{generalaction} Suppose that the following conditions hold: \begin{enumerate}
\item For each $F \in \mathbb H^*_c (\underline G_X)$, $C \in \mathbb H^*_c (G \curvearrowright X \times Y)$, the cohomology class $i_{1,2}^* F \cup i_{2,3}^* C$ is $i_{1,3}$-fiberwise compactly supported;
\item For all $x \in H \backslash G / H$, $z \in H \backslash G / L$ the subset $HxHzL$ is a finite union of $(H,L)$-cosets.
\item $i_{1,3}$ is a finite covering morphism.
\end{enumerate}
Then the derived Hecke algebra $\mathcal H(G,X)$ acts on the compactly supported cohomology $\mathbb H^*_c(G \curvearrowright X \times Y)$, as follows: given $F \in \mathcal H(G,X)$ and $C \in \mathbb H^*_c(G \curvearrowright X \times Y)$, we define their convolution to be the cohomology class resulting from the following diagram \begin{displaymath} \xymatrix{ \underline G_X^F & & \left( G \curvearrowright X \times Y \right)^C \\ & \left( G \curvearrowright X^2 \times Y \right) \ar[lu]^{i_{1,2}} \ar[ru]_{i_{2,3}} \ar[d]_{i_{1,3}} & \\ & \left( G \curvearrowright X \times Y \right) & } \end{displaymath}
where we pullback $F$ and $C$ to $\left( G \curvearrowright X^2 \times Y \right)$, cup them, then pushforward along $i_{1,3}$ to $\left( G \curvearrowright X \times Y \right)$.
\end{fact}
\begin{rem}
Before proving the fact, we notice that if $H$ is compact and open in $G$ and $L$ is closed in $G$ then all conditions are automatically satisfied.
Indeed, conditions 2 follows from $H$ being compact open, as then $HxH$ is a finite union of left $H$-cosets.
Since $i_{1,3}$ is always a covering, we notice that it is a finite covering as long as the index $[G_{x,z}: G_{x,y,z}]$ is finite for all $x,z \in X$ and $y \in Y$. This holds once $H$ is compact open and $L$ is closed because $G_{x,y}$ is closed in the compact set $G_x$, hence $G_{x,y}$ is compact and intersecting with the open $G_z$ yields that $G_{x,y,z}$ is open in $G_{x,y}$, thus finite-index.
Finally, condition 1 is implied by $G_{x,y}$ having finitely many orbits on the set of $z \in X$ such that $F(x,z) \neq 0 \neq C(z,y)$: assuming $F$ is supported on a single connected component $HzH$, and that (by $G$-invariance) $x=x_0$, it suffices to show that the set of such $z$'s is finite: this set is in bijection with $HzH / H$ which is finite as $H$ is open and compact.
\end{rem}
\begin{proof}
Conditions 1 guarantees that we can pushforward along $i_{1,3}$. To show that this pushforward is compactly supported, we notice that by linearity we can assume that $F$ is supported on the $G$-orbit of $(x_0, xx_0)$ and $C$ on the $G$-orbit of $(x_0, zy_0)$ - or in other words, $F$ is supported on the connected component $HxH$ and $C$ on $HzL$. Then unraveling the formulas for pullback, cup product and pushforward shows that the final cohomology class is supported on $HxHzL$: condition 2 assures that this is a finite union of $(H,L)$-cosets - or in other words, it consists of finitely many connected components.

It remains to prove that this defines indeed a left-action, in the sense that given $F_1, F_2 \in \mathcal H(G,X)$ and $C \in \mathbb H^*_c \left( G \curvearrowright X \times Y \right)$, we have $F_1.(F_2.C) = (F_1 \circ F_2).C$. In other words, we need to check that the following two diagrams output the same cohomology class:
\begin{displaymath} \xymatrix{ \underline G_X^{F_1} & & \underline G_X^{F_2} & \left( G \curvearrowright X \times Y \right)^C \\ & \left( G \curvearrowright X^3 \right) \ar[lu]_{i_{1,2}} \ar[ru]^{i_{2,3}} \ar[d]_{i_{1,3}} & & \\ & \underline G_X^{F_1 \circ F_2} & &  \left( G \curvearrowright X^2 \times Y \right) \ar[uu]_{i_{2,3}} \ar[ll]^{i_{1,2}} \ar[d]_{i_{1,3}} \\ & & & \left( G \curvearrowright X \times Y \right)^{(F_1 \circ F_2).C} } \end{displaymath}
\begin{displaymath} \xymatrix{ \underline G_X^{F_1} & \underline G_X^{F_2} & & \left( G \curvearrowright X \times Y \right)^C \\ & & \left( G \curvearrowright X^2 \times Y \right) \ar[lu]_{i_{1,2}} \ar[ru]^{i_{2,3}} \ar[d]_{i_{1,3}} & \\ \left( G \curvearrowright X^2 \times Y \right) \ar[uu]_{i_{1,2}} \ar[rr]^{i_{2,3}} \ar[d]_{i_{1,3}} & & \left( G \curvearrowright X \times Y \right)^{F_2.C} &  \\ \left( G \curvearrowright X \times Y \right)^{F_1.(F_2.C)} & & & } \end{displaymath}
We start with the following fact.
\begin{claim}
The square \begin{displaymath} \xymatrix{ \left( G \curvearrowright X^3 \right) \ar[d]_{i_{1,3}} & \left( G \curvearrowright X^3 \times Y \right) \ar[l]_{i_{1,2,3}} \ar[d]^{i_{1,3,4}}  \\ \underline G_X & \left( G \curvearrowright X^2 \times Y \right) \ar[l]^{i_{1,2}} } \end{displaymath} is a pullback square of groupoids. Moreover, it satisfies the additional assumptions of lemma \ref{cohomswitch}, that is to say: $i_{1,3}$ is a finite covering morphism and $i_{1,2}$ induces an injection of isotropy groups at all objects.
\end{claim}
\begin{proof}[Proof of claim] By assumption on $G \curvearrowright X$, $i_{1,3}$ is a finite covering morphisms, and moreover it is immediate that $i_{1,2}$ induces an injection of isotropy groups at all objects.

By proposition 4.4 in \cite{bhk}, the pullback square of a diagram of groups acting on sets is constructed by taking the pullback of the groups acting on the pullback of the sets. It is immediate that the pullback group is again $G$. Moreover, since $i_{1,3}$ is a covering, the pullback and the homotopy pullback coincide, hence a model for the object set of the pullback is \[ \left\{ \left( (x,y,z) , (a,b,c) \right) \in X^3 \times (X^2 \times Y) \, | \, i_{1,3}(x,y,z) = i_{1,2}(a,b,c) \right\} = \] \[ = \left\{ \left( (x,y,z) , (a,b,c) \right) \in X^3 \times (X^2 \times Y) \, | \, (x,z) = (a,b) \right\} = \left\{ \left( x =a, y, z=b, c \right) \right\} = X^3 \times Y \] with the maps $i_{1,3,4}$ and $i_{1,2,3}$ as above. 
\end{proof}
By assumption on $G \curvearrowright X$, the cohomology class $F = i_{1,2}^*F_1 \cup i_{2,3}^*F_2$ is $i_{1,3}$-fiberwise compactly supported, thus lemma \ref{cohomswitch} applies and we can use the above pullback squares to modify the two original diagrams without impacting the final result in cohomology. The first diagram becomes \begin{displaymath} \xymatrix{ \underline G_X^{F_1} & & \underline G_X^{F_2} & \left( G \curvearrowright X \times Y \right)^C \\ & \left( G \curvearrowright X^3 \right) \ar[lu]_{i_{1,2}} \ar[ru]^{i_{2,3}} &  \left( G \curvearrowright X^3 \times Y \right) \ar[l]_{i_{1,2,3}} \ar[dr]^{i_{1,3,4}} & \\ & & &  \left( G \curvearrowright X^2 \times Y \right) \ar[uu]_{i_{2,3}} \ar[d]_{i_{1,3}} \\ & & & \left( G \curvearrowright X \times Y \right)^{(F_1 \circ F_2).C} } \end{displaymath}
A biproduct of lemma \ref{cohomswitch} was that $\widetilde F = i_{1,2,3}^* F$ is $i_{1,3,4}$-fiberwise compactly supported, thus we can apply lemma \ref{corescupres} to the subdiagram \begin{displaymath} \xymatrix{ \left( G \curvearrowright X^3 \times Y \right) \ar[dr]^{i_{1,3,4}} \ar[r]^{i_{3,4}} & \left( G \curvearrowright X \times Y \right) \\ &  \left( G \curvearrowright X^2 \times Y \right) \ar[u]_{i_{2,3}} } \end{displaymath} to obtain that for each $\widetilde F \in \mathbb H^* \left( \left( G \curvearrowright X^3 \times Y \right) \right)$ that is $i_{1,3,4}$-fiberwise compactly supported and each $C \in \mathbb H^* \left( \left( G \curvearrowright X \times Y \right) \right)$, we have \[ (i_{1,3,4})_* \widetilde F \cup i_{2,3}^* C = (i_{1,3,4})_* \left( \widetilde F \cup i_{3,4}^* C \right) \] and hence we can replace the first diagram with \begin{displaymath} \xymatrix{ \underline G_X^{F_1} & & \underline G_X^{F_2} & \left( G \curvearrowright X \times Y \right)^C \\ & \left( G \curvearrowright X^3 \right) \ar[lu]_{i_{1,2}} \ar[ru]^{i_{2,3}} &  \left( G \curvearrowright X^3 \times Y \right) \ar[l]_{i_{1,2,3}} \ar[dr]^{i_{1,3,4}} \ar[ru]^{i_{3,4}} & \\ & & &  \left( G \curvearrowright X^2 \times Y \right) \ar[d]_{i_{1,3}} \\ & & & \left( G \curvearrowright X \times Y \right)^{(F_1 \circ F_2).C} } \end{displaymath} which is equivalent to \begin{displaymath} \xymatrix{ \underline G^{F_1} & \underline G_X^{F_2} & \left( G \curvearrowright X \times Y \right)^C \\ &  \left( G \curvearrowright X^3 \times Y \right) \ar[lu]_{i_{1,2}} \ar[u]_{i_{2,3}} \ar[d]^{i_{1,4}} \ar[ru]_{i_{3,4}} & \\ & \left( G \curvearrowright X \times Y \right)^{(F_1 \circ F_2).C} & } \end{displaymath}
This last diagram is obviously symmetric in $F_1, F_2, C$, which proves the claim.
\end{proof}

Obviously the compactly supported cohomology $\mathbb H^*_c \left( \left( G \curvearrowright X \times Y \right) \right)$ also admits a right action of the derived Hecke algebra $\mathcal H(G, Y)$ in a completely analogous way, as soon as the action $G \curvearrowright Y$ satisfies the conditions of fact \ref{generaldha}.
\begin{fact} \label{dhabimodule} Suppose the actions $G \curvearrowright X$ and $G \curvearrowright Y$ both satisfy the conditions of fact \ref{generaldha}. Then the compactly supported cohomology $\mathbb H^*_c \left( \left( G \curvearrowright X \times Y \right) \right)$ is a $\left( \mathcal H(G,X), \mathcal H(G,Y) \right)$-bimodule. 
\end{fact}
\begin{proof}
We already described how the two actions are defined, so it remains to show that they commute, i.e. that given $F_1 \in \mathcal H(G,X)$, $F_2 \in \mathcal H(G,Y)$ and $C \in \mathbb H^*_c \left( \left( G \curvearrowright X \times Y \right) \right)$ we have $(F_1.C).F_2 = F_1.(C.F_2)$. This boils down to the following two diagrams outputting the same cohomology class:

\begin{displaymath} \xymatrix{ \underline G_X^{F_1} & & \left( G \curvearrowright X \times Y \right)^C & \underline G_Y^{F_2} \\ & \left( G \curvearrowright X^2 \times Y \right) \ar[lu]_{i_{1,2}} \ar[ru]^{i_{2,3}} \ar[d]_{i_{1,3}} & & \\ & \left( G \curvearrowright X \times Y \right)^{F_1.C} & &  \left( G \curvearrowright X \times Y^2 \right) \ar[uu]_{i_{2,3}} \ar[ll]^{i_{1,2}} \ar[d]_{i_{1,3}} \\ & & & \left( G \curvearrowright X \times Y \right)^{(F_1.C).F_2} } \end{displaymath}
\begin{displaymath} \xymatrix{ \underline G_X^{F_1} & \left( G \curvearrowright X \times Y \right)^C & & \underline G_Y^{F_2} \\ & & \left( G \curvearrowright X \times Y^2 \right) \ar[lu]_{i_{1,2}} \ar[ru]^{i_{2,3}} \ar[d]_{i_{1,3}} & \\ \left( G \curvearrowright X^2 \times Y \right) \ar[uu]_{i_{1,2}} \ar[rr]^{i_{2,3}} \ar[d]_{i_{1,3}} & & \left( G \curvearrowright X \times Y \right)^{C.F_2} &  \\ \left( G \curvearrowright X \times Y \right)^{F_1.(C.F_2)} & & & } \end{displaymath}
We start with the following fact.
\begin{claim}
The square \begin{displaymath} \xymatrix{ \left( G \curvearrowright X^2 \times Y \right) \ar[d]_{i_{1,3}} & \left( G \curvearrowright X^2 \times Y^2 \right) \ar[l]_{i_{1,2,3}} \ar[d]^{i_{1,3,4}}  \\ \left( G \curvearrowright X \times Y \right) & \left( G \curvearrowright X \times Y^2 \right) \ar[l]^{i_{1,2}} } \end{displaymath} is a pullback square of groupoids. Moreover, it satisfies the additional assumptions of lemma \ref{cohomswitch}, that is to say: $i_{1,3}$ is a finite covering morphism and $i_{1,2}$ induces an injection of isotropy groups at all objects.
\end{claim}
\begin{proof}[Proof of claim] By assumption on the actions $G \curvearrowright X$ and $G \curvearrowright Y$, $i_{1,3}$ is a finite covering morphisms, and it is immediate that $i_{1,2}$ induces an injection of isotropy groups at all objects.

By proposition 4.4 in \cite{bhk}, the pullback square of a diagram of groups acting on sets is constructed by taking the pullback of the groups acting on the pullback of the sets. It is immediate that the pullback group is again $G$. Moreover, since $i_{1,3}$ is a covering, the pullback and the homotopy pullback coincide, hence a model for the object set of the pullback is \[ \left\{ \left( (x,y,z) , (a,b,c) \right) \in (X^2 \times Y) \times (X \times Y^2) \, | \, i_{1,3}(x,y,z) = i_{1,2}(a,b,c) \right\} = \] \[ = \left\{ \left( (x,y,z) , (a,b,c) \right) \in (X^2 \times Y) \times (X \times Y^2) \, | \, (x,z) = (a,b) \right\} = \left\{ \left( x =a, y, z=b, c \right) \right\} = X^2 \times Y^2 \] with the maps $i_{1,3,4}$ and $i_{1,2,3}$ as above. 
\end{proof}
By assumption on $G \curvearrowright X$ and $G \curvearrowright Y$, the cohomology class $F = i_{1,2}^*F_1 \cup i_{2,3}^*C$ is $i_{1,3}$-fiberwise compactly supported, thus lemma \ref{cohomswitch} applies and we can use the above pullback squares to modify the two original diagrams without impacting the final result in cohomology. The first diagram becomes \begin{displaymath} \xymatrix{ \underline G_X^{F_1} & & \left( G \curvearrowright X \times Y \right)^C & \underline G_Y^{F_2} \\ & \left( G \curvearrowright X^2 \times Y \right) \ar[lu]_{i_{1,2}} \ar[ru]^{i_{2,3}} &  \left( G \curvearrowright X^2 \times Y^2 \right) \ar[l]_{i_{1,2,3}} \ar[dr]^{i_{1,3,4}} & \\ & & &  \left( G \curvearrowright X \times Y^2 \right) \ar[uu]_{i_{2,3}} \ar[d]_{i_{1,3}} \\ & & & \left( G \curvearrowright X \times Y \right)^{(F_1.C).F_2} } \end{displaymath}
A biproduct of lemma \ref{cohomswitch} was that $\widetilde F = i_{1,2,3}^* F$ is $i_{1,3,4}$-fiberwise compactly supported, thus we can apply lemma \ref{corescupres} to the subdiagram \begin{displaymath} \xymatrix{ \left( G \curvearrowright X^2 \times Y^2 \right) \ar[dr]^{i_{1,3,4}} \ar[r]^{i_{3,4}} & \underline G_Y \\ &  \left( G \curvearrowright X \times Y^2 \right) \ar[u]_{i_{2,3}} } \end{displaymath} to obtain that for each $\widetilde F \in \mathbb H^* \left( \left( G \curvearrowright X^3 \times Y \right) \right)$ that is $i_{1,3,4}$-fiberwise compactly supported and each $F_2 \in \mathbb H^* \left( \underline G_Y \right)$, we have \[ (i_{1,3,4})_* \widetilde F \cup i_{2,3}^* F_2 = (i_{1,3,4})_* \left( \widetilde F \cup i_{3,4}^* F_2 \right) \] and hence we can replace the first diagram with \begin{displaymath} \xymatrix{ \underline G_X^{F_1} & & \left( G \curvearrowright X \times Y \right)^C & \underline G_Y^{F_2} \\ & \left( G \curvearrowright X^2 \times Y \right) \ar[lu]_{i_{1,2}} \ar[ru]^{i_{2,3}} &  \left( G \curvearrowright X^2 \times Y^2 \right) \ar[l]_{i_{1,2,3}} \ar[dr]^{i_{1,3,4}} \ar[ru]^{i_{3,4}} & \\ & & &  \left( G \curvearrowright X \times Y^2 \right) \ar[d]_{i_{1,3}} \\ & & & \left( G \curvearrowright X \times Y \right)^{(F_1 .C).F_2} } \end{displaymath} which is equivalent to \begin{displaymath} \xymatrix{ \underline G^{F_1} & \left( G \curvearrowright X \times Y \right)^C & \underline G_Y^{F_2} \\ &  \left( G \curvearrowright X^2 \times Y^2 \right) \ar[lu]_{i_{1,2}} \ar[u]_{i_{2,3}} \ar[d]^{i_{1,4}} \ar[ru]_{i_{3,4}} & \\ & \left( G \curvearrowright X \times Y \right)^{(F_1.C).F_2} & } \end{displaymath}
This last diagram is obviously symmetric in $F_1, F_2, C$, which proves the claim.
\end{proof}

\section{Two definitions of the derived Hecke algebra} \label{dhadiscussion}
In this section we discuss the relationship between our definition of the derived Hecke algbera as $G$-equivariant cohomology classes and the following `categorical' definition.

Let $\mathrm G$ be a split, connected reductive group over a local, non-archimedean field $F$. Denote $G = \mathrm G(F)$ and let $K$ be an open compact subgroup of $G$. Let $S$ be a coefficient ring, and denote by $\mathbf 1$ the trivial representation of $K$ with $S$-coefficients, so that $S[G/K] = \cInd{G}{K} \mathbf 1$ is the compactly induced representation to $G$.

The classical Hecke algebra can be defined as $H(G,K) = \Hom_{S[G]} \left( S[G/K], S[G/K] \right)$, the endomorphism of $S[G/K]$ as a $G$-module.
In \cite{akshay}, Venkatesh defines the derived Hecke algebra as a graded algebra over the classical one by replacing $\Hom$ with its derived functor $\Ext$: \begin{equation} \label{extdefinition} \mathcal H(G,K) = \bigoplus_{n \ge 0} \Ext^n_{S[G]} \left( S[G/K], S[G/K] \right). \end{equation}
When the residue characteristic $p$ is not invertible in the coefficient ring $S$ (for instance $S = \Z / p^a \Z$), we can also follow Schneider's construction from \cite{schneider}: let $S[G /K] \lra \mathbf I$ be an injective resolution\footnote{the category of $S[G]$-smooth modules has always enough injectives, see for instance \cite{vigneras}, I.5.9c.} and consider the differential graded algebra $\Hom_{S[G]} \left( \mathbf  I, \mathbf I \right)^{op}$.
Its cohomology algebra is \[ \mathbb H^* \left( \Hom_{S[G]} \left( \mathbf  I, \mathbf I \right)^{op} \right) \cong \Ext^*_{S[G]} \left( S[G/K], S[G/K] \right) \] so it coincides with the definition in formula \ref{extdefinition} above.
Since remark 7 in \cite{schneider} holds for any compact open $K \subset G$, it shows that for any injective resolution $S[G/K] \lra \mathbf I$ one has \[ \mathbb H^* \left( \Hom_{S[G]} \left( \mathbf  I, \mathbf I \right)^{op} \right) \cong H^* \left( K, S[G/K] \right). \]
We can use Mackey theory to decompose the $K$-module $S[G/K]$ as $\bigoplus_{g \in K \backslash G / K} \cInd{K}{K \cap gKg^{-1}} \mathbf 1$. Then section 2.6 in \cite{sym} shows that \[ H^n \left( K, \bigoplus_{g \in K \backslash G / K} \cInd{K}{K \cap gKg^{-1}} \mathbf 1 \right) = \Ext^n_{S[K]} \left( \mathbf 1, \bigoplus_{g \in K \backslash G / K} \cInd{K}{K \cap gKg^{-1}} \mathbf 1 \right) \cong \] \[ \cong \bigoplus_{g \in K \backslash G / K} \Ext^n_{S[K]} \left( \mathbf 1, \cInd{K}{K \cap gKg^{-1}} \mathbf 1 \right) = \bigoplus_{g \in K \backslash G / K} H^n \left( K, \cInd{K}{K \cap gKg^{-1}} \mathbf 1 \right) \cong \bigoplus_{g \in K \backslash G / K} H^n \left( K \cap gKg^{-1}, \mathbf 1 \right) \] where the last isomorphism holds by Shapiro's lemma.
In particular, this is exactly the double coset description given in 2.4 of \cite{akshay}. Thus, the two possible definitions of a derived Hecke algebra ($G$-equivariant cohomology classes versus cohomology of the complex $\underline{\Hom}_{D(G)} \left( S[G/K], S[G/K] \right)$) coincide as $S$-modules, and it remains to check that the two multiplication operations agrees.
For the time being we take as definition of the derived Hecke algebra the one as $G$-equivariant cohomology classes given in section \ref{secgroupoids}, which is much more explicit and therefore easier to work with, and we leave it as conjecture that the two coincide:
\begin{conj} The cohomology algebra $\mathbb H^* \left( \Hom_{S[G]} \left( \mathbf  I, \mathbf I \right)^{op} \right)$ and the algebra of $G$-equivariant cohomology classes $\mathcal H_S(G,K)$ as defined in section \ref{secgroupoids} are isomorphic as graded algebras via their common double cosets description.
More precisely, the isomorphism of $S$-modules given by \[ \mathcal H_S(G,K) \lra \bigoplus_{g \in K \backslash G / K} H^n \left( K \cap gKg^{-1}, \mathbf 1 \right) \qquad F \mapsto \left( F(K,gK) \right)_{g \in K \backslash G / K} \] is in fact an algebra isomorphism, where the right hand side is canonically identified with the cohomology algebra $\mathbb H^* \left( \Hom_{S[G]} (\mathbf I, \mathbf I)^{op} \right)$ as explained above.
\end{conj}
\begin{rem}
We have recently been informed by Schneider in a private communication that he and Ollivier proved this conjecture.
\end{rem}

\end{document}